\documentclass[reqno]{amsart}

\usepackage{a4wide}
\usepackage{amssymb}
\usepackage{graphicx}
\usepackage{hyperref}
\usepackage[mathscr]{euscript}
\usepackage{mathtools}
\usepackage{subfigure}
\usepackage{calc}
\usepackage{tikz}

\usepackage{subfigure}

\usetikzlibrary{cd,arrows,matrix}
\usepackage{newtxtext}
\usepackage{newtxmath}
\usepackage{anyfontsize}
\usepackage{bm}

\usepackage[english]{babel}

\usepackage[makeroom]{cancel}

\hypersetup{colorlinks=true,linkcolor=blue}

\usepackage{mathrsfs}

\usepackage{enumitem}

\numberwithin{equation}{section}

\theoremstyle{plain}
\newtheorem{theorem}{Theorem}[section]
\newtheorem{lemma}[theorem]{Lemma}
\newtheorem{proposition}[theorem]{Proposition}
\newtheorem{corollary}[theorem]{Corollary}

\newtheorem{definition/proposition}[theorem]{Definition/Proposition}
\newtheorem{theorem/definition}[theorem]{Theorem/Definition}
\newtheorem*{bigtheorem*}{Main Theorem}

\theoremstyle{definition}
\newtheorem{definition}[theorem]{Definition}
\newtheorem{example}[theorem]{Example}
\newtheorem{example/definition}[theorem]{Example/Definition}

\newtheorem{notation}[theorem]{Notation}

\newtheorem{conjecture}[theorem]{Conjecture}
\newtheorem*{conjecture*}{Conjecture}

\theoremstyle{remark}
\newtheorem{remark}[theorem]{Remark}

\newcommand{\RR}{\mathbb{R}}

\newcommand{\cA}{\mathcal{A}}

\newcommand{\cB}{\mathcal{B}}

\newcommand{\cF}{\mathcal{F}}

\newcommand{\cG}{\mathcal{G}}

\newcommand{\cP}{\mathcal{P}}

\newcommand{\cS}{\mathcal{S}}

\newcommand{\cW}{\mathcal{W}}

\DeclareMathOperator{\diag}{Diag}

\DeclareMathOperator{\im}{Image}

\newcommand{\id}{\mathsf{Id}}

\newcommand{\std}{\mathsf{std}}

\renewcommand{\emptyset}{\varnothing}

\renewcommand{\hat}{\widehat}

\renewcommand{\tilde}{\widetilde}

\newcommand{\differential}{\partial}

\newcommand{\boundary}{\partial}

\newcommand{\field}{\mathbb{K}}

\newcommand{\Left}{\mathsf{L}}
\newcommand{\Right}{\mathsf{R}}

\newcommand{\modulispace}{\mathcal{M}}
\newcommand{\modulistack}{\mathfrak{M}}

\newcommand{\homotopic}{\sim}

\newcommand{\isomorphic}{\cong}

\newcommand{\join}{\bm{\mathrm{*}} }

\newcommand{\aug}{\mathsf{Aug}}

\newcommand{\categoryfont}{\mathscr}

\newcommand{\Sh}{\categoryfont{S}h}

\newcommand{\sol}{\mathcal{S}ol}
\newcommand{\Del}{\textbf{Del}}

\newcommand{\mon}{\mathrm{mon}}

\newcommand{\mumon}{\mu\mathrm{mon}}

\newcommand{\Aff}{\categoryfont{A}ff}
\newcommand{\CAlg}{\mathrm{CAlg}}

\newcommand{\Spec}{\mathrm{Spec}}

\newcommand{\Gpd}{\categoryfont{G}pd}

\newcommand{\Dol}{\mathrm{Dol}}

\newcommand{\NAH}{\mathrm{NAH}}

\newcommand{\Mloc}{\mathcal{L}o}

\newcommand{\tot}{\mathrm{tot}}

\newcommand{\pr}{\mathrm{pr}}

\newcommand{\Br}{\mathrm{Br}}

\newcommand{\mB}{\mathrm{B}}

\newcommand{\mK}{\mathrm{K}}

\newcommand{\mH}{\mathrm{H}}

\newcommand{\ms}{\mathrm{s}}

\newcommand{\permutation}{\mathrm{s}}

\newcommand{\overallpermutation}{\mathbf{o}}

\newcommand{\FBD}{\mathrm{FBD}}

\newcommand{\bmdp}{\mathbf{bmdp}}

\newcommand{\BD}{\mathrm{BD}}

\newcommand{\bmd}{\mathbf{bmd}}

\newcommand{\weakequivalent}{\overset{\text{\tiny$\mathfrak{w}$}}{\sim}}

\newcommand{\braidequivalent}{\overset{\text{\tiny$\mathfrak{b}$}}{\sim}}

\newcommand{\FBr}{\mathrm{FBr}}

\makeatletter
\newcommand{\setword}[2]{%
  \phantomsection
  #1\def\@currentlabel{\unexpanded{#1}}\label{#2}%
}
\makeatother

\newcommand{\bfK}{\mathbf{K}}

\newcommand{\bfH}{\mathbf{H}}

\newcommand{\bfs}{\mathbf{s}}

\newcommand{\fr}{\mathfrak{r}}

\newcommand{\pt}{\mathrm{pt}}

\newcommand{\bfB}{\mathbf{B}}

\newcommand{\bfe}{\mathbf{e}}

\newcommand{\fc}{\mathfrak{c}}

\newcommand{\pa}{\mathrm{pa}}

\newcommand{\St}{\mathrm{St}}

\title{Dual boundary complexes of Betti moduli spaces over the two-sphere with one irregular singularity}

\author[T. Su]{Tao Su}
\email{sutao08@gmail.com}
\address{N03, Ningzhai, Yau Mathematical Sciences Center, Tsinghua University, Beijing 100084, China}

\begin{document}

\begin{abstract}
The weak geometric P=W conjecture of L. Katzarkov, A. Noll, P. Pandit, and C. Simpson states that, 
a smooth Betti moduli space of complex dimension $d$ over a punctured Riemann surface has the dual boundary complex 
homotopy equivalent to a sphere of dimension $d-1$.
Via a microlocal geometric perspective, we verify this conjecture for a class of rank $n$ wild character varieties 
over the two-sphere with one puncture, associated with any Stokes Legendrian link defined by an $n$-strand positive braid.
\end{abstract}

\maketitle

\tableofcontents

\section*{Introduction}\label{sec:intro}

Let $C$ be a genus $g$ closed Riemann surface with $k$ punctures $\sigma=\{p_1,\text{\tiny$\cdots$},p_k\}$, 
and $G=GL_n(\mathbb{C})$. 
Modulo extra input, the tame nonabelian Hodge correspondence over noncompact curves \cite{Sim90,Kon93} 
induces a diffeomorphism
\[
\NAH:\modulispace_{\Dol}\simeq \modulispace_B
\]
between two moduli spaces: 
the Dolbeault moduli space $\modulispace_{\Dol}$ of stable filtered regular (parabolic) $G$-Higgs bundles 
on $(\Sigma,\sigma)$ with parabolic degree $0$; 
and the Betti moduli space $\modulispace_B$ of stable filtered $G$-local systems 
on $\Sigma\setminus\sigma$ with parabolic degree $0$.
There are also various generalizations: wild NAH over curves \cite{BB04}; Replace $GL_n(\mathbb{C})$ by a general linear 
reductive algebraic group \cite{HKSZ22,HS22}; Replace $C$ with a higher dimensional variety \cite{Cor88,Sim92,Sim95,Moc11,Moc21}.
As a consequence, we obtain an isomorphism on cohomology (with rational coefficients)
\[
\NAH^*:H^*(\modulispace_{B})\xrightarrow[]{\sim}H^*(\modulispace_{\Dol}).
\]
The famous (cohomological) P=W conjecture of Cataldo, Hausel, and Migliorini then states:
\begin{conjecture}[{\cite{dCHM12}, Cohomological P=W}]
\[
\NAH^*(W_{2k}H^*(\modulispace_B)=W_{2k+1}H^*(\modulispace_B))=P_kH^*(\modulispace_{\Dol})
\]
where $P_{\bullet}$ is the Perverse-Leray filtration on $H^*(\modulispace_{\Dol})$ associated to 
the Hitchin map $h:\modulispace_{\Dol}\rightarrow\mathbb{A}$.
\end{conjecture}
This has now been given three very different proofs in the major case of twisted $GL_n(\mathbb{C})$-character varieties
\cite{MS22,HMMS22,MSY23}.
For earlier works or various extensions, see also \cite{dCHM12,dCMS22,dCM20,Dav23,Mau21,FM22,MMS22}.

Aiming at a geometric interpretation of the cohomological P=W conjecture, 
L. Katzarkov, A. Noll, P. Pandit, and C. Simpson \cite[Conj.1.1]{KNPS15}, \cite[Conj.11.1]{Sim16} formulated a 
geometric analogue, called the geometric P=W conjecture. 
Via $\NAH$, it relates the `asymptotic behavior of the Betti moduli space at infinity' 
to the `Hitchin map at infinity' on the Dolbeault side. 
More concretely, take a log compactification $\overline{\modulispace}_B$ of the (smooth) Betti moduli space $\modulispace_B$ 
with simple normal crossing boundary divisor $\boundary\modulispace_B:=\overline{\modulispace}_B\setminus\modulispace_B$. 
The combinatorics of the intersections of the irreducible components of $\boundary\modulispace_B$ 
is encoded by a simplicial complex $\mathbb{D}\boundary\modulispace_B$, called the dual boundary complex. 
The homotopy type of the dual boundary complex is an invariant of the algebraic variety $\modulispace_B$.

Now, a weak form of the geometric P=W conjecture states that

\begin{conjecture}[{\cite{KNPS15,Sim16}}, weak geometric P=W]\label{conj:homotopy type conj}
The dual boundary complex $\mathbb{D}\boundary\modulispace_B$ is homotopy equivalent to the sphere $S^{d-1}$, 
where $d=\mathrm{dim}_{\mathbb{C}}\modulispace_B$.
\end{conjecture}

More generally, the full geometric P=W conjecture identifies two fibrations up to homotopy: 
On the Dolbeault side, the `Hitchin fibration at infinity'
\[
\overline{h}:N_{\Dol}^*=\modulispace_{\Dol}\setminus h^{-1}(B_R(0))\xrightarrow[]{h} \mathbb{A}\setminus B_R(0)\rightarrow 
(\mathbb{A}\setminus B_R(0))/\text{\tiny scaling}~= S^{d-1}, R\gg 0;
\] 
On the Betti side, a fibration (up to homotopy) of the form $\alpha: N_B^*\rightarrow \mathbb{D}\partial\modulispace_B$, 
where $N_B^*$ is the punctured tubular neighborhood of the boundary divisor $\partial\modulispace_B$ in $\overline{\modulispace}_B$. 
In other words, 
\begin{conjecture}[{\cite{KNPS15,Sim16}}, Geometric P=W] There exists a homotopy commutative square
\[
\begin{tikzcd}[row sep=1pc,column sep=2pc]
N_{\Dol}^*\arrow{r}{\simeq}[swap]{\phi}\arrow{d}{\overline{h}} & N_B^*\arrow{d}{\alpha}\\
S^{d-1}\arrow{r}{\simeq} & \mathbb{D}\boundary\modulispace_B
\end{tikzcd}
\]
where the bottom arrow corresponds to the weak geometric P=W conjecture.
\end{conjecture}
By \cite[Thm.6.2.6]{MMS22}, the geometric P=W conjecture implies the cohomological P=W conjecture in top weights, 
hence the terminology. Otherwise, there's no implication in either direction. To interpret the latter in all degrees,
it's expected that a refined version of the geometric P=W conjecture is required.

To the best of the author's knowledge, the full geometric P=W conjecture has been established only for a few cases 
when $\NAH$ admits a somewhat concrete description: the Painlev\'{e} cases ((g,n)=(0,2), wild $\modulispace_B$) \cite{NS22,Sza19,Sza21};
$(g,k)=(1,0)$ (singular $\modulispace_B$) or $(k,n)=(0,1)$ \cite[Thm.B]{MMS22}.
From now on, the main concern will be the weak geometric P=W conjecture. 
Apart from the two lists above, previous work of other authors has also established the case
$G=SL_2(\mathbb{C})$ \cite{Kom15,Sim16,FF23}. However, all these proofs use some special features of the moduli spaces under
investigation, which rarely apply to the general case. 
In this article, the main goal is to establish the weak geometric P=W conjecture for 
a class of rank $n$ wild character varieties over $(\mathbb{P}^1,\infty)$. Moreover, following the same strategy of the proof
(cell decomposition), the author has been able to prove the conjecture for all very generic character varieties
in a subsequent paper \cite{Su23}.

Let's add a complementary remark. A \textbf{folklore conjecture} predicts that, all smooth Betti moduli spaces $\modulispace_B$ 
are log Calabi-Yau: $\exists$ a log compactification $\overline{\modulispace}_B$ such that 
$(\overline{X}=\overline{\modulispace}_B, D=\partial\modulispace_B)$ is a log Calabi-Yau pair, 
i.e. $K_{\overline{X}}+D$ is trivial. 
Then, the weak geometric P=W conjecture may be regarded as a special case of the more general conjecture:

\begin{conjecture}[M.Kontsevich]
The dual boundary complex of a log Calabi-Yau variety is a sphere.
\end{conjecture}

The only known partial result is due to J. Koll\'{a}r and C. Xu \cite{KX16}: 
If $X$ is log Calabi-Yau of dimension $\leq 5$, then $\mathbb{D}\partial X$ is a finite quotient of a sphere.

\subsection*{Results}
\addtocontents{toc}{\protect\setcounter{tocdepth}{1}}

Throughout the context, we fix a base field $\field$ (for example, $\field=\mathbb{C}$).
Let $(C,\sigma=\{p_i\}_{i=1}^k,\{\Lambda_i\}_{i=1}^k)$ be a genus $g$ closed Riemann surface with $k$ punctures 
$\sigma=\{p_1,\text{\tiny$\cdots$},p_k\}$, such that $\Lambda_i$ is a Legendrian link in the co-sphere bundle $T^{\infty}(C\setminus\sigma)$ whose front projection encircles $p_i$.
Each $\Lambda_i$ is identified with the cylindrical closure $\beta_i^{\circ}$
of some positive braid $\beta_i\in\Br_{n_i}^+$.
Let $\vec{r}=(\vec{r}_i)_{i=1}^k:\pi_0(\Lambda)=\text{\tiny$\sqcup_{i=1}^k$}\pi_0(\Lambda_i)\rightarrow \mathbb{N}$ be a suitable map.
As in \cite{STZ17,STWZ19}, we may consider the moduli space 
$\modulispace_{\vec{r}}(C,\sigma=\{p_i\};\{\Lambda_i\})$ of (microlocal rank $\vec{r}$) constructible sheaves on $C$ whose
micro-support at infinity are contained in $\Lambda=\text{\tiny$\sqcup_{i=1}^k$}\Lambda_i$ and whose stalks at $\sigma$ are acyclic.
Taking the microlocal stalks gives rise to the microlocal monodromy:
\[
\mumon: \modulistack_{\vec{r}}(C,\{p_i\},\{\Lambda_i\})\rightarrow \text{\tiny$\prod_{i=1}^k$}\Mloc_{\vec{r}_i}(\Lambda_i),
\]
where $\Mloc_{\vec{r}_i}$ is the moduli stack of rank $\vec{r}_i$ local systems in degree $0$ on $\Lambda_i$. 
The Betti moduli stack $\modulistack_L(C,\{p_i\},\{\Lambda_i\})$ is the substack of $\modulistack_{\vec{r}}(C,\{p_i\},\{\Lambda_i\})$ consisting of objects whose microlocal monodromy is quasi-isomorphic 
to a fixed local system $L\in \text{\tiny$\prod_{i=1}^k$}\Mloc_{\vec{r}_i}(\Lambda_i)$. 
The associated good moduli space \cite{Alp13} $\modulispace_L(C,\{p_i\},\{\Lambda_i\})$ is called a \emph{Betti moduli space} in this article.
The terminology is justified by the microlocal reformulation of irregular Riemann-Hilbert correspondence over curves 
(see \cite[Sec.3.3]{STWZ19} or Appendix \ref{sec:irregular_RH}), according to which 
a wild character variety over a complex curve $C$ can be identified with 
a Betti moduli space $\modulispace_L(C,\{p_i\},\{\Lambda_i\})$ where each $\Lambda_i$ is a `Stokes Legendrian link' encircling $p_i$.

Our main theme is the weak geometric P=W conjecture for $\modulispace_B=\modulispace_L(\mathbb{C}P^1,\{\infty\},\{(\Delta\beta\Delta)^{\circ}\})$, 
where $\beta_1=\Delta\beta\Delta\in\Br_n^+$ with $\Delta$ a half-twist, and we take $\vec{r}=1$.
Then, our main theorem is the following:

\begin{bigtheorem*}[Theorem \ref{thm:main}]\label{thm:main_raw}
If the twisted conjugacy class defining $[L]$ is (twisted) semisimple and generic (Definition \ref{def:generic_condition}), 
then the Betti moduli space $\modulispace_B$ (if nonempty) is a smooth affine and connected variety, 
and the weak geometric P=W conjecture holds for $\modulispace_B$, i.e. we have a homotopy equivalence
$\mathbb{D}\partial\modulispace_B \sim S^{\dim\modulispace_B -1}$.
\end{bigtheorem*}
In the initial version of this article, the main theorem was specifically formulated for connected braid closures $(\Delta\beta\Delta)^{\circ}$. Subsequently, this restriction has been lifted, resulting in a more general formulation. Notably, this expansion encompasses cases such as the (unramified/untwisted) wild character varieties over $(\mathbb{C}P^1,\infty)$ discussed in \cite{Boa07}, which were excluded in the initial version.
The author expresses gratitude to the anonymous referee for suggesting the consideration of such a more general situation.

The strategy for proving the main result employs a \emph{remove/reduction lemma} due to C. Simpson \cite[Lem.2.3]{Sim16}.
This lemma enables the elimination of an irreducible closed variety in the form of $\mathbb{A}^1\times Y$ from a smooth variety $X$ without altering the homotopy type of the dual boundary complex of $X$.
Subsequently, the main theorem follows by induction from the cell decomposition of $\modulispace_B$:

\begin{proposition}[Proposition \ref{prop:cell_decomposition}. Cell decomposition]
$\modulispace_B$ admits a cell decomposition into locally closed subvarieties
\[
\modulispace_B(\field)=\sqcup_{p\in\cW^*(\beta_1)} \modulispace_B^p(\field)
\]
such that 
\begin{enumerate}
\item
For each $p\in\cW^*(\beta_1)$, we have a natural isomorphism
\[
\modulispace_B^p(\field)\isomorphic (\field^*)^{a(p)}\times\field^{b(p)},
\]
such that $a(p)+2b(p)=d=\dim\modulispace_B$.

\item
$\exists!p_m\in\cW^*(\beta_1)$ such that $b(p_m)=0$. In addition, $\modulispace_B^{p_m}(\field)$ is open and dense in $\modulispace_B(\field)$, and
\[
\modulispace_B^{p_m}(\field)\isomorphic (\field^*)^{a(p_m)}=(\field^*)^d.
\]
\end{enumerate}
\end{proposition}

The cell decomposition relies a more concrete description of $\modulispace_B$. 
More generally, we have concrete descriptions of the Betti moduli spaces $\modulispace_L(C,\{p_i\},\{\Lambda_i\})$ similar to character varieties \cite{HLRV11,Boa07,Boa14,BY15}:

\begin{theorem}[Theorem \ref{thm:Betti_moduli_stack_via_braids}, Corollary \ref{cor:restricted_Betti_moduli_stack_via_braids}, \ref{cor:restricted_Betti_moduli_space_via_braids}]\label{thm:description_of_Betti_moduli_spaces}
Under mild assumptions, We have concrete descriptions of $\mumon:\modulistack_{\vec{r}}(C,\{p_i\},\{(\beta_i)^{\circ}\})\rightarrow \text{\tiny$\prod_{i=1}^k$}\Mloc_{\vec{r}_i}(\Lambda_i)$, 
$\modulistack_L(C,\{p_i\},\{(\beta_i)^{\circ}\})$ and $\modulispace_L(C,\{p_i\},\{(\beta_i)^{\circ}\})$. In particular,
\[
\modulistack_1(\mathbb{P}^1,\{\infty\},\{(\Delta\beta\Delta)^{\circ}\}) \simeq [X(\Delta\beta;w_0)/T],
\]
where $X(\gamma,w)$ is the \emph{modified braid variety} (see (\ref{eqn:modified_braid_variety})) associated to $\gamma\in\Br_n^+, w\in W=S_n$, $w_0\in W$ is the longest element, 
and $T\subset G=GL(n)$ is the diagonal maximal torus.
\end{theorem}

Now, the cell decomposition of $\modulispace_B$ follows from that of braid varieties. The latter goes back to \cite{Mel19}. 
We remark that in the first version of this article, instead of braid varieties, 
the proof of our main result 
uses augmentation varieties $\aug(\beta^{>},*_1,\text{\tiny$\cdots$},*_n)$ \cite{HR15}, and is somewhat more complicated.
However, as `augmentations are sheaves' \cite{NRSSZ20}, we have natural identifications
\[
[\aug(\beta^{>},*_1,\text{\tiny$\cdots$},*_n)/T] \simeq \modulistack_1(\mathbb{P}^1,\{\infty\},\{(\Delta\beta\Delta)^{\circ}\})\simeq [X(\Delta\beta;w_0)/T].
\]
Under the identification, the two proofs are essentially equivalent, and the cell decomposition of the braid variety $X(\Delta\beta;w_0)$ 
becomes the Henry-Rutherford decomposition \cite{HR15} of the augmentation variety $\aug(\beta^{>},*_1,\text{\tiny$\cdots$},*_n)$.
See Remark \ref{rem:connection_with_augmentations} for more explanations.

Finally, combining Theorem \ref{thm:description_of_Betti_moduli_spaces} and \cite{Su23}, it's possible to generalize the strategy in this article to treat the weak geometric P=W conjecture 
for all `very generic' wild character varieties over arbitrary punctured Riemann surfaces. This will be done in a future paper.

\subsection*{Organization}
\addtocontents{toc}{\protect\setcounter{tocdepth}{1}}

In Section \ref{sec:constructible_sheaves}, we review the background on constructible sheaves and Legendrian knots. 
The upshot is to introduce various moduli stacks/spaces of constructible sheaves: 
the moduli stacks $\modulistack_{\vec{r}}=\modulistack_{\vec{r}}(C,\{p_i\},\{\Lambda_i\})$, the microlocal monodromy 
$\mumon: \modulistack_{\vec{r}}\rightarrow \text{\tiny$\prod_{i=1}^k$}\Mloc_{\vec{r}_i}(\Lambda_i)$, the restricted versions $\modulistack_{L}(C,\{p_i\},\{\Lambda_i\})$ (Betti moduli stacks), 
as well as the associated good moduli spaces $\modulispace_{\vec{r}}(C,\{p_i\},\{\Lambda_i\})$ and $\modulispace_{L}(C,\{p_i\},\{\Lambda_i\})$ (Betti moduli spaces).

Appendix \ref{sec:irregular_RH} complements Section \ref{sec:constructible_sheaves} by a microlocal reformulation of the irregular Riemann-Hilbert correspondence over curves. In particular, 
we explain the notion of Stokes Legendrian links and justify that $\modulispace_{L}(C,\{p_i\},\{\Lambda_i\})$ are indeed a generalizations of (unramified or ramified) wild character varieties.

In Section \ref{sec:Betti_moduli_spaces_and_braids}, the main goal is to give a concrete description of $\mumon: \modulistack_{\vec{r}}\rightarrow \text{\tiny$\prod_{i=1}^k$}\Mloc_{\vec{r}_i}(\Lambda_i)$ (Theorem {\ref{thm:description_of_Betti_moduli_spaces}}). In Section \ref{subsec:diagram_calculus_of_block_matrices}, we introduce some diagram calculus of block matrices. In Section \ref{subsec:Betti_moduli_stacks_associated_to_positive_braids} complemented by Appendix \ref{sec:prove_local_diagram_for_Betti_moduli_stacks}, 
we deal with the local problem, i.e. Betti moduli stacks associated to positive braids. 
In Section \ref{subsec:Betti_moduli_stacks_over_punctured_Riemann_surfaces}-\ref{subsec:Betti_moduli_spaces_over_punctured_Riemann_surfaces}, we deal with the global problem via the gluing property, i.e. Betti moduli stacks/spaces over punctured Riemann surfaces.

In Section \ref{sec:Betti_moduli_spaces_over_punctured_P1_and_braid_varieties}, we give the cell decomposition for the Betti moduli space 
$\modulispace_B=\modulispace_L(\mathbb{C}P^1,\{\infty\},\{(\Delta\beta\Delta)^{\circ}\})$ with $L\in\Mloc_{\vec{r}}((\Delta\beta\Delta)^{\circ})$ and $\vec{r}=1$. The main tool is the cell decomposition of braid varieties.

In Section \ref{sec:main_theorem}, we firstly review the background on dual boundary complexes, and then prove the main theorem (Theorem \ref{thm:main}).

\subsection*{Acknowledgements}
\addtocontents{toc}{\protect\setcounter{tocdepth}{1}}
This work was initiated during the author's visit to IHES in 2020. 
The author expresses gratitude to IHES for its support and M. Kontsevich for a valuable conversation. Special thanks to Penghui Li for helpful discussions at YMSC.
Since the first version, the author received valuable comments and suggestions for revision from several individuals, including the anonymous referee, Vivek Shende, Philip Boalch, etc.
During the preparation of the initial version, the author was supported by the 2021 International Postdoctoral Exchange Fellowship (Talent-Introduction Program) at YMSC.
The revised version was completed during a visit invited by Song-Yan Xie at AMSS, CAS.

\section{Constructible sheaves and Legendrian knots}\label{sec:constructible_sheaves}

In this section, we review the microlocal theory of sheaves with a general reference to \cite{KS90}.
The specific focus of this article is on constructible sheaves defined on a surface $C$, 
with micro-support at infinity contained within a specified collection of Legendrian links near the punctures $\{p_i\}$.

\subsection{Micro-support}
Let $M$ be a real analytic manifold and $\field$ be a fixed base field (or more generally, commutative ring). 
By a sheaf on $M$, we mean a \emph{complex of sheaves of $\field$-modules} on $M$. Denote by $\Sh(M;\field)$ 
the triangulated dg category of sheave on $M$. Given a sheaf $\mathcal{F}$ on $M$, the notion of \emph{micro-support} 
(or \emph{singular support}) $SS(\mathcal{F})$, roughly speaking, measures the co-directions in $T^*M$ along which 
the local sections of $\mathcal{F}$ fail to propagate. More precisely, we have

\begin{definition}[Micro-support. {\cite[Def.5.1.2]{KS90}}]\label{def:micro-support}
Let $p=(x_0,\xi_0)\in T^*M$ be a covector. We say $p\notin SS(\mathcal{F})$ if there exists 
a neighborhood $U$ of $p$, such that for any $x_1\in M$ and $C^1$-function $\varphi$ on a neighborhood of $x_1$ 
satisfying $(x_1,d\varphi(x_1))\in U$ and $\varphi(x_1)=0$, we have
\[
R\Gamma_{\{\varphi(x)\geq 0\}}(\mathcal{F})_{x_1}\simeq 0.
\]
Equivalently, via the exact triangle 
$R\Gamma_{\{\varphi(x)\geq 0\}}(\mathcal{F})\rightarrow\mathcal{F}\rightarrow R\Gamma_{\{\varphi(x)<0\}}(\mathcal{F})\xrightarrow[]{+1}$,
we have a quasi-isomorphism
\[
\mathcal{F}_{x_1}\xrightarrow[]{\sim} R\Gamma_{\{\varphi(x)<0\}}(\mathcal{F})_{x_1}.
\]
\end{definition}

\noindent{}The intuitive meaning of the last quasi-isomorphism is exactly that the local sections of $\mathcal{F}$ 
propagate along the co-direction $p$. 
The key properties of the micro-support $SS(\mathcal{F})$ are as follows:
\begin{enumerate}[wide,labelwidth=!,labelindent=0pt]
\item
$SS(\cF)\cap 0_M=Supp(\cF)$ is the support of $\cF$, where $0_M$ the zero section of $T^*M$.

\item
For any sheaf $\cF$, $SS(\cF)$ is a conic (i.e. invariant under the action of $\RR_+$ which 
scales the cotangent fibers) and closed co-isotropic subset of $T^*M$.

\item
If $\cF$ is a constructible sheaf with respect to a Whitney stratification $\cS$, then $SS(\cF)$ is 
a conic Lagragian (i.e. Lagragian wherever it's smooth) subset of $T^*M$ with 
$SS(\cF)\subset L_{\cS}\coloneqq \cup_{i\in\cS}T^*_{S_i}M$.

\item(Triangular inequality)
If $\cF_1\rightarrow \cF_2\rightarrow \cF_3\xrightarrow[]{+1}$ is an exact triangle in $\Sh(M;\field)$, 
then $SS(\cF_i)\subset SS(\cF_j)\cup SS(\cF_k)$ for all distinct $i,j,k\in\{1,2,3\}$.

\item(Microlocal Morse lemma)
If $f:M\rightarrow \RR$ is a smooth function such that $(x,df(x))\notin SS(\cF)$ for all $x\in f^{-1}([a,b])$, 
and $f$ is proper on the support of $\cF$. Then the restriction map is a quasi-isomorphism:
\[
R\Gamma(f^{-1}(-\infty,b);\cF)\xrightarrow[]{\sim} R\Gamma(f^{-1}(-\infty,a);\cF)
\]  
\end{enumerate}

Let $T^{\infty}M:=(T^*M\setminus 0_M)/\mathbb{R}_+$ (the co-directions at infinity) be the co-sphere bundle of $M$, 
and let $\Lambda\subset T^{\infty}M$ be a given Legendrian subset. 
\begin{definition}
We \emph{define} $\Sh_{\Lambda}(M)$ to be the triangulated dg category of constructible sheaves of 
perfect $\field$-modules on $M$ whose micro-support at infinity is contained in $\Lambda$, localized at quasi-isomorphisms. 
If $\sigma$ is a finite set of points in $M$ which is disjoint from the front projection of $\Lambda$, 
\emph{define} $\Sh_{\Lambda}(M,\sigma)$ to be the full (still triangulated) sub-category of $\Sh_{\Lambda}(M)$ 
consisting of sheaves whose stalks at $\sigma$ vanish.
 \end{definition}
We will always assume $\Lambda\subset T^{\infty}M$ is a smooth Legendrian submanifold, 
and the front projection $\Lambda\rightarrow M$ is in generic position. This is always satisfied in our context.

The categories we've defined are Legendrian isotopy invariants:
\begin{proposition}\cite[Prop.2.7]{STWZ19}\label{prop:invariance_of_sheaf_cats}
Any Legendrian isotopy $\Lambda\sim\Lambda'$ in the complement of the conormal of $\sigma$ induces a dg equivalence of categories 
\[
\Sh_{\Lambda}(M,\sigma)\simeq\Sh_{\Lambda'}(M,\sigma).
\]
\end{proposition}

\subsection{Microlocal monodromy}
There's a refinement encoding the microlocal information on $\Lambda$.
\begin{definition}[Microlocal stalk]
Given $\mathcal{F}\in\Sh_{\Lambda}(M)$ and any $(x,\xi)\in\Lambda$, the \emph{microlocal stalk} 
$\mathcal{F}|_{(x,\xi)}$ of $\mathcal{F}$ at $(x,\xi)$ is defined as follows: 
Take a function $\varphi$ on a neighborhood of $x$ such that $\varphi(x)=0$ and $d\varphi(x)=\xi$, then 
$\mathcal{F}|_{(x,\xi)}:=\mathrm{Cone}(\mathcal{F}_x\rightarrow R\Gamma_{\{\varphi<0\}}(\mathcal{F})_x)
\isomorphic R\Gamma_{\{\varphi\geq 0\}}(\mathcal{F})_x[1]$ is independent of the choice of $\varphi$.
\end{definition}

In this article, the pair $(M,\Lambda):=(C,\Lambda)$ we will consider fall into two classes:
\begin{enumerate}[wide,labelwidth=!, labelindent=0pt]
\item
$C$ is a closed Riemann surface with punctures $\{p_1,\ldots,p_k\}$.
$\Lambda=\cup\Lambda_i\subset T^{\infty}C$ is a union of Legendrian links $\Lambda_i$, 
such that the front projection\footnote{Pick any Riemannian metric on $C$ so that $T^{\infty}C$ is 
identified with the sphere (or unit tangent) bundle. Then, $\Lambda_i$ is identified with 
its front projection together with a unit normal vector field. We always \textbf{assume} that 
the normal vector field points toward $p_i$.} 
of $\Lambda_i$ has no cusp and traverses around $p_i$. 
See Figure \ref{fig:microlocal_Riemann_surface} for an illustration.
We always take the \emph{zero Maslov potential} $\mu$ for $\Lambda$, i.e. 
the zero grading on $\pi_0(\Lambda)$.

\item
$C=\mathbb{C}P^1\isomorphic\mathbb{R}_{x,z}^2\sqcup\{\infty\}$, 
$\Lambda\subset J^1\mathbb{R}_x\isomorphic (\mathbb{R}_{x,y,z}^3,\alpha=dz-ydx)\subset T^{\infty}M$ is a Legendrian link, 
whose front projection $\pi_{xz}(\Lambda)$ in $\mathbb{R}_{xz}^2$ is equipped with a $\mathbb{Z}$-valued Maslov potential, 
that is, a map 
\[
\mu:\text{strands}(\pi_{xz}(\Lambda))\rightarrow \mathbb{Z}\]
such that, $\text{strands}(\pi_{xz}(\Lambda))$ is the set of connected components of the immersed curve 
$\pi_{xz}(\Lambda)\setminus\{\text{cusps}\}$, and when two strands meet at a cusp, then 
$\mu(\text{upper strand})=\mu(\text{lower strand})+1$.
\end{enumerate}

\begin{figure}[!htbp]
  \begin{center}
  \begin{tikzpicture}[baseline=-.5ex,scale=0.2]
  \begin{scope}
   
  \draw[thick] (-20,5) to[out=-30,in=180] (-16,3.5) to[out=0,in=-150] (-12,5);
  \draw[thick] (-19.8,4.9) to[out=30,in=180] (-16,6.5) to[out=0,in=150] (-12.2,4.9);

  \draw[thick] (17,5) to[out=-30,in=180] (21,3.5) to[out=0,in=-150] (25,5);
  \draw[thick] (17.2,4.9) to[out=30,in=180] (21,6.5) to[out=0,in=150] (24.8,4.9);

  \draw[thick] (-18,0) circle (1);  
  \draw[thin] (-18,0) circle (0.1);
  \draw (-18,-1) node[below] {{\tiny$\Lambda_1$}};

  \draw[thick] (0,2) to[out=0,in=180] (2,3);
  \draw[white, fill=white] (1,2.5) circle (0.2);
  \draw[thick] (0,3) to[out=0,in=180] (2,2) to[out=0,in=180] (4,3);
  \draw[white, fill=white] (3,2.5) circle (0.2);
  \draw[thick] (2,3) to[out=0,in=180] (4,2) to[out=0,in=180] (6,3);
  \draw[white, fill=white] (5,2.5) circle (0.2);
  \draw[thick] (4,3) to[out=0,in=180] (6,2);
  \draw[white, fill=white] (7,2.5) circle (0.2);
  
  \draw[thick] (0,3) to[out=180,in=-90] (-0.5,4) to[out=90,in=180] (3,6) to[out=0,in=90] (6.5,4) to[out=-90,in=0] (6,3);
  \draw[thick] (0,2) to[out=180,in=-90] (-1,4) to[out=90,in=180] (3,6.5) to[out=0,in=90] (7,4) to[out=-90,in=0] (6,2);
  \draw[thin] (3,4.3) circle (0.1); 
  
  \draw (3,2) node[below] {{\tiny$\Lambda_2$}};
  
  \draw[thick] (21,-3) to[out=0,in=180] (23,-2);
  \draw[white, fill=white] (22,-2.5) circle (0.2);
  \draw[thick] (21,-2) to[out=0,in=180] (23,-3) to[out=0,in=180] (25,-2);
  \draw[white, fill=white] (24,-2.5) circle (0.2);
  \draw[thick] (23,-2) to[out=0,in=180] (25,-3);
  
  \draw[thick] (21,-2) to[out=180,in=-90] (20.5,-1) to[out=90,in=180] (23,0.5) to[out=0,in=90] (25.5,-1) to[out=-90,in=0] (25,-2);
  \draw[thick] (21,-3) to[out=180,in=-90] (20,-1) to[out=90,in=180] (23,1) to[out=0,in=90] (26,-1) to[out=-90,in=0] (25,-3);
  \draw[thin] (23,-0.8) circle (0.1); 
  
  \draw (23,-3) node[below] {{\tiny$\Lambda_3$}};
  
  \draw[thick] (-28,5) to[out=90,in=180] (-16,15) to[out=0,in=180] (3,10) to[out=0,in=180] (21,15) to[out=0,in=90] (34,5);
  \draw[thick] (-28,5) to[out=-90,in=180] (-16,-7) to[out=0,in=180] (3,-1) to[out=0,in=180] (21,-7) to[out=0,in=-90] (34,5);
  
  \end{scope}
  \end{tikzpicture}
  \end{center}
  \caption{An example of $(C,\{p_i\}_{i=1}^k,\{\Lambda_i\}_{i=1}^k)$: a genus $g$ closed Riemann surface $C$ with 
  Legendrian links $\{\Lambda_i\}$ encircling the punctures $\{p_i\}$. In the figure, $g=2$, $k=3$. The Legendrian `unknot' 
  $\Lambda_1$ means the regular singularity at $p_1$.}
  \label{fig:microlocal_Riemann_surface}
  \end{figure}

\begin{definition/proposition}[Microlocal monodromy. {\cite[Def.5.4]{STZ17}}]
Let $(C,\Lambda)$ be as above.
\begin{enumerate}[wide,labelwidth=!, labelindent=0pt]
\item
There is a natural functor
\[
\mumon: \Sh_{\Lambda}(C)\rightarrow\mathcal{L}oc(\Lambda)
\]
called \emph{microlocal monodromy}, such that for any generic point $(x,\xi)$ of $\Lambda$, we define
\[
\mumon(\mathcal{F})(x,\xi):=\mathcal{F}|_{(x,\xi)}[-\mu(x)].
\]
Here, $\mathcal{L}oc(\Lambda)$ is the triangulated dg category of (complexes of) local systems of perfect $\field$-modules 
on $\Lambda$.

\noindent{}$\mathcal{C}_{\mathrm{B}}(C,\{p_i\},\{\Lambda_i\}):=\Sh_{\cup_i\Lambda_i}(C,\{p_i\})$, 
the microlocal monodromy then restricts to 
\[
\mumon:\mathcal{C}_{\mathrm{B}}(C,\{p_i\},\{\Lambda_i\})\rightarrow\mathcal{L}oc(\Lambda).
\]
\item
For any natural number $r$\footnote{Throughout the context, the same discussions apply if we replace $r$ by a map $\vec{r}:\pi_0(\Lambda)\rightarrow \mathbb{Z}_{\geq 0}$.}, we say a sheaf $\mathcal{F}\in\Sh_{\Lambda}(C)$ has \emph{microlocal rank $r$} 
if the microlocal monodromy $\mumon(\mathcal{F})$ is a local system of free $\field$-modules of rank $r$ 
concentrated in degree $0$.
Define $\mathcal{C}_r(C,\{p_i\},\{\Lambda_i\})$ to be the full dg sub-category 
of $\mathcal{C}_{\mathrm{B}}(C,\{p_i\},\{\Lambda_i\})$ of microlocal rank $r$ objects.
\end{enumerate}
\end{definition/proposition}

The following is a consequence of Proposition \ref{prop:invariance_of_sheaf_cats}:
\begin{corollary}\label{cor:invariance_of_sheaf_cats}
Any Legendrian isotopy $\Lambda=\cup_i\Lambda_i\sim\Lambda'=\cup_i\Lambda_i'$ in the complement of the conormal of 
$\sigma=\{p_i\}$ induces a dg equivalence of categories preserving the microlocal monodromy:
\[
\begin{tikzcd}
\mathcal{C}_{\mathrm{B}}(C,\{p_i\},\{\Lambda_i\})\arrow{r}{\simeq}\arrow{d}{\mumon} & \mathcal{C}_{\mathrm{B}}(C,\{p_i\},\{\Lambda_i'\})\arrow{d}{\mumon}\\
\mathcal{L}oc(\Lambda)\ar[equal]{r} & \mathcal{L}oc(\Lambda')
\end{tikzcd}
\]
In particular,  we obtain a dg equivalence $\mathcal{C}_r(C,\{p_i\},\{\Lambda_i\})\simeq\mathcal{C}_r(C,\{p_i\},\{\Lambda_i'\})$.
\end{corollary}

\subsection{The Betti moduli space}

Let $\mathcal{C}_r(C,\{p_i\},\{\Lambda_i\})$ be the dg category of micro-local rank $r$ constructible sheaves 
on $(C,\{p_i\},\{\Lambda_i\})$. We can form the moduli stack $\modulistack_r(C,\{p_i\},\{\Lambda_i\})$.

\begin{definition}
Define a functor 
\[
\modulistack_r(C,\{p_i\},\{\Lambda_i\};\field): \field-\CAlg=\Aff_{\field}^{op} \rightarrow \Gpd
\]
by sending $A\in\field-\CAlg$ to $\mathcal{C}_r(C,\{p_i\},\{\Lambda_i\};A)^{gpd}$, 
where $\mathcal{C}_r(C,\{p_i\},\{\Lambda_i\};A)$ is the dg category of microlocal rank $r$ constructible sheaves 
of $A$-modules on $(C,\{p_i\},\{\Lambda_i\})$, and $\mathcal{C}_r(C,\{p_i\},\{\Lambda_i\};A)^{gpd}$ is the groupoid 
associated to $\mathcal{C}_r(C,\{p_i\},\{\Lambda_i\};A)$: 
the objects are objects in $\mathcal{C}_r(C,\{p_i\},\{\Lambda_i\};A)$ without negative self-extensions \cite[Rmk.2.18]{STWZ19}; 
the morphisms are quasi-isomorphisms in $\mathcal{C}_r(C,\{p_i\},\{\Lambda_i\};A)$ up to homotopy.
\end{definition}
We remark that, according to \cite[Def./Prop. 2.16]{STWZ19}, the moduli stack $\modulistack_r(C,\{p_i\},\{\Lambda_i\};\field)$ 
is the truncation of the derived moduli stack $\mathbb{R}\modulistack_r(C,\{p_i\},\{\Lambda_i\};\field)$ \cite{TV07} of objects 
in the dg category $\mathcal{C}_r(C,\{p_i\},\{\Lambda_i\})$.
In addition, $\modulistack_r(C,\{p_i\},\{\Lambda_i\};\field)$ is an \emph{algebraic stack} over $\field$ in the classical sense. 

To simplify the study of the (fine moduli) stack $\modulistack$ representing a given moduli problem, 
one would like to associate an algebraic space which still captures much of the geometry. Classically, 
one uses the weaker notion of a \emph{coarse moduli space}, whose points are in bijection with the isomorphism classes 
of objects over the given algebraically closed field. For algebraic stacks with finite inertia 
(e.g. separated Deligne-Mumford stacks), this works well, the coarse moduli spaces exist by the Keel-Mori theorem. 
However, for algebraic stacks with infinite stabilizers as in our case, the coarse moduli space do not exist in general. 

In fact, what we need is a generalization of coarse moduli spaces, i.e. the notion of \emph{good moduli spaces} associated 
to algebraic stacks (possibly with infinite stabilizers) \cite{Alp13}, which we recall below.

\subsubsection{Good moduli spaces for algebraic stacks}

\begin{definition}[{\cite[Def.4.1]{Alp13}}]
A quasi-compact morphism $f:\mathfrak{X}\rightarrow Y$ from an algebraic stack to an algebraic space 
(over $\field$) is a \emph{good moduli space}, if
\begin{enumerate}
\item
The push-forward functor on quasi-coherent sheaves is exact.
\item
The induced morphisms on sheaves $\mathcal{O}_Y\rightarrow f_*\mathcal{O}_{\mathfrak{X}}$ is an isomorphism.
\end{enumerate}
For simplicity, we say in this case that $Y$ is a \emph{good moduli space} for $\mathfrak{X}$.
\end{definition}

A good moduli space captures much of the desired geometric properties in the following sense:
\begin{lemma}[{\cite[\S 1.2]{Alp13}}]
If $f:\mathfrak{X}\rightarrow Y$ is a good moduli space, then
\begin{enumerate}
\item
$f$ is surjective and universally closed.
\item
For any algebraically closed field $k$ over $\field$, two geometric points $x_1$ and $x_2$ in $\mathfrak{X}(k)$ 
are identified in $Y$ if and only if their closures $\overline{\{x_1\}}$ and $\overline{\{x_2\}}$ intersect.
\item
If $Y'\rightarrow Y$ is any map of algebraic spaces over $\field$, then 
the base change $f_{Y'}:Y'\times_Y\mathfrak{X}\rightarrow Y'$ is a good moduli space.
\item
If $\mathfrak{X}$ is locally noetherian, then $f$ is universal for maps to algebraic spaces. In particular, 
the good moduli space, when exists, is \emph{unique}.
\item
If $\mathfrak{X}$ is finite type over $\field$, then so is $Y$.
\end{enumerate} 
\end{lemma}

We also recall two important examples of good moduli spaces:
\begin{lemma}[{\cite[E.g. 8.1, 8.3]{Alp13}}]\label{lem:example_of_good_moduli_space}
Let $\field$ be a base field.
\begin{enumerate}
\item
If $\mathfrak{X}$ is an algebraic stack over $\field$ with finite inertia, and $f:\mathfrak{X}\rightarrow Y$ is 
the associated coarse moduli space. Suppose the push-forward functor $f_*$ on quasi-coherent sheaves is exact, 
then $f$ is a good moduli space. 
\item
If $G$ is a linear reductive algebraic group acting on an affine scheme $X=\Spec A$ over $\field$, then 
the natural map $f:[X/G]\rightarrow \Spec~A^G$ is a good moduli space. Here, $[X/G]$ is the quotient stack.
\end{enumerate}
\end{lemma}

\noindent{}\textbf{Note}: Contrary to (2) above, when $X$ is a non-affine scheme, the quotient stack $[X/G]$ 
may not have a good moduli space, see \cite[e.g.3.16, Thm.4.1]{AHLH23}. 

Now, we can define the Betti moduli spaces in our setting. For any manifold $X$ and any color 
$\vec{n}:\pi_0(X)\rightarrow \mathbb{Z}_{\geq 0}$, let $\Mloc_{\vec{n}}(X;\field)$ be the moduli stack of rank $\vec{n}$ 
local systems on $X$. 
Fix a \emph{map/color} $\vec{r}:\pi_0(\Lambda)\rightarrow\mathbb{Z}_{\geq 0}$ so that the \emph{total rank}
\[
  r_{\tot}:=\text{\tiny$\sum_{[\Lambda_i^j]\in\pi_0(\Lambda_i)}$} \deg(\Lambda_i^j)\cdot \vec{r}([\Lambda_i^j])
\]
is independent of $i$, where $\deg(\Lambda_i^j)$ denotes the winding number of $\Lambda_i^j$ encircling $p_i$. 
Taking the microlocal monodromy, we obtain an induced map of algebraic stacks
\[
\mumon:\modulistack_{\vec{r}}(C,\{p_i\},\{\Lambda_i\};\field) \rightarrow \Mloc_{\vec{r}}(\Lambda;\field).
\]

\begin{definition}\label{def:Betti_moduli_space}
Fix a rank $\vec{r}$ local system $L$ on $\Lambda$.
The \emph{Betti moduli stack} of constructible sheaves with microlocal monodromy $L$ on $(C,\{p_i\},\{\Lambda_i\})$ is
$\modulistack_{\vec{r}}(C,\{p_i\},\{\Lambda_i\};\field):=\mumon^{-1}([L]/\mathrm{Aut}(L))$.
The \emph{Betti moduli space} of constructible sheaves with microlocal monodromy $L$ on $(C,\{p_i\},\{\Lambda_i\})$ 
is the good moduli space $\modulispace_L(C,\{p_i\},\{\Lambda_i\};\field)$ associated to the algebraic stack 
$\modulistack_L(C,\{p_i\},\{\Lambda_i\};\field)$. 
\end{definition}

\section{Betti moduli spaces and braids}\label{sec:Betti_moduli_spaces_and_braids}

In this section, we give a concrete description of the Betti moduli space via braids.
Let $(C,\{p_i\},\{\Lambda_i\})$ and $\vec{r}$ be as above. Denote $G:=GL(r_{\tot})$. Let $D_i$ be a small open disk centered 
at $p_i$ so that the front projection of $\Lambda_i$ lives in $D_i^*=D_i\setminus\{p_i\}$. Let $D_i'$ be a slightly larger disk 
so that $D_i'\setminus\overline{D}_i$ is an open annulus encircling $p_i$. Denote by $S_i$ a fixed clockwise circle in the annulus.
Denote $C^o:=C\setminus\sqcup_i\overline{D}_i$.
By the gluing property, we have a cartesian diagram of algebraic stacks compatible with taking microlocal monodromy:
\begin{equation}\label{eqn:gluing_property_for_global_Betti_moduli_stacks}
\begin{tikzcd}[row sep=1pc,column sep=0.5pc]
\modulistack_{\vec{r}}(C,\{p_i\},\{\Lambda_i\})\arrow[r]\arrow[d]\arrow[dr,phantom,"\lrcorner",very near start] 
& \modulistack_{r_{\tot}}(C^o)=\{(A_j,M_i)\in G^{2g+k}: 
\text{\tiny$\prod_{j=1}^g$}(A_{2j-1},A_{2j})\text{\tiny$\prod_{i=1}^k$}M_i=\id\}/G\arrow[d,"{\pr=\text{\tiny$\prod_i$}\pr_i}"]\\
\modulistack_{\vec{r}}(\sqcup_i D_i',\{p_i\},\{\Lambda_i\})\arrow[r,"{\mon}"]\arrow[d,"{\mumon}"] & 
\Mloc_{r_{\tot}}(\sqcup_i S_i)=\text{\tiny$\prod_{i=1}^k$}G/G\\
\Mloc_{\vec{r}}(\Lambda) & 
\end{tikzcd}
\end{equation}
Here, $\pr_i$ sends $[(A_j,M_i)]$ to $[M_i]$. Each $\Lambda_i$ can be represented by a \emph{cylindrical closure} $\beta_i^{\circ}$
of a positive braid $\beta_i$ in the following sense \cite[\S 6.1]{STZ17}:
\begin{itemize}[wide,labelwidth=!,labelindent=0pt]
\item 
The pair $(D_i'^*,\Lambda_i)$ is identified with $(S_x^1\times\mathbb{R}_z,\beta_i^{\circ})$, where $\beta_i^{\circ}$ is 
a Legendrian link in $J^1S_x^1=T^{\infty,-}(S_x^1\times\mathbb{R}_z)$ defined as 
in Figure \ref{fig:cylindrical_closure}. Then, $S_i$ may be identified with $S_x^1\times\{z_0\}$ for some large $z_0$.
\end{itemize}

\begin{figure}[!htbp]
  \vspace{-0.2in}
  \begin{center}
  \begin{tikzcd}[row sep=1pc,column sep=8pc,ampersand replacement=\&]
  
  \begin{tikzpicture}[baseline=-.5ex,scale=0.4]
  \begin{scope}
  \draw[thick] (0,0) to[out=0,in=180] (2,1);
  \draw[white, fill=white] (1,0.5) circle (0.2);
  \draw[thick] (0,1) to[out=0,in=180] (2,0) to[out=0,in=180] (4,1);
  \draw[white, fill=white] (3,0.5) circle (0.2);
  \draw[thick] (2,1) to[out=0,in=180] (4,0) to[out=0,in=180] (6,1);
  \draw[white, fill=white] (5,0.5) circle (0.2);
  \draw[thick] (4,1) to[out=0,in=180] (6,0) to[out=0,in=180] (8,1);
  \draw[white, fill=white] (7,0.5) circle (0.2);
  \draw[thick] (6,1) to[out=0,in=180] (8,0) to[out=0,in=180] (10,1);
  \draw[white, fill=white] (9,0.5) circle (0.2);
  \draw[thick] (8,1) to[out=0,in=180] (10,0);
  \draw (5,-0.5) node[below] {{\small $\beta=\sigma_1^5$}};
  \end{scope}
  \end{tikzpicture}
  
  \arrow[r,"{\tiny cylindrical~closure}"]\&
  
  \begin{tikzpicture}[baseline=-.5ex,scale=0.4]
  \begin{scope}
  \draw[dashed] (11,-2) arc(0:180:6cm and 1.2cm);
  \draw (-1,-2) arc(-180:0:6cm and 1.2cm);
  \draw[black, fill=black] (5,-2) circle (0.05);
  \draw (5,5) ellipse (6cm and 1.2cm);
  \draw[black, fill=black] (5,5) circle (0.05);
  \draw[thin,dashed] (5,-4)--(5,6);
  \draw[->] (5,5)--(5,7);
  \draw (5,7) node[right] {{\small $z$}};
  \draw[thick] (-1,-2)--(-1,5);
  \draw[thick] (11,-2)--(11,5);

  \draw[thick,green] (-1,1) to[out=-90,in=180] (0,0) to[out=0,in=180] (2,1);
  \draw[white, fill=white] (1,0.5) circle (0.2);
  \draw[thick,green] (-1,2) to[out=-90,in=180] (0,1) to[out=0,in=180] (2,0) to[out=0,in=180] (4,1);
  \draw[white, fill=white] (3,0.5) circle (0.2);
  \draw[thick,green] (2,1) to[out=0,in=180] (4,0) to[out=0,in=180] (6,1);
  \draw[white, fill=white] (5,0.5) circle (0.2);
  \draw[thick,green] (4,1) to[out=0,in=180] (6,0) to[out=0,in=180] (8,1);
  \draw[white, fill=white] (7,0.5) circle (0.2);
  \draw[thick,green] (6,1) to[out=0,in=180] (8,0) to[out=0,in=180] (10,1) to[out=0,in=-90] (11,2);
  \draw[white, fill=white] (9,0.5) circle (0.2);
  \draw[thick,green] (8,1) to[out=0,in=180] (10,0) to[out=0,in=-90] (11,1);
  
  \draw[thick,dashed,green] (11,1) arc(0:180:6cm and 1.2cm);
  \draw[thick,dashed,green] (11,2) arc(0:180:6cm and 1.2cm);
  
  \draw (5.5,-0.5) node[below] {{\small {\color{green}$\beta^{\circ}$}$\hookrightarrow J^1S_x^1$}};

  \draw[->] (0,-3) arc(-150:-30:6cm and 1.2cm);
  \draw(10,-3) node[right] {{\tiny $S_x^1$}};
   
  \end{scope}
  \end{tikzpicture}
  
  \end{tikzcd}
  \end{center}
  \vspace{-0.2in}
  \caption{Cylindrical closure: $\beta=\sigma_1^5\in\mathrm{Br}_2^+$ $\rightsquigarrow$ 
  $\beta^{\circ}\subset J^1S_x^1\isomorphic T^{\infty,-}(S_x^1\times\mathbb{R}_z)$, 
  where we identifify the Legendrian link $\beta^{\circ}$ with its front projection in $S_x^1\times\mathbb{R}_z$.}
  \label{fig:cylindrical_closure}
  \end{figure}

As a consequence, the description of $\modulistack_{\vec{r}}(C,\{p_i\},\{\Lambda_i\})$ is reduced to the following problem:
\begin{itemize}[wide,labelwidth=!,labelindent=0pt]
  \item[\textbf{Problem}:]
  Let $\beta\in\Br_n^+$ be an $n$-strand positive braid with cylindrical closure $\beta^{\circ}$ in 
  $S_x^1\times\mathbb{R}_z$. Fix a color $\vec{r}:\pi_0(\beta^{\circ})\rightarrow\mathbb{Z}_{\geq 0}$ with total rank $r_{\tot}$. 
  Fix a point $q_m$ below $\beta^{\circ}$ in $S_x^1\times\mathbb{R}_z$ and a large number $z_0\in\mathbb{R}_z$. 
  Denote
  $\modulistack_{\vec{r}}(\beta^{\circ}):=\modulistack_{\vec{r}}(S_x^1\times\mathbb{R}_z,\{q_m\},\beta^{\circ})$, 
  the moduli stack of microlocal rank $\vec{r}$ constructible sheaves in 
  $\Sh_{\beta^{\circ}}(S_x^1\times\mathbb{R}_z,\{q_m\})$. We would like to describe the following diagram of algebraic stacks:
\[
\begin{tikzcd}
\Mloc_{r_{\tot}}(S_x^1\times\{z_0\}) & \modulistack_{\vec{r}}(\beta^{\circ})\arrow[r,"{\mumon}"]\arrow[l,swap,"{\mon}"] & \Mloc_{\vec{r}}(\beta^{\circ})
\end{tikzcd}
\]
\end{itemize}

To deal with this problem, we make some preparations.

\subsection{Diagram calculus of block matrices}\label{subsec:diagram_calculus_of_block_matrices}

We introduce some diagram calculus of \textbf{block} matrices. See also \cite[\S 1.3]{Su23}.
Denote
\begin{eqnarray*}
&&\FBr_n^+:=\langle\sigma_1,\text{\tiny$\cdots$},\sigma_{n-1}\rangle;~~(\text{free monoid of $n$-strand (positive) braid presentations})\\
&&\Br_n^+:=\FBr_n^+/(\sigma_i\sigma_{i+1}\sigma_i = \sigma_{i+1}\sigma_i\sigma_{i+1}, \forall i;~\sigma_i\sigma_j =\sigma_j\sigma_i, \forall |i-j|>1);~~(\text{$n$-strand positive braids})\\
&&\permutation:\Br_n^+\rightarrow \Br_n^+/(\sigma_i^2=1,\forall i)\isomorphic S_n: \sigma_k\mapsto \ms_k:=\permutation(\sigma_k)=(k~k+1).~~(\text{underlying permuations})
\end{eqnarray*}

\noindent{}\textbf{Convention} \setword{{\color{blue}$1$}}{convention:braid_diagrams}: 
As in Figure \ref{fig:braid_diagram}, any $\beta=\sigma_{i_{\ell}}\text{\tiny$\cdots$}\sigma_{i_1}\in\FBr_n^+$ is represented by the \emph{braid diagram} $[\sigma_{i_1}|\text{\tiny$\cdots$}|\sigma_{i_{\ell}}]$ going \emph{from left to right}, where $[\beta_1|\beta_2]$ is the concatenation of $\beta_1$ with $\beta_2$. Label the left (resp. right) ends \emph{from bottom to top} by $1,2,\text{\tiny$\cdots$},n$. Denote $[n]:=\{1,2,\text{\tiny$\cdots$},n\}$.

\begin{figure}[!htbp]
\begin{center}
\begin{tikzpicture}[baseline=-.5ex,scale=0.5]
\begin{scope}

\draw[thin] (0,0) to [out=0,in=180] (2,1) to [out=0,in=180] (4,2);

\draw[white,fill=white] (1,0.5) circle(0.1);
\draw[white,fill=white] (3,1.5) circle(0.1);

\draw[thin] (0,1) to [out=0,in=180] (2,0)--(4,0);

\draw[thin] (0,2)--(2,2) to [out=0,in=180] (4,1);

\draw (0,0) node[left] {\text{\tiny$1$}};
\draw (0,1) node[left] {\text{\tiny$2$}};
\draw (0,2) node[left] {\text{\tiny$3$}};

\draw (4,0) node[right] {\text{\tiny$1$}};
\draw (4,1) node[right] {\text{\tiny$2$}};
\draw (4,2) node[right] {\text{\tiny$3$}};

\draw (2,-0.5) node[below] {\text{\tiny$[\sigma_1|\sigma_2]=\sigma_2\sigma_1$}};

\end{scope}
\end{tikzpicture}
\end{center}
\vspace{-0.2in}
\caption{Braid diagram for a positive braid: $n=3$.}
\label{fig:braid_diagram}
\end{figure}
\vspace{-0.1in}

\begin{definition}[Colored positive braids and permutations]\label{def:colored_positive_braids}
Fix $0\leq n\leq N$.
\begin{enumerate}[wide,labelwidth=!,labelindent=0pt]
\item
A \emph{colored positive braid presentation} of rank $(n,N)$ is a pair $(\beta,\vec{r})$ (or $\beta^{\vec{r}}$), where $\beta\in\FBr_n^+$ and $\vec{r}:\pi_0(\beta)\rightarrow\mathbb{N}$ is a \emph{color} of total rank $N$.
Restricting to the left ends (resp. right ends), the color can be identified with a map $\vec{r}_{\Left}:[n]\rightarrow\mathbb{N}$ (resp. $\vec{r}_{\Right}:[n]\rightarrow\mathbb{N}$). Clearly, $\vec{r}_{\Left}=\vec{r}_{\Right}\circ \permutation(\beta)$.

Denote by $\FBr_{n,N}^+$ the set of all colored positive braid presentations of rank $(n,N)$. If $[\beta_1^{\vec{r}_1}|\beta_2^{\vec{r}_2}]$ is a compatible concatenation in $\FBr_{n,N}^+$, i.e. $\vec{r}_{1,\Right}=\vec{r}_{2,\Left}$, 
define $\beta_2^{\vec{r}_2}\circ\beta_1^{\vec{r}_1}:=[\beta_1^{\vec{r}_1}|\beta_2^{\vec{r}_2}]\in\FBr_{n,N}$, called \emph{admissible composition}.

\item
Replace $\FBr_n^+$ by $\Br_n^+$ (resp. $S_n$), define a \emph{colored positive braid} (resp. \emph{colored permutation}) of rank $(n,N)$ and $\Br_{n,N}^+$ (resp. $S_{n,N}$) similarly. 
The morphism respecting admissible compositions is induced by:
\[
\permutation:\Br_{n,N}^+\rightarrow S_{n,N}: \sigma_k^{\vec{r}}\mapsto s_k^{\vec{r}}=(s_k,\vec{r}).
\]
\end{enumerate}
\end{definition}

Fix $N\in\mathbb{N}$. For any color $\vec{r}:[n]\rightarrow \mathbb{N}$ with $r_{\tot}=N$, and any matrix $\epsilon\in M_{\vec{r}(i)\times\vec{r}(j)}(\field)$, define
\[
E_{i,j}^{\vec{r}}(\epsilon):=\left(\begin{array}{cccc}
0 & \cdots & 0 & 0\\
0 & \cdots & \epsilon & 0\\
\vdots & \ddots & \vdots & \vdots\\
0 & \cdots & 0 & 0
\end{array}\right)\in M_{N\times N}(\field),
\]
where $\epsilon$ sits in the $(i,j)$-block. When there's no confusion, we may also write $E_{i,j}(\epsilon):=E_{i,j}^{\vec{r}}(\epsilon)$.

Then, we make the following
\begin{definition}\label{def:elementary_braid_matrix_diagrams}
Let $\vec{r}_{\Left},\vec{r_{\Right}}:[n]\rightarrow\mathbb{N}$ be two maps such that $r_{L,\tot}=r_{R,\tot}=N$. Denote $G:=GL(N,\field)$.
\begin{enumerate}[wide,labelwidth=!,labelindent=0pt]
\item
Suppose $\vec{r}_{\Left}=\vec{r}_{\Right}=:\vec{r}$ defines a color on $\id_n\in\mathrm{Br}_n^+$. For any $1\leq k\leq n$ and any $\epsilon\in GL(\vec{r}(k),\field)$, define
\[
\bfK_k^{\vec{r}}(\epsilon):=\text{\tiny$\sum_{i\neq k}$} E_{i,i}(I_{\vec{r}(i)}) + E_{k,k}(\epsilon) \in G;\quad [\bfK_k^{\vec{r}}(\epsilon)]:= \text{ Figure \ref{fig:elementary_braid_matrix_diagrams} (left)}.
\]
If $\vec{r}=1$, simply denote $\mK_k(\epsilon):=\bfK_k^{\vec{r}}(\epsilon)$, etc.

\item
Suppose $\vec{r}_{\Left}=\vec{r}_{\Right}=:\vec{r}$ defines a color on $\id_n\in\mathrm{Br}_n^+$. 
For any $1\leq i< j\leq n$ and any $\epsilon\in M_{\vec{r}(i)\times\vec{r}(j)}(\field)$, define
\[
\bfH_{i,j}^{\vec{r}}(\epsilon):=I_N + E_{i,j}(\epsilon) \in G;\quad [\bfH_{i,j}^{\vec{r}}(\epsilon)]:= \text{ Figure \ref{fig:elementary_braid_matrix_diagrams} (middle)}.
\]
Denote $\bfH_k^{\vec{r}}(\epsilon):=\bfH_{k,k+1}^{\vec{r}}(\epsilon)$, etc.

\item
Suppose $\vec{r}_{\Left}=\vec{r}_{\Right}\circ s_k$ for some $1\leq k\leq n-1$, so $\vec{r}:=(\vec{r}_{\Right},\vec{r}_{\Left})$ defines a color on $\sigma_k\in\mathrm{Br}_n^+$.
Define 
\[
\bfs_k^{\vec{r}}:=\text{\tiny$\sum_{i\neq k,k+1}$} E_{i,i}(I_{\vec{r}(i)}) + E_{k+1,k}^{\vec{r}_{\Left}}(I_{\vec{r}_{\Left}(k)}) + E_{k,k+1}^{\vec{r}_{\Left}}(I_{\vec{r}_{\Left}(k+1)})\in G;
\quad [\bfs_k^{\vec{r}}]=[s_k^{\vec{r}}]:=\sigma_k^{\vec{r}}~(\text{Figure \ref{fig:elementary_braid_matrix_diagrams} (right)}).
\]
Define a morphism respecting admissible compositions
\[
\overallpermutation:S_{n,N}\rightarrow S_N\subset G: s_k^{\vec{r}}=(s_k,\vec{r})\mapsto \bfs_k^{\vec{r}}~~(\text{\emph{overall permutation}});
~~\rightsquigarrow~~ \bfs:=\overallpermutation\circ\permutation:\Br_{n,N}^+\rightarrow G.
\]
\end{enumerate}
Each of $[\bfK_k^{\vec{r}}(\epsilon)]$, $[\bfH_{i,j}^{\vec{r}}(\epsilon)]$, $[\bfs_k^{\vec{r}}]$ is called a \emph{(block) elementary braid matrix diagram} of rank $(n,N)$ over $\field$.
\end{definition}

\begin{figure}[!htbp]
\vspace{-0.2in}
\begin{center}
\begin{tikzcd}[row sep=0.5pc,column sep=1pc,ampersand replacement=\&]
\begin{tikzpicture}[baseline=-.5ex,scale=0.5]
\begin{scope}

\draw[thin] (0,0)--(4,0);
\draw[thin] (0,2.5)--(4,2.5);
\draw[thin] (0,5)--(4,5);
\draw (2,1.25) node[right] {{\tiny$\vdots$}};
\draw (2,3.75) node[right] {{\tiny$\vdots$}};

\draw (0,0) node[left] {{\tiny$1$}};
\draw (0,1.25) node[left] {{\tiny$\vdots$}};
\draw (0,2.5) node[left] {{\tiny$k$}};
\draw (0,3.75) node[left] {{\tiny$\vdots$}};
\draw (0,5) node[left] {{\tiny$n$}};

\draw (4,0) node[right] {{\tiny$1$}};
\draw (4,1.25) node[right] {{\tiny$\vdots$}};
\draw (4,2.5) node[right] {{\tiny$k$}};
\draw (4,3.75) node[right] {{\tiny$\vdots$}};
\draw (4,5) node[right] {{\tiny$n$}};

\draw (0.5,5) node[above] {{\tiny$\vec{r}$}};

\draw (1.5,3) node[below] {$\ast$};
\draw (1.5,2.5) node[above] {$\epsilon$};

\draw (2,-0.5) node[below] {{\tiny$[\bfK_k^{\vec{r}}(\epsilon)]$}};

\end{scope}
\end{tikzpicture}
\&
\begin{tikzpicture}[baseline=-.5ex,scale=0.5]
\begin{scope}

\draw[thin] (0,0)--(4,0);
\draw[thin] (0,1.5)--(4,1.5);
\draw[thin] (0,3.5)--(4,3.5);
\draw[thin] (0,5)--(4,5);
\draw (2,0.75) node[right] {{\tiny$\vdots$}};
\draw (2,2.5) node[right] {{\tiny$\vdots$}};
\draw (2,4.25) node[right] {{\tiny$\vdots$}};

\draw (0,0) node[left] {{\tiny$1$}};
\draw (0,0.75) node[left] {{\tiny$\vdots$}};
\draw (0,1.5) node[left] {{\tiny$i$}};
\draw (0,2.5) node[left] {{\tiny$\vdots$}};
\draw (0,3.5) node[left] {{\tiny$j$}};
\draw (0,4.25) node[left] {{\tiny$\vdots$}};
\draw (0,5) node[left] {{\tiny$n$}};

\draw (4,0) node[right] {{\tiny$1$}};
\draw (4,0.75) node[right] {{\tiny$\vdots$}};
\draw (4,1.5) node[right] {{\tiny$i$}};
\draw (4,2.5) node[right] {{\tiny$\vdots$}};
\draw (4,3.5) node[right] {{\tiny$j$}};
\draw (4,4.25) node[right] {{\tiny$\vdots$}};
\draw (4,5) node[right] {{\tiny$n$}};

\draw (0.5,5) node[above] {{\tiny$\vec{r}$}};

\draw[line width=0.5mm] (1.5,1.5)--(1.5,3.5);
\draw (1.5,2.5) node[left] {$\epsilon$};

\draw (2,-0.5) node[below] {{\tiny$[\bfH_{i,j}^{\vec{r}}(\epsilon)]$}};

\end{scope}
\end{tikzpicture}
\&
\begin{tikzpicture}[baseline=-.5ex,scale=0.5]
\begin{scope}

\draw[thin] (0,0)--(4,0);
\draw[thin] (0,2)--(1,2) to [out=0,in=180] (3,3)--(4,3);
\draw[white,fill=white] (2,2.5) circle (0.2);
\draw[thin] (0,3)--(1,3) to[out=0,in=180] (3,2)--(4,2);
\draw[thin] (0,5)--(4,5);
\draw (1.6,1) node[right] {{\tiny$\vdots$}};
\draw (1.6,4) node[right] {{\tiny$\vdots$}};

\draw (0,0) node[left] {{\tiny$1$}};
\draw (0,1) node[left] {{\tiny$\vdots$}};
\draw (0,2) node[left] {{\tiny$k$}};
\draw (0,3) node[left] {{\tiny$k+1$}};
\draw (0,4) node[left] {{\tiny$\vdots$}};
\draw (0,5) node[left] {{\tiny$n$}};

\draw (4,0) node[right] {{\tiny$1$}};
\draw (4,1) node[right] {{\tiny$\vdots$}};
\draw (4,2) node[right] {{\tiny$k$}};
\draw (4,3) node[right] {{\tiny$k+1$}};
\draw (4,4) node[right] {{\tiny$\vdots$}};
\draw (4,5) node[right] {{\tiny$n$}};

\draw (0.5,5) node[above] {{\tiny$\vec{r}_{\Left}$}};
\draw (3.5,5) node[above] {{\tiny$\vec{r}_{\Right}$}};

\draw (2,-0.5) node[below] {{\tiny$[\bfs_k^{\vec{r}}]=\sigma_k^{\vec{r}}$}};

\end{scope}
\end{tikzpicture}

\end{tikzcd}
\end{center}
\vspace{-0.2in}
\caption{Block elementary matrix diagrams representing block elementary matrices: $[\bfK_k^{\vec{r}}(\epsilon)]$ (\emph{block scaling}), $[\bfH_{i,j}^{\vec{r}}(\epsilon)]$ for $i<j$ (\emph{block handleslide}), $[\bfs_k^{\vec{r}}]$ (\emph{block transposition}).} 
\label{fig:elementary_braid_matrix_diagrams}
\end{figure}

\begin{definition}[(Block) braid matrix diagram presentations]~\label{def:braid_matrix_diagram_presentations}
\begin{enumerate}[wide,labelwidth=!,labelindent=0pt]
\item
A \emph{braid matrix diagram presentation} ($\bmdp$) of rank $(n,N)$ over $\field$ is a finite concatenation of elementary braid matrix diagrams of rank $(n,N)$, 
where the concatenation is compatible with the coloring.
Define $\FBD_{n,N}$ to be the set of all $\bmdp$'s of rank $(n,N)$, modulo the relation generated by $[\mK_k^{\vec{r}}(I_{\vec{r}(k)})]=\id_n^{\vec{r}}=[\mH_{i,j}^{\vec{r}}(0)]$ under compatible concatenations.
If $[\Gamma_1|\Gamma_2]$ is a compatible concatenation in $\FBD_{n,N}$, define $\Gamma_2\circ\Gamma_1:=[\Gamma_1|\Gamma_2]$, called \emph{admissible composition}.

\item
Define a morphism respecting (admissible) compositions by 
\[
g_{-}:\FBD_{n,N}\rightarrow G: \sigma_k^{\vec{r}}\mapsto \bfs_k^{\vec{r}},~~[\bfK_k^{\vec{r}}(\epsilon)]\mapsto \bfK_k^{\vec{r}}(\epsilon),~~[\bfH_{i,j}^{\vec{r}}(\epsilon')]\mapsto \bfH_{i,j}^{\vec{r}}(\epsilon').
\]
Two $\bmdp$'s $\Gamma_1,\Gamma_2\in\FBD_{n,N}$ are \emph{weakly equivalent} if $g_{\Gamma_1}=g_{\Gamma_2}$, denoted as $\Gamma_1\weakequivalent \Gamma_2$.
\end{enumerate}
\end{definition}

The following lemma is straightforward:
\begin{lemma}[\emph{Elementary moves} of $\bmdp$'s]~\label{lem:elementary_moves}
For $s_k\in S_n$, denote $i':=\ms_k(i)$, then
\begingroup
\setlength{\tabcolsep}{4pt} 
\renewcommand{\arraystretch}{1.2} 

\begin{tabular}{ll}
\minipage{0.42\linewidth}
\[
\bfK_k^{\vec{r}}(\epsilon_1)\text{\tiny$\circ$}\bfK_{\ell}^{\vec{r}}(\epsilon_2)=\left\{\begin{array}{ll}
\bfK_k^{\vec{r}}(\epsilon_1\epsilon_2) & \text{{\small$k=\ell$}},\\
\bfK_{\ell}^{\vec{r}}(\epsilon_2)\text{\tiny$\circ$} \bfK_k^{\vec{r}}(\epsilon_1) & \text{{\small$k\neq\ell$}}.
\end{array}\right.
\]
\endminipage
&
\minipage{0.53\linewidth}
\[
\bfH_{i,j}^{\vec{r}}(\epsilon_1)\text{\tiny$\circ$}\bfK_k^{\vec{r}}(\epsilon_2)=\left\{\begin{array}{ll}
\bfK_i^{\vec{r}}(\epsilon_2)\text{\tiny$\circ$}\bfH_{i,j}^{\vec{r}}(\epsilon_2^{-1}\epsilon_1) & \text{$k=i$},\\
\bfK_i^{\vec{r}}(\epsilon_2)\text{\tiny$\circ$}\bfH_{i,j}^{\vec{r}}(\epsilon_1\epsilon_2) & \text{$k=j$},\\
\bfK_i^{\vec{r}}(\epsilon_2)\text{\tiny$\circ$}\bfH_{i,j}^{\vec{r}}(\epsilon_1) & \text{$k\neq i,j$}.
\end{array}\right.
\]
\endminipage
\end{tabular}
\endgroup

\begingroup
\setlength{\tabcolsep}{4pt} 
\renewcommand{\arraystretch}{1.2} 
\begin{tabular}{ll}
\minipage{0.27\textwidth}
\[
\left.\begin{array}{l}
\bfs_k^{\vec{r}}\text{\tiny$\circ$}\bfK_i^{\vec{r}_{\Left}}(\epsilon)=\bfK_{i'}^{\vec{r}_{\Right}}(\epsilon)\text{\tiny$\circ$}\bfs_k^{\vec{r}};\\
\bfs_k^{\vec{r}}\text{\tiny$\circ$}\bfH_{i,j}^{\vec{r}_{\Left}}(\epsilon)=\bfH_{i',j'}^{\vec{r}_{\Right}}(\epsilon)\text{\tiny$\circ$}\bfs_k^{\vec{r}}.
\end{array}\right.
\]
\endminipage\hfill
&
\minipage{0.68\textwidth}
\[
\bfH_{i,j}^{\vec{r}}(\epsilon_1)\text{\tiny$\circ$}\bfH_{k,\ell}^{\vec{r}}(\epsilon_2)=\left\{\begin{array}{ll}
\bfH_{i,j}^{\vec{r}}(\epsilon_1\text{\tiny$+$}\epsilon_2) & \text{$(k,\ell)=(i,j)$},\\
\bfH_{j,\ell}^{\vec{r}}(\epsilon_2)\text{\tiny$\circ$}\bfH_{i,\ell}^{\vec{r}}(\epsilon_1\epsilon_2)\text{\tiny$\circ$}\bfH_{i,j}^{\vec{r}}(\epsilon_1) & \text{$k=j$},\\
\bfH_{k,i}^{\vec{r}}(\epsilon_2)\text{\tiny$\circ$}\bfH_{k,j}^{\vec{r}}(\text{\tiny$-$}\epsilon_2\epsilon_1)\text{\tiny$\circ$}\bfH_{i,j}^{\vec{r}}(\epsilon_1) & \text{$\ell=i$},\\
\bfH_{k,\ell}^{\vec{r}}(\epsilon_2)\text{\tiny$\circ$}\bfH_{i,j}^{\vec{r}}(\epsilon_1) & \text{else}.
\end{array}\right.
\]
\endminipage\hfill
\end{tabular}
\endgroup

\noindent{}Each identity is either trivial (commutative), or represented by an (\emph{elementary}) move in Figure \ref{fig:elementary_moves}.
\end{lemma}

\begin{figure}[!htbp]
\vspace{-0.1in}
\begin{center}
\begin{tikzcd}[row sep=0.5pc,column sep=0.5pc,ampersand replacement=\&]

\begin{tikzpicture}[baseline=-.5ex,scale=0.5]
\begin{scope}

\draw[thin] (0,0)--(1.5,0);
\draw[thin] (0,1)--(1.5,1);
\draw[thin] (0,2)--(1.5,2);

\draw (0.5,1.5) node[below] {$\ast$};
\draw (0.5,1) node[below] {\text{\tiny$\epsilon_2$}};

\draw (1,1.5) node[below] {$\ast$};
\draw (1,1) node[above] {\text{\tiny$\epsilon_1$}};

\end{scope}
\end{tikzpicture}

\arrow[r,leftrightarrow,line width=0.1mm,shift left=4,"{(1)}"]\&

\begin{tikzpicture}[baseline=-.5ex,scale=0.5]
\begin{scope}

\draw[thin] (0,0)--(1.5,0);
\draw[thin] (0,1)--(1.5,1);
\draw[thin] (0,2)--(1.5,2);

\draw (0.75,1.5) node[below] {$\ast$};
\draw (0.75,1) node[below] {\text{\tiny$\epsilon_1\epsilon_2$}};

\end{scope}
\end{tikzpicture}

\&

\begin{tikzpicture}[baseline=-.5ex,scale=0.5]
\begin{scope}

\draw[thin] (0,0)--(1.5,0);
\draw[thin] (0,2)--(1.5,2);

\draw (0.5,0.5) node[below] {$\ast$};
\draw (0.5,0) node[below] {\text{\tiny$\epsilon_2$}};

\draw[line width=0.5mm] (1,0)--(1,2);
\draw (0.8,1.2) node[right] {\text{\tiny$\epsilon_1$}};

\end{scope}
\end{tikzpicture}

\arrow[r,leftrightarrow,line width=0.1mm,shift left=4,"{(2)}"]\&

\begin{tikzpicture}[baseline=-.5ex,scale=0.5]
\begin{scope}

\draw[thin] (0,0)--(1.5,0);
\draw[thin] (0,2)--(1.5,2);

\draw (1,0.5) node[below] {$\ast$};
\draw (1,0) node[below] {\text{\tiny$\epsilon_2$}};

\draw[line width=0.5mm] (0.5,0)--(0.5,2);
\draw (0.3,1.2) node[right] {\text{\tiny$\epsilon_2^{-1}\epsilon_1$}};

\end{scope}
\end{tikzpicture}

\&

\begin{tikzpicture}[baseline=-.5ex,scale=0.5]
\begin{scope}

\draw[thin] (0,0)--(1.5,0);
\draw[thin] (0,2)--(1.5,2);

\draw (0.5,2.5) node[below] {$\ast$}; 
\draw (0.7,2) node[above] {\text{\tiny$\epsilon_2$}};

\draw[line width=0.5mm] (1,0)--(1,2);
\draw (0.8,0.8) node[right] {\text{\tiny$\epsilon_1$}};

\end{scope}
\end{tikzpicture}

\arrow[r,leftrightarrow,line width=0.1mm,shift left=4,"{(3)}"]\&

\begin{tikzpicture}[baseline=-.5ex,scale=0.5]
\begin{scope}

\draw[thin] (0,0)--(1.5,0);
\draw[thin] (0,2)--(1.5,2);

\draw (1,2.5) node[below] {$\ast$}; 
\draw (1,2) node[above] {\text{\tiny$\epsilon_2$}};

\draw[line width=0.5mm] (0.5,0)--(0.5,2);
\draw (0.3,0.8) node[right] {\text{\tiny$\epsilon_1\epsilon_2$}};

\end{scope}
\end{tikzpicture}

\&

\begin{tikzpicture}[baseline=-.5ex,scale=0.5]
\begin{scope}

\draw[thin] (0,0.5)--(0.5,0.5) to [out=0,in=180] (2.5,1.5);
\draw[white,fill=white] (1.5,1) circle (0.2);
\draw[thin] (0,1.5)--(0.5,1.5) to[out=0,in=180] (2.5,0.5);

\draw (0.5,1) node[below] {$\ast$}; 
\draw (0.5,0.5) node[below] {\text{\tiny$\epsilon$}};

\end{scope}
\end{tikzpicture}

\arrow[r,leftrightarrow,line width=0.1mm,shift left=4,"{(4)}"]\&

\begin{tikzpicture}[baseline=-.5ex,scale=0.5]
\begin{scope}

\draw[thin] (0,0.5) to [out=0,in=180] (2,1.5)--(2.5,1.5);
\draw[white,fill=white] (1,1) circle (0.2);
\draw[thin] (0,1.5) to [out=0,in=180] (2,0.5)--(2.5,0.5);

\draw (2,2) node[below] {$\ast$}; 
\draw (2,1.5) node[above] {\text{\tiny$\epsilon$}};

\end{scope}
\end{tikzpicture}

\end{tikzcd}
\end{center}

\begin{center}
\begin{tikzcd}[row sep=0.5pc,column sep=0.5pc,ampersand replacement=\&]

\begin{tikzpicture}[baseline=-.5ex,scale=0.5]
\begin{scope}

\draw[thin] (0,0)--(1.5,0);
\draw[thin] (0,2)--(1.5,2);

\draw[line width=0.5mm] (0.5,0)--(0.5,2);
\draw (0.8,1) node[left] {\text{\tiny$\epsilon_2$}};

\draw[line width=0.5mm] (1,0)--(1,2);
\draw (0.7,1) node[right] {\text{\tiny$\epsilon_1$}};

\end{scope}
\end{tikzpicture}

\arrow[r,leftrightarrow,line width=0.1mm,shift left=4,"{(5)}"]\&

\begin{tikzpicture}[baseline=-.5ex,scale=0.5]
\begin{scope}

\draw[thin] (0,0)--(1.5,0);
\draw[thin] (0,2)--(1.5,2);

\draw[line width=0.5mm] (0.5,0)--(0.5,2);
\draw (0.2,1) node[right] {\text{\tiny$\epsilon_1+\epsilon_2$}};

\end{scope}
\end{tikzpicture}

\&

\begin{tikzpicture}[baseline=-.5ex,scale=0.5]
\begin{scope}

\draw[thin] (0,0)--(1.5,0);
\draw[thin] (0,1)--(1.5,1);
\draw[thin] (0,2)--(1.5,2);

\draw[line width=0.5mm] (1,0)--(1,1);
\draw (0.7,0.5) node[right] {\text{\tiny$\epsilon_1$}};

\draw[line width=0.5mm] (0.5,1)--(0.5,2);
\draw (0.8,1.5) node[left] {\text{\tiny$\epsilon_2$}};

\end{scope}
\end{tikzpicture}

\arrow[r,leftrightarrow,line width=0.1mm,shift left=4,"{(6)}"]\&

\begin{tikzpicture}[baseline=-.5ex,scale=0.5]
\begin{scope}

\draw[thin] (0,0)--(2,0);
\draw[thin] (0,1)--(2,1);
\draw[thin] (0,2)--(2,2);

\draw[line width=0.5mm] (0.5,0)--(0.5,1);
\draw (0.8,0.5) node[left] {\text{\tiny$\epsilon_1$}};

\draw[line width=0.5mm] (1.5,1)--(1.5,2);
\draw (1.2,1.5) node[right] {\text{\tiny$\epsilon_2$}};

\draw[line width=0.5mm] (1,0)--(1,2);
\draw (0.7,0.5) node[right] {\text{\tiny$\epsilon_1\epsilon_2$}};

\end{scope}
\end{tikzpicture}

\&

\begin{tikzpicture}[baseline=-.5ex,scale=0.5]
\begin{scope}

\draw[thin] (0,0)--(1.5,0);
\draw[thin] (0,1)--(1.5,1);
\draw[thin] (0,2)--(1.5,2);

\draw[line width=0.5mm] (1,1)--(1,2);
\draw (0.7,1.5) node[right] {\text{\tiny$\epsilon_1$}};

\draw[line width=0.5mm] (0.5,0)--(0.5,1);
\draw (0.8,0.5) node[left] {\text{\tiny$\epsilon_2$}};

\end{scope}
\end{tikzpicture}

\arrow[r,leftrightarrow,line width=0.1mm,shift left=4,"{(7)}"]\&

\begin{tikzpicture}[baseline=-.5ex,scale=0.5]
\begin{scope}

\draw[thin] (0,0)--(2,0);
\draw[thin] (0,1)--(2,1);
\draw[thin] (0,2)--(2,2);

\draw[line width=0.5mm] (0.5,1)--(0.5,2);
\draw (0.8,1.5) node[left] {\text{\tiny$\epsilon_1$}};

\draw[line width=0.5mm] (1.5,0)--(1.5,1);
\draw (1.2,0.5) node[right] {\text{\tiny$\epsilon_2$}};

\draw[line width=0.5mm] (1,0)--(1,2);
\draw (0.7,1.5) node[right] {\text{\tiny$-\epsilon_2\epsilon_1$}};

\end{scope}
\end{tikzpicture}

\&

\begin{tikzpicture}[baseline=-.5ex,scale=0.5]
\begin{scope}

\draw[thin] (0.3,0)--(0.5,0) to [out=0,in=180] (2.5,1);
\draw[white,fill=white] (1.5,0.5) circle (0.2);
\draw[thin] (0.3,1)--(0.5,1) to[out=0,in=180] (2.5,0);
\draw[thin] (0.3,2)--(2.5,2);

\draw[line width=0.5mm] (0.5,1)--(0.5,2);
\draw (0.2,1.5) node[right] {\text{\tiny$\epsilon$}};

\end{scope}
\end{tikzpicture}

\arrow[r,leftrightarrow,line width=0.1mm,shift left=4,"{(8)}"]\&

\begin{tikzpicture}[baseline=-.5ex,scale=0.5]
\begin{scope}

\draw[thin] (0,0) to [out=0,in=180] (2,1)--(2.2,1);
\draw[white,fill=white] (1,0.5) circle (0.2);
\draw[thin] (0,1) to [out=0,in=180] (2,0)--(2.2,0);
\draw[thin] (0,2)--(2.2,2);

\draw[line width=0.5mm] (2,0)--(2,2);
\draw (1.7,1.5) node[right] {\text{\tiny$\epsilon$}};

\end{scope}
\end{tikzpicture}

\end{tikzcd}
\end{center}
\vspace{-0.1in}
\caption{\emph{Elementary moves} for $\bmdp$'s: The trivial ones are skipped.
Each move is a weak equivalence in $\FBD_{n,N}$ representing an (admissible) composition identity in Lemma \ref{lem:elementary_moves}, and vice versa. The coloring is omitted.
The composition goes from left to right: $\Gamma_2\circ\Gamma_1=[\Gamma_1|\Gamma_2]$.}
\label{fig:elementary_moves}
\end{figure}
\vspace{-0.1in}

\begin{definition}[(Block) braid matrix diagrams]\label{def:braid_matrix_diagrams}~
\begin{enumerate}[wide,labelwidth=!,labelindent=0pt]
\item
Let $\underline{\FBD}_{n,N}$ be the quotient of $\FBD_n^N$ by elementary moves. 
Then $\exists$ a morphism respecting admissible compositions: 
\[
\beta_{-}:\underline{\FBD}_{n,N}\rightarrow \FBr_{n,N}^+: \sigma_k^{\vec{r}}\mapsto \sigma_k^{\vec{r}},~~[\mK_k^{\vec{r}}(\epsilon)]\mapsto \id_n^{\vec{r}},~~[\mH_{i,j}^{\vec{r}}(\epsilon')]\mapsto \id_n^{\vec{r}}.
\]
\vspace{-0.1in} 

\item
The set $\BD_{n,N}$ of \emph{braid matrix diagrams} ($\bmd$) of rank $(n,N)$ is the quotient of $\underline{\FBD}_{n,N}$ by \emph{(colored) braid relations/moves}.
So, $\BD_{n,N}=\FBD_{n,N}/\braidequivalent$, with $\braidequivalent$ generated by elementary and braid moves.

\item
$\beta_{-}:\underline{\FBD}_{n,N}\rightarrow \FBr_{n,N}^+$ and $g_{-}:\FBD_{n,N}\rightarrow G$ induce morphisms respecting admissible compositions:
\[
\beta_{-}:\BD_{n,N}\rightarrow \Br_{n,N}^+,\quad g_{-}:\BD_{n,N}\rightarrow G.
\]
\end{enumerate}
\end{definition}

\begin{remark}\label{rem:positive_braids_vs_braid_matrix_diagrams}
We have natural morphisms (respecting admissible compositions)
\[
i:\FBr_{n,N}^+\rightarrow \underline{\FBD}_{n,N}: \sigma_k^{\vec{r}}\mapsto \sigma_k^{\vec{r}},\quad\rightsquigarrow\quad i:\Br_{n,N}^+\rightarrow \BD_{n,N}: \sigma_k^{\vec{r}}\mapsto \sigma_k^{\vec{r}}.
\]
$\beta_{-}\circ i=\id$, so $i$ induces embeddings $\FBr_{n,N}^+\hookrightarrow \underline{\FBD}_{n,N}$ and $\Br_{n,N}^+\hookrightarrow \BD_{n,N}$. This morally explains the terminology.
Altogether, we get a commutative diagram respecting admissible compositions:
\begin{equation}
\begin{tikzcd}[row sep=2pc,column sep=2pc]
&\FBr_{n,N}^+\arrow[d,hookrightarrow,swap,"{i}"]\arrow[dl,hookrightarrow]\arrow[r,twoheadrightarrow] & \Br_{n,N}^+\arrow[d,hookrightarrow,swap,"{i}"]\arrow[r,twoheadrightarrow,"{\permutation}"]\arrow[dr,"{\bfs}"] & S_{n,N}\arrow[d,"{\overallpermutation}"]\\
\FBD_{n,N}\arrow[r,twoheadrightarrow] & \underline{\FBD}_{n,N}\arrow[d,twoheadrightarrow,"{\beta_{-}}"]\arrow[r,twoheadrightarrow] & \BD_{n,N}\arrow[d,twoheadrightarrow,"{\beta_{-}}"]\arrow[r,twoheadrightarrow,"{g_{-}}"] & G\\
&\FBr_{n,N}^+\arrow[r,twoheadrightarrow] & \Br_{n,N}^+ & \\
\end{tikzcd}
\vspace{-0.2in}
\end{equation}
\end{remark}

Next, $\permutation:\Br_{n,N}^+\rightarrow S_{n,N}$ admits a \emph{canonical section} (as a map of sets) characterized by:
\begin{equation}\label{eqn:positive_braid_lifting_a_permutation}
[-]:S_{n,N}\rightarrow \Br_{n,N}^+\subset\BD_{n,N}: w^{\vec{r}}\mapsto [w^{\vec{r}}], \text{ with } \permutation([w^{\vec{r}}])=w^{\vec{r}}, \ell([w^{\vec{r}}])=\ell(w^{\vec{r}}).
\end{equation}
Here, $\ell(-)$ stands for length.
As we have seen, $[s_k^{\vec{r}}]=\sigma_k^{\vec{r}}$. Up to a choice, we may assume $[w^{\vec{r}}]\in\FBr_{n,N}^+$. By definition,
\begin{equation}\label{eqn:composition_of_positive_braids_lifting_permutations}
[w_1^{\vec{r}_1}]\circ[w_2^{\vec{r}_2}]=[w_1^{\vec{r}_1}w_2^{\vec{r}_2}]\in\Br_{n,N}^+ \Leftrightarrow \ell(w_1^{\vec{r}_1}w_2^{\vec{r}_2})=\ell(w_1^{\vec{r}_1})+\ell(w_2^{\vec{r}_2}).
\end{equation}
Also, for any color $\vec{r}:[n]\rightarrow\mathbb{N}$ of total rank $N$, let $P_{\vec{r}}\subset G$ denote the parabolic subgroup of $(\vec{r}(1),\text{\tiny$\cdots$},\vec{r}(n))$-block upper-triangular matrices, 
then $\exists$ a canonical \emph{morphism} respecting admissible compositions, extending Definition \ref{def:elementary_braid_matrix_diagrams}:
\begin{equation}\label{eqn:braid_matrix_diagrams_lifting_Borel_elements}
[-]=[-]^{\vec{r}}:P_{\vec{r}} \rightarrow \underline{\FBD}_{n,N}: b\mapsto [b]^{\vec{r}} \text{ with } g_{[b]^{\vec{r}}}=b.
\end{equation}

\subsection{Betti moduli stacks associated to positive braids}\label{subsec:Betti_moduli_stacks_associated_to_positive_braids}

Let $\beta^{\vec{r}}\in\FBr_{n,N}^+$ be a colored positive braid presentation of rank $(n,N)$. We may write
\[
\beta^{\vec{r}} = \sigma_{i_{\ell}}^{\vec{r}_{\ell}}\text{\tiny$\cdots$}\sigma_{i_1}^{\vec{r}_1}.
\]
That is, $\vec{r}_k$ is a color on $\sigma_{i_k}\in\FBr_n^+$ of total rank $N$, and $\vec{r}_{k,\Right}=\vec{r}_{k+1,\Left}$. Moreover, $\ell$ is the length of $\beta^{\vec{r}}$.

As in \cite[\S 6.1]{STZ17}, we may view $\beta$ as a Legendrian submanifold in $J^1U=T^{\infty,-}(U\times\mathbb{R}_z)$ for some interval $U=(x_{\Left},x_{\Right})$.
As in the beginning of Section \ref{sec:Betti_moduli_spaces_and_braids}, replacing $S_x^1$ by $U$, we may define a diagram of algebraic stacks
\[
\begin{tikzcd}
\Mloc_{N}(U\times\{z_0\}) &
\modulistack_{\vec{r}}(\beta):=\modulistack_{\vec{r}}(U\times\mathbb{R}_z,\{q_m\},\beta)\arrow[r,"{\mumon}"]\arrow[l,swap,"{\mon}"] 
& \Mloc_{\vec{r}}(\beta)
\end{tikzcd}
\]
For any fixed $0<\delta\ll1$, denote $U_{\Left}:=(x_{\Left},x_{\Left}+\delta), U_{\Right}:=(x_{\Right}-\delta,x_{\Right})$. Let $\beta_{\Left}^{\vec{r}_{\Left}}$ (resp. $\beta_{\Right}^{\vec{r}_{\Right}}$) denote the restriction of $\beta^{\vec{r}}$ over $U_{\Left}$ (resp. $U_{\Right}$). In particular, $\beta_{\Left}=\beta_{\Right}=\id_n\in\FBr_n^+$. By restriction, we obtain a commutative diagram of algebraic stacks
\begin{equation}\label{eqn:local_diagram_for_Betti_moduli_stacks}
\begin{tikzcd}[row sep=2pc,column sep=2pc]
\Mloc_{N}(U_{\Left}\times\{z_0\}) &
\Mloc_{N}(U\times\{z_0\})\arrow[r,"{\isomorphic}"]\arrow[l,swap,"{\isomorphic}"] 
& \Mloc_{N}(U_{\Right}\times\{z_0\})\\
\modulistack_{\vec{r}_{\Left}}(\beta_{\Left})\arrow[u,swap,"{\mon}"]\arrow[d,"{\mumon}"] &
\modulistack_{\vec{r}}(\beta)\arrow[u,swap,"{\mon}"]\arrow[d,"{\mumon}"]\arrow[l,swap,"{\fr_{\Left}}"]\arrow[r,"{\fr_{\Right}}"] 
& \modulistack_{\vec{r}_{\Right}}(\beta_{\Right})\arrow[u,swap,"{\mon}"]\arrow[d,"{\mumon}"]\\
\Mloc_{\vec{r}_{\Left}}(\beta_{\Left}) &
 \Mloc_{\vec{r}}(\beta)\arrow[r,"{\fr_{\Right}}"]\arrow[l,swap,"{\fr_{\Left}}"] 
& \Mloc_{\vec{r}_{\Right}}(\beta_{\Right})
\end{tikzcd}
\end{equation}
We would like to give a concrete description of this diagram. To do that, we make some definitions.

\begin{definition}\label{def:braid_matrices}
The \emph{(colored/block) braid matrix (resp. $\bmdp$) with coefficient $\epsilon\in M_{\vec{r}_{\Left}(k)\times\vec{r}_{\Left}(k+1)}(\field)$} associated to $\sigma_k^{\vec{r}}\in\FBr_{n,N}^+$ is
\begin{equation}
\bfB_k^{\vec{r}}(\epsilon):=\bfs_k^{\vec{r}}\bfH_k^{\vec{r}_{\Left}}(\epsilon)\in G;
\quad [\bfB_k^{\vec{r}}(\epsilon)]=\sigma_k^{\vec{r}}\circ[\bfH_k^{\vec{r}_{\Left}}(\epsilon)]\in\FBD_{n,N}~~(\text{Figure \ref{fig:braid_matrix_diagram_with_coeffficient_for_a_single_crossing}}).
\end{equation}
For $\beta^{\vec{r}}=\sigma_{i_{\ell}}^{\vec{r}_{\ell}}\cdots\sigma_{ki_1}^{\vec{r}_1}\in\FBr_{n,N}^+$, 
and $\vec{\epsilon}=(\epsilon_i)_{i={\ell}}^1\in\text{\tiny$\prod_{k=\ell}^1$}M_{\vec{r}_{k,\Left}(i_k)\times\vec{r}_{k,\Left}(i_k+1)}(\field)$, define
\begin{equation}
\bfB_{\beta}^{\vec{r}}(\vec{\epsilon}):=\bfB_{i_{\ell}}^{\vec{r}_{\ell}}(\epsilon_{\ell}) \text{\tiny$\cdots$} \bfB_{i_1}^{\vec{r}_1}(\epsilon_1) \in G;\quad
[\bfB_{\beta}^{\vec{r}}(\vec{\epsilon})]':=[\bfB_{i_{\ell}}^{\vec{r}_{\ell}}(\epsilon_{\ell})]\circ \text{\tiny$\cdots$} \circ [\bfB_{i_1}^{\vec{r}_1}(\epsilon_1)]\in\FBD_{n,N}.
\end{equation}
If $\vec{r}=1$ (uncolored case), simply denote $\mB_{\beta}(\vec{\epsilon}):=\bfB_{\beta}^{\vec{r}}(\vec{\epsilon})$.
\end{definition}

\begin{figure}[!htbp]
\vspace{-0.1in}
\begin{center}
\begin{tikzpicture}[baseline=-.5ex,scale=0.3]
\begin{scope}
\draw[thin] (0,0)--(4,0);
\draw[thin] (0,2)--(1,2) to [out=0,in=180] (3,3)--(4,3);
\draw[white,fill=white] (2,2.5) circle (0.2);
\draw[thin] (0,3)--(1,3) to[out=0,in=180] (3,2)--(4,2);
\draw[thin] (0,5)--(4,5);

\draw[line width=0.5mm] (1,2)--(1,3);
\draw (1.2,2.5) node[left] {\text{\tiny$\epsilon$}};

\draw (3.2,1.3) node[left] {\text{\tiny$\vdots$}};
\draw (3.2,4.3) node[left] {\text{\tiny$\vdots$}};

\draw (0,0) node[left] {\text{\tiny$1$}};
\draw (0,2) node[left] {\text{\tiny$k$}};
\draw (0,3) node[left] {\text{\tiny$k+1$}};
\draw (0,5) node[left] {\text{\tiny$n$}};

\draw (0.5,5) node[above] {\text{\tiny$\vec{r}_{\Left}$}};

\draw (4,0) node[right] {\text{\tiny$1$}};
\draw (4,2) node[right] {\text{\tiny$k$}};
\draw (4,3) node[right] {\text{\tiny$k+1$}};
\draw (4,5) node[right] {\text{\tiny$n$}};

\draw (3.5,5) node[above] {\text{\tiny$\vec{r}_{\Right}$}};

\draw (2,-0.5) node[below] {\text{\tiny$[\bfB_k^{\vec{r}}(\epsilon)]$}};

\end{scope}
\end{tikzpicture}
\vspace{-0.1in}
\end{center}
\caption{The (colored/block) braid matrix diagram with coefficient $\epsilon\in M_{\vec{r}_{\Left}(k)\times\vec{r}_{\Left}(k+1)}(\field)$ associated to $\sigma_k^{\vec{r}}$.}
\label{fig:braid_matrix_diagram_with_coeffficient_for_a_single_crossing}
\end{figure}
\vspace{-0.1in}

\begin{definition/proposition}
Given $\beta^{\vec{r}}=\sigma_{i_{\ell}}^{\vec{r}_{\ell}}\cdots\sigma_{i_1}^{\vec{r}_1}\in\FBr_{n,N}^+$ as above, denote $P_{\Left}:=P_{\vec{r}_{\Left}}$ and $P_{\Right}:=P_{\vec{r}_{\Right}}$.
For any map $\vec{t}:[n]\rightarrow\mathbb{N}$, denote $GL(\vec{t}):=\text{\tiny$\prod_{k=1}^n$}GL(\vec{t}(k))$.
\begin{enumerate}[wide,labelwidth=!,labelindent=0pt]
\item
Define a \emph{pre-braid variety} $\hat{X}_{\vec{r}}(\beta)$ with \emph{monodromy} and \emph{microlocal monodromy} as: 
\begin{equation}\label{eqn:pre-braid_varieties}
\begin{tikzcd}
G=GL(N) & \hat{X}_{\vec{r}}(\beta):=P_{\Right}\times \text{\tiny$\prod_{k=\ell}^1$} M_{\vec{r}_{k,\Left}(i_k)\times\vec{r}_{k,\Left}(i_k+1)}(\field)
\arrow[l,swap,"{\mon_0}"]\arrow[r,"{\mumon_0}"] &  GL(\vec{r}_{\Right})
\end{tikzcd}
\end{equation}
where for any $\epsilon=(\epsilon_{\Right},\vec{\epsilon})\in\hat{X}_{\vec{r}}(\beta)$, set
\[
\mon_0(\epsilon):=\epsilon_{\Right}\bfB_{\vec{\beta}}^{\vec{r}}(\vec{\epsilon})\in G;\quad \mumon_0(\epsilon):= D(\epsilon_{\Right})\in GL(\vec{r}_{\Right})~~(\text{block diagonal part})
\]

\item
$\exists$ an algebraic (left) action of $(g_{\Right},g_{\Left})\in P_{\Right}\times P_{\Left}$ on $\epsilon=(\epsilon_{\Right},\vec{\epsilon})\in\hat{X}_{\vec{r}}(\beta)$ defined as follows:
$\hat{\epsilon}=(\hat{\epsilon}_{\Right},\hat{\vec{\epsilon}}):=(g_{\Right},g_{\Left})\cdot \epsilon$ is uniquely determined by
\begin{equation}\label{eqn:action_on_pre-braid_varieties}
[g_{\Right}]^{\vec{r}_{\Right}}\circ[\epsilon_{\Right}]\circ [\bfB_{\beta}^{\vec{r}}(\epsilon)]'\circ[g_{\Left}^{-1}]^{\vec{r}_{\Left}} = [\hat{\epsilon}_{\Right}]\circ[\bfB_{\beta}(\hat{\vec{\epsilon}})]' \in\underline{\FBD}_{n,N}.
\end{equation}
\emph{Diagrammatically}, the equation means: the left hand side is the concatenation 
\[
[g_{\Left}^{-1}|\bfH_{i_1}^{\vec{r}_1}(\epsilon_1)|\sigma_{i_1}^{\vec{r}_1}|\text{\tiny$\cdots$}|\bfH_{i_{\ell}}^{\vec{r}_{\ell}}(\epsilon_{i_{\ell}})|\sigma_{i_{\ell}}^{\vec{r}_{\ell}}|\epsilon_{\Right}|g_{\Right}]\in \FBD_{n,N};
\]
 push $[g_{\Left}^{-1}]=[g_{\Left}^{-1}]^{\vec{r}_{\Left}}$ by elementary moves (Figure \ref{fig:elementary_moves}) to the right as far as possible, the outcome becomes
$[\bfH_{i_1}^{\vec{r}_1}(\hat{\epsilon}_1)|\sigma_{i_1}^{\vec{r}_1}|\text{\tiny$\cdots$}|\bfH_{i_{\ell}}^{\vec{r}_{\ell}}(\hat{\epsilon}_{i_{\ell}})|\sigma_{i_{\ell}}^{\vec{r}_{\ell}}|\hat{\epsilon}_{\Right}]\in \FBD_{n,N}$, 
which is the right hand side.

\item
The diagram (\ref{eqn:pre-braid_varieties}) is equivariant with respect to the obvious morphisms of linear algebraic groups
\[
\begin{tikzcd}[column sep=6pc]
G^2 & P_{\Right}\times P_{\Left}\arrow[l,hookrightarrow,swap,"{\mon_1}"]\arrow[r,"{\mumon_1:=D(-)\times D(-)}"] & GL(\vec{r}_{\Right})\times GL(\vec{r}_{\Left}),
\end{tikzcd}
\]
where $(g_{\Right},g_{\Left})\in G^2$ acts on $\epsilon_{\infty}\in G$ by $(g_{\Right},g_{\Left})\cdot \epsilon_{\infty}:=g_{\Right}\epsilon_{\infty}g_{\Left}^{-1}$, 
and $(\overline{g}_{\Right},\overline{g}_{\Left})\in GL(\vec{r}_{\Right})\times GL(\vec{r}_{\Left})$ acts on $\overline{\epsilon}_{\Right}\in GL(\vec{r}_{\Right})$ by
\begin{equation}
(\overline{g}_{\Right},\overline{g}_{\Left})\cdot\overline{\epsilon}_{\Right}:=\overline{g}_{\Right}\overline{\epsilon}_{\Right}(\overline{g}_{\Left}^{\bfs(\beta^{\vec{r}})})^{-1};
\quad \overline{g}_{\Left}^{\bfs(\beta^{\vec{r}})}:=\bfs(\beta^{\vec{r}})\overline{g}_{\Left}(\bfs(\beta^{\vec{r}}))^{-1}.
\end{equation}
As a consequence, this defines a \emph{pre-braid stack} with \emph{monodromy} and \emph{microlocal monodromy}:
\begin{equation}\label{eqn:pre-braid_stacks}
\begin{tikzcd}
{[G/G^2]} & {[\hat{X}_{\vec{r}}(\beta)/(P_{\Right}\times P_{\Left})]} \arrow[l,swap,"{\mon}"]\arrow[r,"{\mumon}"] &  {[GL(\vec{r}_{\Right})/(GL(\vec{r}_{\Right})\times GL(\vec{r}_{\Left}))]}.
\end{tikzcd}
\end{equation}
\end{enumerate}
\end{definition/proposition}

\begin{remark}\label{rem:braid_varieties}
The terminology is explained by the following: In the uncolored case ($\vec{r}=1$ and $n=N$), denote $\hat{X}(\beta):=\hat{X}_{\vec{r}=1}(\beta)=B\times\mathbb{A}^{\ell}$. 
Recall that $\mB_{\beta}(\vec{\epsilon})=\bfB_{\beta}^{\vec{r}=1}(\vec{\epsilon})$. Then,
the \emph{braid variety} (see \cite{Mel19}, \cite[\S 4.1]{GKS22}) associated to $\beta$ is
\begin{equation}
X(\beta):=\{(\epsilon_{\Right},\vec{\epsilon})\in\hat{X}(\beta): \epsilon_{\Right}\mB_{\beta}(\vec{\epsilon})=\id\}\hookrightarrow\mathbb{A}^{\ell}.
\end{equation}
\end{remark}

\begin{remark}\label{rem:braid_relations_for_braid_matrices}
Similar to \cite[Lem.1.13]{Su23}, some diagram calculus in $\BD_{n,N}$ shows the following \emph{braid relations for (colored/block) braid matrices or $\bmdp$'s}:
Fix a color $\vec{r}_{\Left}:[n]\rightarrow\mathbb{N}$ of total rank $N$. Then:
\begin{enumerate}[wide,labelwidth=!,labelindent=0pt]
\item
For any $|i-j|>1$ in $[n]$, and any $\epsilon_1\in M_{\vec{r}_{\Left}(i)\times \vec{r}_{\Left}(i+1)}(\field), \epsilon_2\in M_{\vec{r}_{\Left}(j)\times \vec{r}_{\Left}(j+1)}(\field)$, we have
\[
[\bfB_i^{\vec{r}}(\epsilon_1)]\circ [\bfB_j^{\vec{r}}(\epsilon_2)] = [\bfB_j^{\vec{r}}(\epsilon_2)]\circ [\bfB_i^{\vec{r}}(\epsilon_1)] \in\BD_{n,N}.
\]
Here, $\vec{r}_{\Left}$ extends to a color $\vec{r}$ on $\sigma_j\sigma_i=\sigma_i\sigma_j\in\mathrm{Br}_n^+$.

\item
For any $1\leq i\leq n-2$, and any $\epsilon_1\in M_{\vec{r}_{\Left}(i+1)\times \vec{r}_{\Left}(i+2)}(\field), \epsilon_2\in M_{\vec{r}_{\Left}(i)\times \vec{r}_{\Left}(i+2)}(\field), \epsilon_3\in M_{\vec{r}_{\Left}(i)\times \vec{r}_{\Left}(i+1)}(\field)$, we have
\[
[\bfB_i^{\vec{r}}(\epsilon_1)]\circ [\bfB_{i+1}^{\vec{r}}(\epsilon_2)]\circ [\bfB_i^{\vec{r}}(\epsilon_3)] = [\bfB_{i+1}^{\vec{r}}(\epsilon_3)]
\circ [\bfB_i^{\vec{r}}(\epsilon_2-\epsilon_3\epsilon_1)]\circ [\bfB_{i+1}^{\vec{r}}(\epsilon_1)] \in\BD_{n,N}.
\]
Here, $\vec{r}_{\Left}$ extends to a color $\vec{r}$ on $\sigma_i\sigma_{i+1}\sigma_i=\sigma_{i+1}\sigma_i\sigma_{i+1}\in\mathrm{Br}_n^+$.
\end{enumerate}
As a consequence, the diagram (\ref{eqn:pre-braid_stacks}), up to a canonical isomorphism, depends only on $\beta^{\vec{r}}\in\Br_{n,N}^+$.
\end{remark}

Now, we can describe the diagram (\ref{eqn:local_diagram_for_Betti_moduli_stacks}):
\begin{proposition}\label{prop:local_diagram_for_Betti_moduli_stacks}
There is a natural identification of diagrams of algebraic stacks
\[
\begin{tikzcd}[row sep=2pc,column sep=0.5pc]
\Mloc_N(U_{\Left}\times\{z_0\})\simeq[\pt/G] &\Mloc_{N}(U\times\{z_0\})\simeq[G/G^2]\arrow[l,swap,"{\simeq}"]\arrow[r,"{\simeq}"] & \Mloc_N(U_{\Right}\times\{z_0\})\simeq[\pt/G] \\
\modulistack_{\vec{r}_{\Left}}(\beta_{\Left})\simeq [\pt/P_{\Left}]\arrow[u]\arrow[d] &
\modulistack_{\vec{r}}(\beta)\simeq [\hat{X}_{\vec{r}}(\beta)/(P_{\Left}\times P_{\Right})]\arrow[u,swap,"{\mon}"]\arrow[d,"{\mumon}"]\arrow[l,swap,"{\fr_{\Left}}"]\arrow[r,"{\fr_{\Right}}"] 
& \modulistack_{\vec{r}_{\Right}}(\beta_{\Right})\simeq [\pt/P_{\Right}]\arrow[u]\arrow[d]\\
\Mloc_{\vec{r}_{\Left}}(\beta_{\Left})\simeq [\pt/GL(\vec{r}_{\Left})] &\Mloc_{\vec{r}}(\beta)\simeq [GL(\vec{r}_{\Right})/(GL(\vec{r}_{\Left})\times GL(\vec{r}_{\Right}))]\arrow[l,swap]\arrow[r] & \Mloc_{\vec{r}_{\Right}}(\beta_{\Right})\simeq[\pt/GL(\vec{r}_{\Right})] 
\end{tikzcd}
\]
Here, the arrows in the middle column are given by (\ref{eqn:pre-braid_stacks}), and all the other arrows are the obvious ones.
\end{proposition}

\begin{proof}
See Section \ref{sec:prove_local_diagram_for_Betti_moduli_stacks}.
\end{proof}

As an immediate application, we describe the Betti moduli stacks associated to cylindrical closures of positive braids.
Suppose that $\vec{r}_{\Left}=\vec{r}_{\Right}$, so $\vec{r}$ defines a color on $\beta^{\circ}$. 

\begin{proposition}\label{prop:Betti_moduli_stacks_for_cylindrical_braid_closures}
We have a natural identification of diagrams of algebraic stacks
\begin{equation}
\begin{tikzcd}[column sep=1.5pc]
\Mloc_{N}(U\times\{z_0\})\simeq[G/G] & \modulistack_{\vec{r}}(\beta^{\circ})\simeq [\hat{X}_{\vec{r}}(\beta)/P_{\Right}]\arrow[l,swap,"{\mon}"]\arrow[r,"{\mumon}"]
& \Mloc_{\vec{r}}(\beta)\simeq [GL(\vec{r}_{\Right})/^{\bfs(\beta^{\vec{r}})}GL(\vec{r}_{\Right})],
\end{tikzcd}
\end{equation}
where $P_{\Right}$ (resp. $GL(\vec{r}_{\Right})$) acts on $\hat{X}_{\vec{r}}(\beta)$ (resp. $GL(\vec{r}_{\Right})$) via the diagonal embedding 
$P_{\Right}\hookrightarrow P_{\Left}\times P_{\Right}$ (resp. $GL(\vec{r}_{\Right})\hookrightarrow GL(\vec{r}_{\Left})\times GL(\vec{r}_{\Right})$). 
\end{proposition}

\noindent{}\textbf{Note}: in $[GL(\vec{r}_{\Right})/^{\bfs(\beta^{\vec{r}})}GL(\vec{r}_{\Right})]$, $\overline{g}_{\Right}\in GL(\vec{r}_{\Right})\subset G$ acts on $\overline{\epsilon}_{\Right}\in GL(\vec{r}_{\Right})\subset G$ via $\bfs(\beta^{\vec{r}})$-twisted conjugation:
\[
\overline{g}_{\Right}\cdot\overline{\epsilon}_{\Right} := \overline{g}_{\Right}\epsilon_{\Right}(\overline{g}_{\Right}^{\bfs(\beta^{\vec{r}})})^{-1}.
\]

\begin{proof}
By the gluing property of Betti moduli stacks, we have a natural equivalence
\[
\modulistack_{\vec{r}}(\beta^{\circ}) \simeq \modulistack_{\vec{r}}(\beta)\times_{\modulistack_{\vec{r}_{\Left}}(\beta_{\Left})\times\modulistack_{\vec{r}_{\Right}}(\beta)} \modulistack_{\vec{r}_{\Right}}(\beta)
\]
compatible with monodromy and microlocal monodromy. By Proposition \ref{prop:local_diagram_for_Betti_moduli_stacks}, it follows that
\[
\modulistack_{\vec{r}}(\beta^{\circ}) \simeq [\hat{X}_{\vec{r}}(\beta)/(P_{\Right}\times P_{\Left})] \times_{[\pt/P_{\Right}]\times[\pt/P_{\Left}]} [\pt/P_{\Right}] \simeq [\hat{X}_{\vec{r}}(\beta)/P_{\Right}],
\]
compatible with monodromy and microlocal monodromy. Done.
\end{proof}

\subsection{Betti moduli stacks over punctured Riemann surfaces}\label{subsec:Betti_moduli_stacks_over_punctured_Riemann_surfaces}

We now go back to the global case. Recall that $(C,\{p_i\}_{i=1}^k,\{\Lambda_i\}_{i=1}^k)$ is a closed genus $g$ Riemann surface with Legendrian links $\Lambda_i\isomorphic \beta_i^{\circ}$ around the punctures $p_i$, and $\vec{r}:\pi_0(\Lambda=\sqcup_{i=1}^k\Lambda_i)\rightarrow\mathbb{N}$ is a map such that $\vec{r}_i:=\vec{r}|_{\pi_0(\Lambda_i)}$ is a color of total rank $N$ for each $i$.
So, $\beta_i^{\vec{r}_i}\in\FBr_{n_i,N}^+$ for some $n_i\geq 1$. Denote $P_{i,\Right}:=P_{\vec{r}_{i,\Right}}$ with unipotent radical $N_{i,\Right}$, and $\modulistack_{\vec{r}}:=\modulistack_{\vec{r}}(C,\{p_i\},\{\Lambda_i\})$.

\begin{theorem}\label{thm:Betti_moduli_stack_via_braids}
We have a natural identification of diagrams of algebraic stacks:
{\small\begin{equation}
\begin{tikzcd}
{\modulistack_{\vec{r}}\simeq [\{A_j,x_i\in G, \epsilon_i\in\hat{X}_{\vec{r}_i}(\beta_i): \text{\tiny$\prod_{j=1}^g$}(A_{2j-1},A_{2j}) \text{\tiny$\prod_{i=1}^k$}x_i\mon_i(\epsilon_i)x_i^{-1}=\id\}/(G\times\text{\tiny$\prod_{i=1}^k$}P_{i,\Right})]\arrow[d,"\mumon=\text{\tiny$\prod_i$}\mumon_i"]}\\
{\Mloc_{\vec{r}}(\Lambda) \simeq \text{\tiny$\prod_{i=1}^k$} [GL(\vec{r}_{i,\Right})/^{\bfs(\beta_i^{\vec{r}_i})}GL(\vec{r}_{i,\Right})]}
\end{tikzcd}
\end{equation}}
Here, $(h,h_{i,\Right})=(h,(h_{i,\Right})_{i=1}^R)\in G\times\prod_{i=1}^kP_{i,\Right}$ acts on $(A_j,x_i,\epsilon_i)=((A_j)_{j=1}^{2g},(x_i,\epsilon_i)_{i=1}^k)$ by
{\small\begin{equation}
(h,h_{i,\Right})\cdot (A_j,x_i,\epsilon_i):=(hA_jh^{-1},hx_ih_{i,\Right}^{-1},h_{i,\Right}\cdot\epsilon_i).
\end{equation}}
The microlocal monodromy $\mumon_i$ is given by:
{\small\begin{eqnarray}
&&(\mumon_i)_0(A_j,x_i,\epsilon_i):=(\mumon_i)_0(\epsilon_i)=D(\epsilon_{i,\Right})\in GL(\vec{r}_{i,\Right}),\\ 
&&(\mumon_i)_1(h,h_{i,\Right}):=(\mumon_i)_1(h_{i,\Right})=D(h_{i,\Right})\in GL(\vec{r}_{i,\Right}).\nonumber
\end{eqnarray}}
\end{theorem}

\begin{proof}
This is a direct consequence of the gluing property (\ref{eqn:gluing_property_for_global_Betti_moduli_stacks}) and Proposition \ref{prop:Betti_moduli_stacks_for_cylindrical_braid_closures}.
\end{proof}

Now, Let $L$ be any fixed local system of rank $\vec{r}$ on $\Lambda=\sqcup_{i=1}^k\Lambda_i$.
Let's describe the \emph{(restricted) Betti moduli stack} $\modulistack_L:=\modulistack_L(C,\{p_i\},\{\Lambda_i\})$.

\begin{enumerate}[wide,labelwidth=!,labelindent=0pt]
\item
For each $1\leq i\leq k$, the orbits of $s({\beta_i})\in S_{n_i}$ induces a partition 
\[
[n_i] = \sqcup_{v=1}^{c_i}O_{i,v}
\]
corresponding to the connected components of $\Lambda_i=\beta_i^{\circ}$. Let $\Lambda_i^v$ be the connected component determined by $O_{i,v}$. 
Denote by $o_{i,v}\in O_{i,v}$ the least element of $O_{i,v}$, and $r_{i,v}:=\vec{r}_i(\Lambda_i^v)=\vec{r}_i(o_{i,v})\in\mathbb{N}$. 

\item
Denote $L_i:=L|_{\Lambda_i}$ and $L_i^v:=L_i|_{\Lambda_i^v}$.
Taking the monodromy, each $L_i$ corresponds to an element $\vec{C}_{i,\Right}=(C_{i,\Right}^1,{\scriptstyle\cdots},C_{i,\Right}^{n_i})$ of $GL(\vec{r}_{i,\Right})\subset G$ 
up to $\bfs({\beta_i}^{\vec{r}_i})$-twisted conjugation. After such a conjugation, we can \textbf{assume} that
\begin{equation}\label{eqn:reduced_microlocal_monodromy}
C_{i,\Right}^j=\id,~\forall j\in O_{i,v}\setminus\{o_{i,v}\}.
\end{equation}
In this case, we say that $\vec{C}_{i,\Right}$ is \textbf{reduced}, and may use the identification
\begin{equation}\label{eqn:microlocal_monodromy_via_twisted_conjugacy_classes}
\vec{C}_{i,\Right}\isomorphic (C_{i,v})_{v=1}^{c_i}=:\vec{C}_i;\quad C_{i,v}:=C_{i,\Right}^{o_{i,v}}\in GL(\vec{r}_{i,\Right}(o_{i,v}))=GL(r_{i,v}).
\end{equation}
Let $Z(\vec{C}_{i,\Right})$ be the stablizer of $\vec{C}_{i,\Right}$ under the $\bfs({\beta_i}^{\vec{r}_i})$-twisted conjugation of $GL(\vec{r}_{i,\Right})$:
\begin{eqnarray}\label{eqn:reduced_stablizer}
Z(\vec{C}_{i,\Right})&:=&\{\vec{h}_{i,\Right}=(h_{i,\Right}^j)_{j=1}^{n_i}\in GL(\vec{r}_{i,\Right}) : \vec{h}_{i,\Right}\cdot^{\bfs(\beta_i^{\vec{r}_i})} \vec{C}_{i,\Right}=\vec{C}_{i,\Right}\},\\\nonumber
&=&\{\vec{h}_{i,\Right}=(h_{i,\Right}^j)_{j=1}^{n_i}\in GL(\vec{r}_{i,\Right}) : h_{i,\Right}^j=h_{i,\Right}^{o_{i,v}}\in Z(C_{i,v}), \forall j\in O_{i,v}\} \isomorphic \text{\tiny$\prod_{v=1}^{c_i}$} Z(C_{i,v}).\nonumber
\end{eqnarray}
Here, $Z(C_{i,v})$ is the centralizer of $C_{i,v}$ in $GL(r_{i,v})$.

\item
By definition and above, we then obtain a pullback diagram of algebraic stacks
\[
\begin{tikzcd}
{\modulistack_{L}(C,\{p_i\},\{\Lambda_i\})}\arrow[r]\arrow[d]\arrow[dr,phantom,"\lrcorner",very near start] & {\modulistack_{\vec{r}}(C,\{p_i\},\{\Lambda_i\})}\arrow[d,"\mumon"]\\
{[pt/Z(\vec{C}_{i,\Right})]}\arrow[r,"{[L]}"] & {\prod_{i=1}^k[GL(\vec{r}_{i,\Right})/^{\bfs(\beta_i^{\vec{r}_i})}GL(\vec{r}_{i,\Right})]}
\end{tikzcd}
\]
\end{enumerate}

As an immediate consequence of Theorem \ref{thm:Betti_moduli_stack_via_braids}, we now obtain
\begin{corollary}\label{cor:restricted_Betti_moduli_stack_via_braids}
We have a naturally identification of algebraic stacks
\begin{eqnarray}
&&\modulistack_L\simeq [(\mumon)_0^{-1}((\vec{C}_{i,\Right})_{i=1}^k)/(\mumon)_1^{-1}(\text{\tiny$\prod_{i=1}^k$}Z(\vec{C}_{i,\Right}))]\\\nonumber
&=&[\{A_j,x_i\in G,\epsilon_i\in\hat{X}_{\vec{r}_i}(\beta_i,\vec{C}_i):
\text{\tiny$\prod_{j=1}^g$} (A_{2j-1},A_{2j}) \text{\tiny$\prod_{i=1}^k$} x_i\mon_i(\epsilon_i)x_i^{-1}=\id\}/(G\times\text{\tiny$\prod_{i=1}^k$}P_{i,\Right}(\vec{C}_i))]\nonumber
\end{eqnarray}
Here, $\hat{X}_{\vec{r}_i}(\beta_i,\vec{C}_i):=\{\epsilon_i\in\hat{X}_{\vec{r}_i}(\beta_i) : D(\epsilon_{i,\Right})=\vec{C}_{i,\Right}\}$ (\emph{restricted pre-braid variety}), and 
$P_{i,\Right}(\vec{C}_i):=\{h_{i,\Right}\in P_{i,\Right} : D(h_{i,\Right})\in Z(\vec{C}_{i,\Right})\isomorphic \text{\tiny$\prod_{v=1}^{c_i}$} Z(C_{i,v})\} = Z(\vec{C}_{i,\Right})\ltimes N_{i,\Right}$.
\end{corollary}

\subsection{Betti moduli spaces over punctured Riemann surfaces}\label{subsec:Betti_moduli_spaces_over_punctured_Riemann_surfaces}

\subsubsection{Existence of Betti moduli spaces}
We firstly introduce a \emph{modified version of pre-braid varieties}.
Suppose $\beta^{\vec{r}}=\gamma^{\vec{r}}\Delta^{\vec{r}}\in\FBr_{n,N}^+$, where $\gamma\in\FBr_n^+$ and $\Delta=\Delta_n\in\FBr_n^+$ is a half-twist. So, $\permutation(\Delta)=w_0\in S_n$ is the longest element of length $\ell_0=\binom{n}{2}$.
By a bit abuse of notations, write $\beta^{\vec{r}}=\sigma_{i_{\ell}}^{\vec{r}_{\ell}}\text{\tiny$\cdots$}\sigma_{i_1}^{\vec{r}_1}$ as before. 
So, $\Delta^{\vec{r}}=\sigma_{i_{\ell_0}}^{\vec{r}_{\ell_0}}\text{\tiny$\cdots$}\sigma_{i_1}^{\vec{r}_1}$. Define
\begin{eqnarray*}
&&\hat{X}_{\vec{r}}(\gamma;w_0):=\{\epsilon=(\epsilon_{\Right},\vec{\epsilon}=(\epsilon_{\ell},\text{\tiny$\cdots$},\epsilon_1))\in\hat{X}_{\vec{r}}(\beta): \epsilon_j=0, \forall 1\leq j\leq \ell_0\};\\
&&\hat{X}_{\vec{r}}(\gamma,\vec{C}_{\Right};w_0):=\{\epsilon\in\hat{X}_{\vec{r}}(\gamma;w_0): D(\epsilon_{\Right})=\vec{C}_{\Right}\in GL(\vec{r}_{\Right})\}, \forall \vec{C}_{\Right}\in GL(\vec{r}_{\Right}).
\end{eqnarray*}
Let $N_{\Left}$ be the unipotent radical of $P_{\Left}=P_{\vec{r}_{\Left}}=P_{\vec{r}_{1,\Left}}$. We have a natural identification 
\[
\hat{X}_{\vec{r}}(\gamma\Delta)\isomorphic \hat{X}_{\vec{r}}(\gamma;w_0)\times N_{\Left}:
\epsilon=(\epsilon_{\Right},\vec{\epsilon}) \isomorphic (\epsilon'=(\epsilon_{\Right},\vec{\epsilon}'=(\epsilon_{\ell},\text{\tiny$\cdots$},\epsilon_{\ell_0+1},0^{\ell_0})),\epsilon_u)
\] 
via 
$[\epsilon_{\Right}]^{\vec{r}_{\Right}}\circ[\bfB_{\beta}^{\vec{r}}(\vec{\epsilon})]' = [\epsilon_{\Right}]^{\vec{r}_{\Right}}\circ [\bfB_{\beta}^{\vec{r}}(\vec{\epsilon}')]'\circ[\epsilon_u]^{\vec{r}_{\Left}}\in\underline{\FBD}_{n,N}$. 
Then, the action of $(g_{\Right},g_{\Left})\in P_{\Right}\times P_{\Left}$ on $\epsilon\isomorphic (\epsilon',\epsilon_u)\in \hat{X}_{\vec{r}}(\gamma\Delta)\isomorphic\hat{X}_{\vec{r}}(\gamma;w_0)\times N_{\Left}$ becomes:
\[
(g_{\Right},g_{\Left})\cdot (\epsilon',\epsilon_u):=((g_{\Right},D(g_{\Left}))\cdot\epsilon',D(g_{\Left})\epsilon_ug_{\Left}^{-1}).
\]
In particular, $\hat{X}_{\vec{r}}(\gamma\Delta)/N_{\Left}\isomorphic \hat{X}_{\vec{r}}(\gamma;w_0)$ is a geometric quotient.

For simplicity, denote
\begin{equation}
R_{\vec{r}}:= \{A_j,x_i\in G, \epsilon_i\in\hat{X}_{\vec{r}_i}(\gamma_i;s(\Delta_{n_i})):
\text{\tiny$\prod_{j=1}^g$} (A_{2j-1},A_{2j}) \text{\tiny$\prod_{i=1}^k$} x_i\mon_i(\epsilon_i)x_i^{-1}=\id\}.
\end{equation}
Taking the determinant in the defining equation, we obtain
$\text{\tiny$\prod_{i=1}^k$}\det(D(\epsilon_{i,\Right})\bfs(\beta_i^{\vec{r}_i})) =1$. Thus, we define
\begin{equation}
\mu:R_{\vec{r}}\rightarrow G_{\vec{r},1}:=\{(\vec{C}_{i,\Right})_{i=1}^k\in\text{\tiny$\prod_{i=1}^k$}GL(\vec{r}_{i,R}): \text{\tiny$\prod_{i=1}^k$}\det(\vec{C}_{i,R}\bfs(\beta_i^{\vec{r}_i})) =1\}: 
(A_j,x_i,\epsilon_i)\mapsto (D(\epsilon_{i,\Right}))_{i=1}^k.
\end{equation}
\textbf{From now on}, we always assume that $(\vec{C}_{i,\Right})_{i=1}^k$ is reduced and belongs to $G_{\vec{r},1}$, unless otherwise stated.

\begin{definition}
Let $\beta^{\vec{r}}\in\FBr_{n,N}^+$. We say $\vec{C}_{\Right}' \in GL(\vec{r}_{\Right})$ is \emph{$\bfs(\beta^{\vec{r}})$-twisted semisimple} if it's $\bfs(\beta^{\vec{r}})$-twisted conjugate to some reduced
$GL(\vec{r}_{\Right})\ni\vec{C}_{\Right}\isomorphic (C_v)_{v=1}^c=\vec{C}$ as in (\ref{eqn:microlocal_monodromy_via_twisted_conjugacy_classes}), with each $C_v\in GL(\vec{r}_{\Right}(o_v))$ semisimple.
\end{definition}

\begin{corollary}\label{cor:restricted_Betti_moduli_space_via_braids}
In Corollary \ref{cor:restricted_Betti_moduli_stack_via_braids}, if each $\beta_i=\gamma_i\Delta_{n_i}\in\FBr_{n_i}^+$ contains a half-twist $\Delta_{n_i}$, then 
\[
\modulistack_{\vec{r}}\simeq [R_{\vec{r}}/(G\times\text{\tiny$\prod_{i=1}^k$}GL(\vec{r}_{i,\Right}))];
\quad \modulistack_L\simeq[\mu^{-1}((\vec{C}_{i,\Right})_{i=1}^k)/(G\times\text{\tiny$\prod_{i=1}^k$}Z(\vec{C}_{i,\Right}))].
\]
If furthermore 
each $\vec{C}_{i,\Right}\isomorphic \vec{C}_i$ is ($\bfs(\beta_i^{\vec{r}_i})$-twisted) semisimple, then
$\modulispace_L=\modulispace_L(C,\{p_i\},\{\beta_i^{\circ}\})$ (Betti moduli space ) exists and is the affine GIT quotient\footnote{Here, we assume that $\field$ is an algebraically closed field of characteristic $0$. For simplicity, $\field=\mathbb{C}$.} 
$\mu^{-1}((\vec{C}_{i,\Right})_{i=1}^k)//(G\times\text{\tiny$\prod_{i=1}^k$}Z(\vec{C}_{i,\Right}))$. That is,
\[
\modulispace_L\isomorphic \{A_j,x_i\in G, \epsilon_i\in\hat{X}_{\vec{r}_i}(\gamma_i,\vec{C}_{i,\Right};s(\Delta_{n_i})):
\text{\tiny$\prod_{j=1}^g$} (A_{2j-1},A_{2j}) \text{\tiny$\prod_{i=1}^k$} x_i\mon_i(\epsilon_i)x_i^{-1}=\id\}//(G\times\text{\tiny$\prod_{i=1}^k$}Z(\vec{C}_{i,\Right})).
\]
\end{corollary}

\begin{proof}
By the previous discussion, 
\begin{eqnarray*}
&&\{A_j,x_i\in G, \epsilon_i\in\hat{X}_{\vec{r}_i}(\beta_i,\vec{C}_{i,\Right}):
\text{\tiny$\prod_{j=1}^g$} (A_{2j-1},A_{2j}) \text{\tiny$\prod_{i=1}^k$} x_i\mon_i(\epsilon_i)x_i^{-1}=\id\}/\text{\tiny$\prod_{i=1}^k$}N_{i,\Right}\\
&\isomorphic& \{A_j,x_i\in G, \epsilon_i\in\hat{X}_{\vec{r}_i}(\gamma_i,\vec{C}_{i,\Right};s(\Delta_{n_i})):
\text{\tiny$\prod_{j=1}^g$} (A_{2j-1},A_{2j}) \text{\tiny$\prod_{i=1}^k$} x_i\mon_i(\epsilon_i)x_i^{-1}=\id\}
\end{eqnarray*}
is a geometric quotient. Then, the first statement follows from Theorem \ref{thm:Betti_moduli_stack_via_braids} and Corollary \ref{cor:restricted_Betti_moduli_stack_via_braids}. 
If furthermore each $\vec{C}_{i,\Right}$ is semisimple, then $Z(\vec{C}_{i,\Right})$ is reductive. Now, by Lemma \ref{lem:example_of_good_moduli_space}.(2), 
the good moduli space $\modulispace_L$ for $\modulistack_L$ exists and is given by the desired affine GIT quotient.
\end{proof}

\begin{remark}
The Betti moduli spaces $\modulispace_L$ give a microlocal reformulation and generalization of all known character varieties over (punctured) Riemann surfaces.
For example, if each $\beta_i=\id\in\Br_1^+$ and $\vec{r}_i=N$, then $\modulispace_L$ recovers the tame $GL(N)$-character varieties studied in \cite{HLRV11}.
If each $\beta_i\in\Br_{n_i}^+$ is of some special form (morally, iterated cablings of even powers $\Delta^{2m}$ of full twists in the unramified/untwisted case, resp. iterated cablings of `$\Delta^{\frac{2m}{d}}$' in the ramified/twisted case), then $\modulispace_L$ recovers the untwisted/unramified (resp. twisted/ramified) wild character varieties in \cite{Boa07,Boa14} (resp. \cite{BY15}). 
\end{remark}

\subsubsection{Smoothness and connectedness}

\textbf{From now on}, we will always assume the following: 
\begin{itemize}[wide,labelwidth=!,labelindent=0pt]
\item
For each $1\leq i\leq k$, $\beta_i^{\vec{r}_i}=\gamma_i^{\vec{r}_i}\Delta_{n_i}^{\vec{r}_i}\in\FBr_{n_i,N}^+$.
\end{itemize}

\begin{lemma}\label{lem:smoothness_of_representation_varieties}
If there exists some $1\leq i\leq k$, say, $i=k$, such that $\beta_k=\Delta_{n_k}\beta_k'\Delta_{n_k}\in\FBr_{n_k}^+$, then $R_{\vec{r}}$ is smooth affine and connected of dimension
$n^2(2g+k-1) + \text{\tiny$\sum_{i=1}^k$}\dim \hat{X}_{\vec{r}_i}(\gamma_i;\permutation(\Delta_{n_i}))$.
\end{lemma}

\begin{proof}
Denote $P:=P_{\vec{r}_{k,R}}$ and $N_P=N(P)$ is the unipotent radical. 
Let $W=S_n$ be the Weyl group of $G=GL(N)$ and $W_P:=W(GL(\vec{r}_{k,\Right}))$ be the Weyl group of the Levi subgroup $GL(\vec{r}_{k,\Right})$ of $P$.
Denote $w_{0,P}:=\bfs(\Delta_{n_k}^{\vec{r}_k})\in S_N$. Observe that $w_{0,P}$ is the shortest representative of $w_0$ in $W_p\backslash W$.
We use the identification 
\[
\hat{X}_{\vec{r}_k}(\gamma_k;\Delta_{n_k})\isomorphic P\times N_P\times \mathbb{A}^{\dim\hat{X}_{\vec{r}_k}(\gamma_k;\Delta_N)-n^2}:
\epsilon_k=(\epsilon_{k,\Right},\vec{\epsilon}_k=(\epsilon_{k,\ell(\beta_k)},\text{\tiny$\cdots$},\epsilon_{k,1}))\mapsto (\epsilon_{k,\Right},u,\epsilon_{k,\leq \ell(\beta_k'\Delta_{n_k})}),
\] 
where $\epsilon_{k,\leq \ell(\beta_k'\Delta_{n_k})}:=(\epsilon_{k,\ell(\beta_k'\Delta_{n_k})},\text{\tiny$\cdots$},\epsilon_{k,1})$ and $u\in N_P$ is defined by
$\bfB_{\Delta_{n_k}}^{\vec{r}_k}(\epsilon_{k,\ell(\beta_k)},\text{\tiny$\cdots$},\epsilon_{k,\ell(\beta_k'\Delta_{n_k})+1}) = w_{0,P}u$.
Then,
the defining equation for $R_{\vec{r}}$ may be written as 
\[
\text{\tiny$\prod_{j=1}^g$} (A_{2j-1},A_{2j}) \text{\tiny$\prod_{i=1}^{k-1}$} (x_i\mon_i(\epsilon_i)x_i^{-1}) x_k\epsilon_{k,\Right}w_{0,P}u \bfB_{\beta_k'\Delta_N}^{\vec{r}_k}(\epsilon_{k,\leq\ell(\beta_k'\Delta_N)})x_k^{-1}=\id.
\]
Equivalently, $Pw_{0,P}N_P\ni \epsilon_{k,\Right}w_{0,P}u = f((A_j)_{j=1}^{2g},(x_i)_{i=1}^k,(\epsilon_i)_{i=1}^{k-1},\epsilon_{k,\leq\ell(\beta_k'\Delta_N)})$, 
where $f=(f_{a,b})_{1\leq a,b\leq N}:Y:=G^{2g+k}\times \text{\tiny$\prod_{i=1}^{k-1}$}\hat{X}_{\vec{r}_i}(\gamma_i;\Delta_{n_i})\times \mathbb{A}^{\dim\hat{X}_{\vec{r}_k}(\gamma_k;\Delta_N)-n^2} \rightarrow G$ is 
an algebraic map.
Observe that $P\times N_P\isomorphic Pw_{0,P}N_P=Pw_{0,P}B \subset G$ is the maximal cell in the Bruhat cell decomposition $G=\sqcup_{[w]\in W_P\backslash W} P\dot{w}B$, where $\dot{w}$ denotes the shortest representative in $W$ of $[w]\in W_P\backslash W$.
We can solve for $\epsilon_{k.\Right}, u$, then $R_{\vec{r}}\isomorphic\{y\in Y: f(y)\in Pw_{0,P}N_P=Pw_{0,P}B\}$. 
By Gaussian elimination, observe that 
\[
Pw_{0,P}N_P=Pw_{0,P}B = \{g=(g_{a,b})_{1\leq a,b\leq N}\in G: Q_{m_j}(g)\neq 0, \forall 1\leq j\leq n_k\},
\] 
where $m_j:=\text{\tiny$\sum_{q=n_k-j+1}^{n_k}$}\vec{r}_{k,\Right}(q)$ and
$Q_m(g):=\det(g_{a,b})_{N-m+1\leq a\leq N,1\leq b\leq m}$.
Thus,
$R_{\vec{r}}\isomorphic\{y\in Y: Q_{m_j}(f(y))\neq 0, \forall 1\leq j\leq n_k\}$ is isomorphic to an open affine subvariety of $Y$. The rest is clear.
\end{proof}

\subsubsection{Free action}

\begin{definition}[Generic condition]\label{def:generic_condition}
Let $(C,\{p_i\},\{\Lambda_i\})$ and $\vec{r}:\pi_0(\Lambda=\text{\tiny$\sqcup_i$}\Lambda_i)\rightarrow \mathbb{N}$ as above.
\begin{enumerate}[wide,labelwidth=!,labelindent=0pt]
\item
Given any color $\vec{r}'=(\vec{r}_i')_{i=1}^k$ of total rank $N'$ on $\Lambda=\sqcup\Lambda_i$, i.e. $\vec{r}_i':\pi_0(\Lambda_i=\beta_i^{\circ})\rightarrow\mathbb{N}$ is a color of total rank $N'$, we say that $\vec{r}'\leq \vec{r}$ if, 
$\vec{r}_i'\leq\vec{r}_i$ for all $1\leq i\leq k$.

\item
Any reduced $(\vec{C}_{i,\Right})_{i=1}^k\in G_{\vec{r},1}$ is called \emph{generic} if, for any $0<N'<N$, any color $\vec{r}'=(\vec{r}_i')_{i=1}^k<\vec{r}=(\vec{r}_i)_{i=1}^k$ of total rank $N'$, and any choice of 
$r_{i,v}':=\vec{r}_{i,\Right}'(o_{i,v})$ eigenvalues $\alpha_{i,v}^1,\text{\tiny$\cdots$},\alpha_{i,v}^{\vec{r}_{i,v}'}$ of $C_{i,v}$ for each $1\leq v\leq c_i,1\leq i\leq k$, we have
\begin{equation}
\text{\tiny$\prod_{i=1}^k\prod_{v=1}^{c_i}\prod_{j=1}^{\vec{r}_{i,v}'}$}\alpha_{i,v}^j \neq \text{\tiny$\prod_{i=1}^k$} \det \bfs(\beta_i^{\vec{r}_i'}).
\end{equation}
More generally, $(\vec{C}_{i,\Right}')_{i=1}^k\in G_{\vec{r},1}$ is \emph{generic} if it's $(\bfs(\beta_i^{\vec{r}_i}))_{i=1}^k$-twisted conjugate to a reduced and generic $(\vec{C}_{i,\Right})_{i=1}^k\in G_{\vec{r},1}$.
\end{enumerate}
\end{definition}

\begin{lemma}\label{lem:generic_condition_implies_free_action}
In the setting of Corollary \ref{cor:restricted_Betti_moduli_space_via_braids}:
\begin{enumerate}[wide,labelwidth=!,labelindent=0pt]
\item
If $(\vec{C}_{i,\Right})_{i=1}^k\in G_{\vec{r},1}$ is reduced and generic, then $PG_{\pa}:=(G\times\text{\tiny$\prod_{i=1}^k$}Z(\vec{C}_{i,\Right}))/\field^{\times}$ acts freely on $\mu^{-1}((\vec{C}_{i,\Right})_{i=1}^k)$.

\item
If furthermore $(\vec{C}_{i,\Right})_{i=1}^k$ is semisimple and  a regular value of $\mu$, and $\beta_k=\Delta_{n_k}\beta_k'\Delta_{n_k}\in\FBr_{n_k}^+$, then
\[
\modulispace_L \isomorphic \mu^{-1}((\vec{C}_{i,\Right})_{i=1}^k)/PG_{\pa}
\]
is smooth affine and of pure dimension.
\end{enumerate}
\end{lemma}

\begin{proof}

\noindent{}(1).
Fix any $(A_j,x_i,\epsilon_i)\in R_{\vec{r}}$.
If $(g,(\vec{h}_{i,\Right})_{i=1}^k)\in G\times \text{\tiny$\prod_{i=1}^k$}Z(\vec{C}_{i,\Right})$ satisfies
\[
(g,\vec{h}_{i,\Right})\cdot (A_j,x_i,\epsilon_i) = (gA_jg^{-1},gx_i\vec{h}_{i,\Right}^{-1},\vec{h}_{i,\Right}\cdot\epsilon_i) = (A_j,x_i,\epsilon_i), 
\]
we would like to show that for some $\lambda\in\field^{\times}$, we have $g=\vec{h}_{i,\Right}\in \lambda I_N, \forall 1\leq i\leq k$.

Denote $V_{\lambda}:=\ker(g-\lambda I_N)\subset \field^N, \forall \lambda\in\field^{\times}$. As $gx_i=x_i\vec{h}_{i,\Right}$, we see that for any $0\neq v\in\field^N$, we have
\[
v\in\ker(\vec{h}_{i,\Right}-\lambda I_N) \Leftrightarrow x_iv\in V_{\lambda}.
\]
That is, $\ker(\vec{h}_{i,\Right}-\lambda I_N)=x_i^{-1}V_{\lambda}$. Now, fix any $\lambda\in\field^{\times}$ such that $V:=V_{\lambda}\neq 0$.
Observe that $gA_j=A_jg$ and $g(x_i\mon_i(\epsilon_i)x_i^{-1})=(x_i\mon_i(\epsilon_i)x_i^{-1})g$, so $A_j$ and $x_i\mon_i(\epsilon_i)x_i^{-1}$ preserve $V$. 
By the defining equation of $R_{\vec{r}}$, we then have
\[
\text{\tiny$\prod_{j=1}^g$}(A_{2j-1}|_V,A_{2j}|_V)\text{\tiny$\prod_{i=1}^k$} (x_i\mon_i(\epsilon_i)x_i^{-1})|_V = \id_V.
\]
Taking the determinant, this implies that $\text{\tiny$\prod_{i=1}^k$} \det (x_i\mon_i(\epsilon_i)x_i^{-1})|_V=1$.
Equivalently, $\text{\tiny$\prod_{i=1}^k$} \det (\mon_i(\epsilon_i)|_{V_i})=1$, where $V_i:=x_i^{-1}V=\ker(\vec{h}_{i,\Right}-\lambda I_N)$. 
We may write $V_i=\oplus_{j=1}^{n_i}V_i^j$ with $V_i^j:=\ker(h_{i,\Right}^j-\lambda I_{\vec{r}_{i,\Right}(j)}) \hookrightarrow \field^N$.
Recall that $\epsilon_i=(\epsilon_{i,\Right},\vec{\epsilon}_i=(\epsilon_{i,\ell(\beta_i)},\text{\tiny$\cdots$},\epsilon_{i,1}))$ 
and $\mon_i(\epsilon_i)=\epsilon_{i,\Right}\bfB_{\beta_i}^{\vec{r}_i}(\vec{\epsilon}_i)$. As $\vec{h}_{i,\Right}\cdot\epsilon_i = \epsilon_i$, by Lemma \ref{lem:stablizers_for_pre-braid_stacks} below, we have
\[
1= \text{\tiny$\prod_{i=1}^k$} \det (\mon_i(\epsilon_i)|_{V_i}) = \text{\tiny$\prod_{i=1}^k\prod_{v=1}^{c_i}$}\det(C_{i,v}|_{V_i^v}) \cdot \text{\tiny$\prod_{i=1}^k$}\det((\bfs(\beta_i^{\vec{r}_i}))|_{V_i})),
\]
By the generic assumption, we must have $V_i=V=\field^N$, and hence $g=\vec{h}_{i,\Right}=\lambda I_N$, as desired.

\noindent{}(2). This is a direct consequence of (1) and Lemma \ref{lem:smoothness_of_representation_varieties}.
\end{proof}

\begin{lemma}\label{lem:stablizers_for_pre-braid_stacks}
Let $\beta^{\vec{t}}=\sigma_{i_{\ell}}^{\vec{t}_{\ell}}\text{\tiny$\cdots$}\sigma_{i_1}^{\vec{t}_1}\in\FBr_{n,N}^+$ with $\vec{t}_{\Left}=\vec{t}_{\Right}$. 
Suppose $\vec{h}_{\Right} = (h_{\Right}^j)_{j=1}^n\in GL(\vec{t}_R)=\text{\tiny$\prod_{j=1}^n$} GL(\vec{t}_{\Right}(j))$ and $\epsilon\in\hat{X}_{\vec{t}}(\beta)$ such that $\vec{h}_{\Right}\cdot\epsilon=\epsilon$. 
Let $W=\ker(\vec{h}_{\Right}-\lambda I_N)\subset\field^N$ be any eigenspace of $\vec{h}_{\Right}$, we may write $W=\oplus_{j=1}^n W^j$ with $W^j:=\ker(h_{\Right}^j-\lambda I_{\vec{t}_{\Right}(j)}) \hookrightarrow \field^N$.
Then,
\begin{enumerate}[wide,labelwidth=!,labelindent=0pt]
\item
$W$ is stable under $D(\epsilon_{\Right})\bfs(\beta^{\vec{t}})$ and
\[
\det(\mon(\epsilon)|_W) = \det((D(\epsilon_{\Right})\bfs(\beta^{\vec{t}}))|_W) = \det [\bfs(\beta^{\vec{t}})D(\epsilon_{\Right}))|_{\bfs(\beta^{\vec{t}})W}].
\]

\item
If furthermore $\vec{C}_{\Right}:=D(\epsilon_{\Right})\isomorphic (C_v)_{v=1}^c =:\vec{C}$ is reduced as in (\ref{eqn:reduced_microlocal_monodromy}), then:
$\bfs(\beta^{\vec{t}})$ acts by cyclic rotation on $\oplus_{j\in O_v}W^j$, with $W^{\permutation(\beta)(j)} = \bfs(\beta^{\vec{t}})W^j\hookrightarrow \field^N$; 
$\vec{C}_{\Right}$ preserves each $W^j\hookrightarrow \field^N$ with $\vec{C}_{\Right}=\id$ on $W^j$ if $j\neq o_v$; and
\begin{equation}
\det(\mon(\epsilon)|_W) = \det(\vec{C}_{\Right}|_W) \det(\bfs(\beta^{\vec{t}})|_W) = \text{\tiny$\prod_{v=1}^c$} \det(C_v|_{W^{o_v}}) \det\bfs(\beta^{\vec{t}'}),
\end{equation}
where $\vec{t}'$ is determined by $\vec{t}_{\Right}'(j):=\dim W^j$. In particular, $\det(\bfs(\beta^{\vec{t}'})) = (-1)^{\text{\tiny$\sum_{v=1}^c$}(|O_v|-1)\vec{t}_{\Right}'(v)}$.
\end{enumerate}
\end{lemma}

\begin{proof}

\noindent{}(1). 
Denote $h_p:=h_R^{\bfs_{i_p}^{\vec{t}_p}\text{\tiny$\cdots$}\bfs_{i_1}^{\vec{t}_1}} 
= \bfs_{i_p}^{\vec{t}_p}\text{\tiny$\cdots$}\bfs_{i_1}^{\vec{t}_1}h_R(\bfs_{i_p}^{\vec{t}_p}\text{\tiny$\cdots$}\bfs_{i_1}^{\vec{t}_1})^{-1}$ and $W_p:=\bfs_{i_p}^{\vec{t}_p}\text{\tiny$\cdots$}\bfs_{i_1}^{\vec{t}_1}W$.
Then $W_p=\ker(h_p-\lambda I_N)$.
Recall that $\mon(\epsilon)=\epsilon_{\Right}\bfB_{i_{\ell}}^{\vec{t}_{\ell}}(\epsilon_{\ell})\text{\tiny$\cdots$}\bfB_{i_1}^{\vec{t}_1}(\epsilon_1)$ and 
$\bfB_{i_p}^{\vec{t}_p}(\epsilon_p)=\bfs_{i_p}^{\vec{t}_p}\bfH_{i_p}^{\vec{t}_p}(\epsilon_p)$. 
We may write $\epsilon_{\Right}=D(\epsilon_{\Right})u_{\Right}$, $u_{\Right}\in N_{\vec{t}_R}=N(P_{\vec{t}_R})$.
By definition, $\epsilon=h_{\Right}\cdot\epsilon$ means the following equality in $\underline{\FBD}_{n,N}$:
\begin{eqnarray*}
&&[\epsilon_{\Right}]^{\vec{t}_{\ell,\Right}} \circ [\bfB_{i_{\ell}}^{\vec{t}_{\ell}}(\epsilon_{\ell})] \circ \text{\tiny$\cdots$}\circ [\bfB_{i_1}^{\vec{t}_1}(\epsilon_1)]
=[h_{\Right}]^{\vec{t}_{\ell,\Right}}\circ [\epsilon_{\Right}]^{\vec{t}_{\ell,\Right}} \circ [\bfB_{i_{\ell}}^{\vec{t}_{\ell}}(\epsilon_{\ell})] \circ \text{\tiny$\cdots$} \circ [\bfB_{i_1}^{\vec{t}_1}(\epsilon_1)] \circ [h_{\Right}^{-1}]^{\vec{t}_{1,\Left}}\\
&=&[h_{\Right}\epsilon_{\Right}(h_{\Right}^{\bfs(\beta^{\vec{t}})})^{-1}]^{\vec{t}_{\ell,\Right}} \circ \sigma_{i_{\ell}}^{\vec{t}_{\ell}}\circ [h_{\ell-1}\bfH_{i_{\ell}}^{\vec{t}_{\ell}}(\epsilon_{\ell})h_{\ell-1}^{-1}]^{\vec{t}_{\ell,\Left}} 
\circ \text{\tiny$\cdots$}\circ \sigma_{i_1}^{\vec{t}_1}\circ [h_0\bfH_{i_1}^{\vec{t}_{1,\Left}}(\epsilon_1)h_0^{-1}]^{\vec{t}_{1,\Left}}.
\end{eqnarray*}
Equivalently, $\bfH_{i_p}^{\vec{t}_{p,\Left}}(\epsilon_p) = h_{p-1}\bfH_{i_p}^{\vec{t}_{p,\Left}}(\epsilon_p)h_{p-1}^{-1}$, $N(\epsilon_{\Right}) = h_{\ell}N(\epsilon_{\Right})h_{\ell}^{-1}$, and
$D(\epsilon_{\Right}) = h_{\Right}D(\epsilon_{\Right})h_{\ell}^{-1}$.
In particular, $D(\epsilon_R)\bfs(\beta^{\vec{t}})$ commutes with $h_{\Right}$, hence preserves $W$.

As $\bfH_{i_1}^{\vec{t}_{1,\Left}}(\epsilon_1) = h_{\Right}\bfH_{i_1}^{\vec{t}_{1,\Left}}(\epsilon_1)h_{\Right}^{-1}$, the unipotent matrix $\bfH_{i_1}^{\vec{t}_{1,\Left}}(\epsilon_1)$ preserves $W$. Then, we have
\begin{eqnarray*}
\det(\mon(\epsilon)|_W) 
&=& \det [(\epsilon_{\Right}\bfB_{i_{\ell}}^{\vec{t}_{\ell}}(\epsilon_{\ell})\text{\tiny$\cdots$}\bfB_{i_2}^{\vec{t}_2}(\epsilon_2)\bfs_{i_1}^{\vec{t}_1})|_W] \cdot \det (\bfH_{i_1}^{\vec{t}_{1,\Left}}(\epsilon_1)|_W)
= \det [(\epsilon_{\Right}\bfB_{i_{\ell}}^{\vec{t}_{\ell}}(\epsilon_{\ell})\text{\tiny$\cdots$}\bfB_{i_2}^{\vec{t}_2}(\epsilon_2)\bfs_{i_1}^{\vec{t}_1})|_W]\\
&=& \det [(\bfs_{i_1}^{\vec{t}_1}\epsilon_{\Right}\bfB_{i_{\ell}}^{\vec{t}_{\ell}}(\epsilon_{\ell})\text{\tiny$\cdots$}\bfB_{i_2}^{\vec{t}_2}(\epsilon_2))|_{W_1}].
\end{eqnarray*}
Now, we repeat the argument inductively to obtain
\[
\det(\mon(\epsilon)|_W)  =  \det [(\bfs_{i_1}^{\vec{t}_1}\epsilon_{\Right}\bfB_{i_{\ell}}^{\vec{t}_{\ell}}(\epsilon_{\ell})\text{\tiny$\cdots$}\bfB_{i_2}^{\vec{t}_2}(\epsilon_2))|_{W_1}] 
= \text{\tiny$\cdots$} = \det[(\bfs(\beta^{\vec{t}})\epsilon_{\Right})|_{W_{\ell}}].
\]
Finally, as $N(\epsilon_{\Right}) = h_{\ell}N(\epsilon_{\Right})h_{\ell}^{-1}$, the unipotent matrix $N(\epsilon_{\Right})$ preserves $W_{\ell}$. Then, we have
\[
\det(\mon(\epsilon)|_W)  =  \det [(\bfs(\beta^{\vec{t}})D(\epsilon_{\Right}))|_{W_{\ell}}] \cdot \det (N(\epsilon_{\Right})|_{W_{\ell}})
=  \det [(\bfs(\beta^{\vec{t}})D(\epsilon_{\Right}))|_{\bfs(\beta^{\vec{t}})W}].
\]
The rest is clear.

\noindent{}(2). Suppose $\vec{C}_{\Right}:=D(\epsilon_{\Right})\isomorphic (C_v)_{v=1}^c =\vec{C}$ is reduced. Recall by (\ref{eqn:reduced_microlocal_monodromy}), this means that $C_{\Right}^j=\id$ for all $j\in O_v\setminus\{o_v\}$ and $C_{\Right}^{o_v}=C_v\in GL(\vec{t}_{\Right}(o_v))$. In particular, $C_{\Right}^j=\id$ on $W^j\hookrightarrow \field^N$ for all $j\in O_v\setminus\{o_v\}$. 
$\vec{h}_{\Right}=(h_{\Right}^j)_{j=1}^n \in Z(\vec{C}_{\Right})$ as in (\ref{eqn:reduced_stablizer}): $h_{\Right}^j=h_{\Right}^{v}, \forall j\in O_v$, $h_{\Right}^v\in Z(C_v)\subset GL(\vec{t}_{\Right}(v))$.
In particular, $C_v$ preserves $W^{o_v}$.
Observe that $\bfs(\beta^{\vec{t}})(h_{\Right}^j)_{j=1}^n\bfs(\beta^{\vec{t}})^{-1} = (h_{\Right}^{\permutation(\beta)(j)})_{j=1}^n$, it follows that $\bfs(\beta^{\vec{t}})$ acts by cyclic rotation on $\oplus_{j\in O_v}W^j$ for each $v$, in particular preserves $W$. The rest is clear.
\end{proof}

\section{Cell decomposition of Betti moduli spaces over $(\mathbb{C}P^1,\infty)$}\label{sec:Betti_moduli_spaces_over_punctured_P1_and_braid_varieties}

Now, we focus on the case $(C,\{p_i\},\{\Lambda_i\})=(\mathbb{C}P^1,\{\infty\},\{\beta_1^{\circ}\})$ with $\beta_1=\Delta_n\beta\Delta_n\in\FBr_n^+$ and $\vec{r}=\vec{r}_1=1:\pi_0(\beta_1^{\circ})\rightarrow\mathbb{N}$.
So, the total rank of $\vec{r}$ is $N=n$, and $G=GL(N)=GL(n)$.
Recall that  we identify $\beta_{1,\Left}=\beta_{1,\Right}$ with $[n]$, and the orbits of $\permutation(\beta_1)\in S_n$ induces a partition
$
[n] = \sqcup_{v=1}^cO_v,
$
corresponding to the connected components of $\beta_1^{\circ}$. So, $c$ is the number of connected components. $o_v\in O_v$ denotes the least element of $O_v$.

Denote 
\[
\modulistack_1=\modulistack_1(\beta_1^{\circ}):=\modulistack_1(\mathbb{C}P^1,\{\infty\},\{\beta_1^{\circ}\});
\quad \modulistack_L=\modulistack_L(\beta_1^{\circ}):=\modulistack_L(\mathbb{C}P^1,\{\infty\},\{\beta_1^{\circ}\}).
\]
For $*=1,L$, let $\modulispace_*$ denote the good moduli space associated to $\modulistack_*$. 
Now, let $L$ correspond to an $\permutation(\beta_1)$-twisted conjugacy class of $\vec{C}_\Right$ in $GL(\vec{r}_{\Right})=T$.
Recall that 
\[
G_{\vec{r},1} = G_{1,1} = \{\vec{C}_{\Right}\in GL(\vec{r}_{\Right})=T: \det \vec{C}_{\Right} = \det \permutation(\beta_1) = (-1)^{n-c}\}.
\]
We \textbf{assume} that $G_{1,1}\ni \vec{C}_{\Right}\isomorphic (C_v)_{v=1}^c=:\vec{C}$ is \emph{reduced} and \emph{generic}:
$\vec{C}_{\Right} = (C_{\Right}^j)_{j=1}^n$ with
$C_{\Right}^j=1, \forall j\in [n]\setminus\{o_v:1\leq v\leq c\}$, and $C_{\Right}^{o_v}=C_v\in\field^{\times}$;
$\text{\tiny$\prod_{v=1}^c$}C_v = (-1)^{n-c}$; and for any nonempty proper subset $I\subset[c]$, we have
\[
\text{\tiny$\prod_{v\in I}$}C_v \neq (-1)^{\text{\tiny$\sum_{v\in I}$}(|O_v|-1)}.
\]

For any $w\in S_n$ and $\gamma\in\FBr_n^+$, define the \emph{modified braid variety} with microlocal monodromy:
\begin{equation}\label{eqn:modified_braid_variety}
X(\gamma;w):=\{(\epsilon_{\Right},\vec{\epsilon})\in B\times\mathbb{A}^{\ell(\gamma)}:\epsilon_{\Right}\mB_{\gamma}(\vec{\epsilon})w=\id\}\xrightarrow[]{\mumon} T: (\epsilon_{\Right},\vec{\epsilon})\mapsto D(\epsilon_{\Right}).
\end{equation}
Denote $X(\gamma,\vec{C}_{\Right};w):=\mumon^{-1}(\vec{C}_{\Right})$ (\emph{restricted modified braid variety}).
Denote $\Delta:=\Delta_n$. Recall that
\[
R_{\vec{r}}=R_1 = \{x\in G, \epsilon\in\hat{X}_1(\Delta\beta;w_0): x\mon(\epsilon)x^{-1}=\id\} = G \times X(\Delta\beta;w_0) \xrightarrow[]{\mu} G_{\vec{r},1},
\]
where $\mu(x,\epsilon=(\epsilon_{\Right},\vec{\epsilon})) = \mumon(\epsilon)=D(\epsilon_{\Right})$. 
By Corollary \ref{cor:restricted_Betti_moduli_stack_via_braids}, \ref{cor:restricted_Betti_moduli_space_via_braids}, we have
\begin{equation*}
\modulistack_1\simeq [R_1/(G\times GL(\vec{r}_{\Right}))] \simeq [X(\Delta\beta;w_0)/T];
\quad \modulistack_L\simeq [\mu^{-1}(\vec{C}_{\Right})/(G\times Z(\vec{C}_{\Right}))] \simeq [X(\Delta\beta,\vec{C}_{\Right};w_0)/Z(\vec{C}_{\Right})].
\end{equation*}
Recall that $Z(\vec{C}_{\Right})=\{(h_{\Right}^j)_{j=1}^n\in T: h_{\Right}^j=h_{\Right}^{o_v}\in\field^{\times}, \forall j\in O_v\}\isomorphic (\field^{\times})^c$.
By Lemma \ref{lem:smoothness_of_representation_varieties}, $R_1$ and hence $X(\Delta\beta;w_0)=R_1/G$ is smooth affine and connected.
By Lemma \ref{lem:generic_condition_implies_free_action}, $PZ:=Z(\vec{C}_{\Right})/\field^{\times}\isomorphic (\field^{\times})^{c-1}$ acts freely on $X(\Delta\beta,\vec{C}_{\Right};w_0)$.
Here, $\field$ is an algebraically closed field of characteristic $0$.

\begin{proposition}[smoothness and cell decomposition of Betti moduli spaces]\label{prop:cell_decomposition}
Let $G_{1,1}\ni \vec{C}_{\Right}\isomorphic \vec{C}$ be reduced and generic. Suppose that the Betti moduli space 
$\modulispace_B=\modulispace_L$ is nonempty, then: 
\begin{enumerate}[wide,labelwidth=!,labelindent=0pt]
\item
$\vec{C}_{\Right}$ is a regular value of $\mu$ and $\modulispace_B$ is a smooth affine and connected variety 
of dimension $d=d_{\beta_1}:=\ell(\beta_1)-2\ell_0-n-c+2=\ell(\beta)-n-c+2$ over $\field$;

\item
There exists a cell decomposition into locally closed subvarieties
\[
\modulispace_B(\field)=\sqcup_{p\in\cW^*(\beta_1)} \modulispace_B^p(\field)
\]
such that 
\begin{enumerate}
\item
For each $p\in\cW^*(\beta_1)$, we have a natural isomorphism
\[
\modulispace_B^p(\field)\isomorphic (\field^*)^{a(p)}\times\field^{b(p)},\quad (a(p)=|S_p|-n-c+2,~~b(p)=|U_p|-\ell_0)
\]
such that $d=a(p)+2b(p)$.

\item
$\exists!p_m\in\cW^*(\beta_1)$ such that $b(p_m)=0$. In addition, $\modulispace_B^{p_m}(\field)$ is open and dense in $\modulispace_B(\field)$, and
\[
\modulispace_B^{p_m}(\field)\isomorphic (\field^*)^{a(p_m)}=(\field^*)^d.
\]
\end{enumerate}
\end{enumerate}
\end{proposition}

This will be a direct consequence of the cell decomposition of braid varieties, which we discuss below.

\begin{remark}
The moduli space $\modulispace_1$ has been studied in \cite[\S 6]{STZ17}. However, it's the restricted version $\modulispace_L$ that generalizes the Betti moduli space in nonabelian Hodge theory, P=W conjectures, etc. See \cite{HLRV11,Boa14,BY15} for the cases of character varieties.
On the other hand, if $c=1$, i.e. $\Lambda=\beta_1^{\circ}=(\Delta\beta\Delta)^{\circ}$ is connected, then $G_{\vec{r},1}=G_{1,1}=\{\vec{C}_{\Right}=\diag((-1)^{n-1},1,\text{\tiny$\cdots$},1)\}$, and we get a coincidence
$\modulispace_B=\modulispace_L=\modulispace_1$.
\end{remark}

\subsection{Cell decomposition of braid varieties}\label{subsec:braid_varieties}

By a little abuse of notations, fix any $\beta_1=\sigma_{i_{\ell}}\cdots\sigma_{i_1}\in\FBr_n^+$. For each $1\leq j\leq \ell$, denote
\begin{equation}\label{eqn:canonical_morphism_for_braid_matrix_with_coefficients}
f_j:\mathbb{A}^j \rightarrow G: (\epsilon_j,\cdots,\epsilon_1)\mapsto \mB_{i_j}(\epsilon_j)\cdots\mB_{i_1}(\epsilon_1).
\end{equation}
Recall that the braid variety in Remark \ref{rem:braid_varieties} is $X(\beta_1)\isomorphic f_{\ell}^{-1}(B)$. 
Define
\begin{equation}
p:\mathbb{A}^{\ell}\rightarrow W^{\ell+1}:\vec{\epsilon}\mapsto p(\vec{\epsilon})=(p_{\ell},\text{\tiny$\cdots$},p_0),\quad \mB_{i_m}(\epsilon_m)\text{\tiny$\cdots$} \mB_{i_1}(\epsilon_1)\in Bp_mB.
\end{equation}

\begin{definition}[{\cite[\S.5.4]{Mel19}}]\label{def:walks}
Let $p=(p_{\ell},\text{\tiny$\cdots$},p_0)\in W^{\ell+1}$. If for any \emph{position} $1\leq m\leq \ell$:
\begin{eqnarray*}
p_m=\left\{\begin{array}{cc}
\ms_{i_m}p_{m-1}~(\text{go-up})& \text{if } \ms_{i_m}p_{m-1}>p_{m-1},\\
\ms_{i_m}p_{m-1}~(\text{go-down}) \text{ or } p_{m-1}~(\text{stay}) & \text{ if } \ms_{i_m}p_{m-1}<p_{m-1},
\end{array}\right.
\end{eqnarray*}
and $p_0=p_{\ell}=\id$, we say $p$ is a \emph{walk} of $\beta_1$. Denote:
{\small\begin{equation*}
U_p:=\{\text{go-up's}\},~~S_p:=\{\text{stays}\},~~D_p:=\{\text{go-down's}\}. \Rightarrow [\ell]=\{1,\text{\tiny$\cdots$},\ell\}= U_p\sqcup D_p\sqcup S_p.
\end{equation*}}
By a length count: $|U_p|-|D_p|=\ell(p_{\ell})-\ell(p_0)=0$. Denote $\cW(\beta_1):=\{\text{walks of } \beta_1\}$.
\end{definition}

\noindent{}\textbf{Convention} \setword{{\color{blue}$2$}}{convention:group_action}: If the context is clear, for a group action of $h\in H$ on any variety $Y\ni y$, denote
\begin{equation}\label{eqn:convention_for_an_action}
\hat{y}:=h\cdot y.
\end{equation}

For any $1\leq m\leq \ell=\ell(\beta_1)$, denote
\begin{equation}
\ms_{<m}(\beta_1):=\text{\tiny$\prod_{q=m-1}^1$} \ms_{i_q},\quad \ms_{>m}(\beta_1):= \text{\tiny$\prod_{q=\ell}^{m+1}$} \ms_{i_q}.
\end{equation}
We use \textbf{Convention} \ref{convention:group_action}. Recall that $x^y=yxy^{-1}$ in $G$, and we write {\small$t=\diag(t_1,\text{\tiny$\cdots$},t_n)\in T$}. 
\begin{proposition}\label{prop:cell_decomposition_of_braid_varieties}
We have \emph{$B$-equivariant} decompositions into locally closed subvarieties:
\begin{eqnarray}\label{eqn:cell_decomposition_of_braid_varieties}
&&X(\beta_1)=\sqcup_{p\in\cW(\beta_1)} X_p(\beta_1),\quad \varphi:X_p(\beta_1):=X_p^{\ell}(\beta_1)\isomorphic (\field^{\times})^{|S_p|}\times \field^{|U_p|}:\vec{\epsilon}\mapsto (\epsilon_m')_{m\in S_p\sqcup U_p},\\
&&X(\beta_1,C)=\sqcup_{p\in\cW(\beta_1)} X_p(\beta_1,C),\quad X_p(\beta_1,C):=X_p(\beta_1)\cap X(\beta_1,C),\nonumber
\end{eqnarray}
such that
\begin{enumerate}[wide,labelwidth=!,labelindent=0pt]
\item
The inherited action of $b\in B$ on $(\epsilon_m')_{m\in S_p\sqcup U_p}\in (\field^{\times})^{|S_p|}\times \field^{|U_p|}$ satisfies:
\begin{enumerate}
\item
If $b=u\in U\subset B$, then
$
\hat{\epsilon}_m'=\epsilon_m', \forall m\in S_p.
$

\item
If $b=t\in T\subset B$, then
\begin{equation}
\hat{\epsilon}_m'=(t^{p_{m-1}})_{i_m}(t^{p_{m-1}})_{i_m+1}^{-1}\epsilon_m', \forall m\in S_p\sqcup U_p.
\end{equation}
\end{enumerate}

\item
$\mumon:X_p(\beta_1)\rightarrow T$ is identified with
\begin{equation}\label{eqn:formula_for_mumon}
\mumon((\epsilon_m')_{m\in S_p\sqcup U_p})=\mumon((\epsilon_m')_{m\in S_p})= \text{\tiny$\prod_{m\in S_p}$} (\mK_{i_m}(-\epsilon_m'^{-1})\mK_{i_m+1}(\epsilon_m'))^{\ms_{>m}(\beta_1)}.
\end{equation}
In particular, $\det(\mumon((\epsilon_m')_{m\in S_p\sqcup U_p}))=(-1)^{|S_p|}=(-1)^{\ell}$.
\end{enumerate}
\end{proposition}

The cell decomposition goes back to \cite{Mel19}, see also \cite{CGGS20}. For a detailed proof of the above proposition, see \cite[Prop.1.16]{Su23}.
Here we only sketch the main points.

\begin{proof}[Sketch of proof]
The proof is done by diagram calculus.
Let $p\in\cW(\beta_1)$, $0\leq m\leq\ell$. Define a closed subvariety of $\mathbb{A}^m$ and subsets of $[\ell]$:
\begin{eqnarray}
&&X_p^m(\beta_1):=\cap_{1\leq j\leq m}f_j^{-1}(Bp_jB)\subset\mathbb{A}^m,\quad X_p^0(\beta_1)=\pt.\\
&&U_p^m:=U_p\cap[m],~~S_p^m:=S_p\cap[m],~~D_p^m:=D_p\cap[m].~~\Rightarrow~~ [m]= U_p^m\sqcup D_p^m\sqcup S_p^m.\nonumber
\end{eqnarray}
Then the decomposition $X(\beta_1)=\sqcup_{\in\cW(\beta_1)}X_p(\beta_1)$, etc. and the structure of each cell $X_p(\beta_1)$ are established by induction via the following lemma:
\begin{lemma}\label{lem:inductive_cells_for_braid_varieties}
We have $p(X(\beta_1))\subset\cW(\beta_1)$. Moreover, for any $p\in\cW(\beta_1)$ and $1\leq m\leq \ell$,
\begin{eqnarray*}
X_p^m(\beta_1)\isomorphic\left\{\begin{array}{ll}
\field_{\epsilon_m'}\times X_p^{m-1}(\beta_1) & \text{if $p_m=\ms_{i_m}p_{m-1}>p_{m-1}$\quad (go-up)};\\
X_p^{m-1}(\beta_1) & \text{if $p_m=\ms_{i_m}p_{m-1}<p_{m-1}$\quad (go-down)};\\
\field_{\epsilon_m'}^{\times}\times X_p^{m-1}(\beta_1) & \text{if $\ms_{i_m}p_{m-1}<p_{m-1}$ and $p_m=p_{m-1}$ \quad (stay)}.
\end{array}\right.
\end{eqnarray*}
\end{lemma}
The rest is a straightforward application of diagram calculus.
\end{proof}

\begin{corollary}\label{cor:free_action_vs_image}
In Proposition \ref{prop:cell_decomposition_of_braid_varieties}, for each $p\in\cW(\beta_1)$, $PT:=T/\field^{\times}$ acts freely on $X_p(\beta_1)$ if and only if 
$\mumon(X_p(\beta_1))=G_{1,1}=\{(\vec{C}_{\Right})\in T: \det\vec{C}_{\Right} = (-1)^{\ell}\}$.
\end{corollary}

\begin{proof}
Recall $p=(p_{\ell}=\id,\text{\tiny$\cdots$},p_1,p_0=\id)\in W^{\ell+1}$.
Define 
\[
J_p:=\{\underline{\permutation}_{i_m}:=p_{m-1}^{-1}\permutation_{i_m}p_{m-1}, m\in S_p\};\quad \tilde{J}_p:=\{\permutation_{>m}(\beta_1)\permutation_{i_m}\permutation_{>m}(\beta_1)^{-1}, m\in S_p\}.
\]
For any subset $I\subset W=S_n$, let $\langle I\rangle$ denote the subgroup generated by $I$, and
\[
\overline{\langle I \rangle}:=\langle (a~\tau(a)), \forall a\in[n], \tau\in I\rangle \subset S_n.
\]
Also, denote $J_p':=\{\permutation_{<m}^{-1}\permutation_{i_m}\permutation_{<m}(\beta_1):m\in S_p\} = \permutation(\beta_1)^{-1}\tilde{J}_p\permutation(\beta_1)$.
By Proposition \ref{prop:cell_decomposition_of_braid_varieties}.($1.b$), $PT:=T/\field^{\times}$ acts freely on $X_p(\beta_1)$ if and only if $\overline{\langle J_p \rangle}=S_n$.
By Proposition \ref{prop:cell_decomposition_of_braid_varieties}.($1.c$), $\mumon(X_p(\beta_1))=G_{1,1}$ if and only if 
$\overline{\langle \tilde{J}_p \rangle}=S_n \Leftrightarrow \overline{\langle J_p' \rangle}=S_n$.
By \cite[(2.4.18)]{Su23}, $\overline{\langle J_p \rangle} = \overline{\langle J_p' \rangle}$. 
We're done.
\end{proof}

Now, we go back to the proof of Proposition \ref{prop:cell_decomposition}.

\begin{proof}[Proof of Proposition \ref{prop:cell_decomposition}.]
Recall that $\beta_1=\Delta\beta\Delta$. 
Recall that $X(\Delta\beta;w_0)=\{\vec{\epsilon}=(\epsilon_{\ell},\text{\tiny$\cdots$},\epsilon_1)\in X(\beta_1): \epsilon_1=\text{\tiny$\cdots$}=\epsilon_{\ell_0}=0\}$, where $\ell_0=\ell(w_0)$.
Define $X_p(\Delta\beta;w_0)$ and $X_p(\Delta\beta,\vec{C}_{\Right};w_0)$ accordingly.
By definition, for any $p\in\cW(\beta_1)$, we must have $\{1,\text{\tiny$\cdots$},\ell_0\}\subset U_p$ and $\{\ell,\text{\tiny$\cdots$},\ell-\ell_0+1\}\subset D_p$.
Then by Proposition \ref{prop:cell_decomposition_of_braid_varieties}, we obtain a cell decomposition
\[
X(\Delta\beta;w_0)=\sqcup_{p\in\cW(\beta_1)} X_p(\Delta\beta;w_0),\quad X_p(\Delta\beta;w_0)\isomorphic (\field^{\times})^{|S_p|} \times \field^{|U_p|-\ell_0},
\]
and $X_p(\beta_1)= X_p(\Delta\beta;w_0)\times \field^{\ell_0}$.
By Proposition \ref{prop:cell_decomposition}.(2), 
$X_p(\beta_1)\cap X(\beta_1,\vec{C}_{\Right})=(X_p(\Delta\beta;w_0)\cap X(\Delta\beta,\vec{C}_{\Right};w_0))\times\field^{\ell_0}$.
Denote $\cW^*(\beta_1):=\{p\in\cW(\beta_1): X_p(\beta_1)\cap X(\beta_1,\vec{C}_{\Right})\neq\emptyset\}$. Then, we obtain a cell decomposition
\[
X(\Delta\beta,\vec{C}_{\Right};w_0)=\sqcup_{p\in\cW^*(\beta_1)} X_p(\Delta\beta,\vec{C}_{\Right};w_0).
\]
By Lemma \ref{lem:generic_condition_implies_free_action}, $PZ:=Z(\vec{C}_{\Right})/\field^{\times}$ acts freely on $X(\Delta\beta,\vec{C}_{\Right};w_0)$, hence $PT:=T/\field^{\times}$ acts freely on $\mumon^{-1}(T\cdot \vec{C}_{\Right})$. So by Proposition \ref{prop:cell_decomposition_of_braid_varieties}.($1.b$), for any $p\in\cW^*(\beta_1)$, $PT$ acts freely on $X_p(\beta_1)$ and $X_p(\Delta\beta;w_0)$, hence $\mumon(X_p(\beta_1))=\mumon(X_p(\Delta\beta;w_0))=G_{1,1}\isomorphic (\field^{\times})^{n-1}$ by Corollary \ref{cor:free_action_vs_image}. It follows that $\vec{C}_{\Right}$ is a \emph{regular value} of
$
\mumon: X(\Delta\beta;w_0) (=\sqcup_{p\in\cW^*(\beta_1)}X_p(\Delta\beta;w_0)) \rightarrow G_{1,1},
$
and hence $X(\Delta\beta,\vec{C}_{\Right};w_0)$ is \emph{smooth affine of pure dimension}. Moreover,
\[
 X_p(\Delta\beta,\vec{C}_{\Right};w_0) \isomorphic (\field^{\times})^{|S_p|-n+1}\times \field^{|U_p|-\ell_0}, \forall p\in \cW^*(\beta_1).
\]

Define $p_m\in W^{\ell+1}$ so that  $S_{p_m}=[\ell]-\{1,\text{\tiny$\cdots$},\ell_0,\ell-\ell_0+1,\text{\tiny$\cdots$},\ell\}$ is maximal. Clearly, $p_m\in\cW(\beta_1)$.
By Proposition \ref{prop:cell_decomposition_of_braid_varieties}.(2), $\mumon(X_{p_m}(\beta_1))\subset G_{1,1}$ has maximal dimension. As $X(\Delta\beta;w_0)\neq\emptyset$, there exists $p\in \cW^*(\beta_1)\neq \emptyset$,
so $\mumon(X_p(\beta_1))=G_{1,1}$. Thus, we must have $\mumon(X_{p_m}(\beta_1))= G_{1,1}$. In particular, $p_m\in\cW^*(\beta_1)$.
Now, (2) follows by passing to the quotient $\modulispace_B=X(\Delta\beta,\vec{C}_{\Right};w_0)/PZ$. 
For $(1)$, we have seen that $\vec{C}_{\Right}$ is a regular value of $\mu$, and $\modulispace_B=X(\Delta\beta,\vec{C}_{\Right};w_0)/PZ$ is smooth affine of pure dimension.
But in the cell decomposition of $\modulispace_B$, there exists a unique maximal cell $\modulispace_B^{p_m}:=X_{p_m}(\Delta\beta,\vec{C}_{\Right};w_0)/PZ\isomorphic (\field^{\times})^{|S_{p_m}|-n-c+2}=(\field^{\times})^{\ell-2\ell_0-n-c+2}$. This means that $\modulispace_B$ is connected of dimension $d=\ell-2\ell_0-n-c+2$.
This completes the proof.
\end{proof}

\begin{remark}\label{rem:connection_with_augmentations}
For any $\beta_1=\Delta\beta\Delta\in\Br_n^+$, we have seen that
$\modulistack_1\isomorphic [X(\Delta\beta;w_0)/T]$.
In fact, $\modulistack_1$ (or $\modulispace_1$) is also related to some other moduli stacks/spaces.

\begin{enumerate}[wide,labelwidth=!,labelindent=0pt]
\item
One can form the \emph{rainbow closure $\beta^{>}$} as in Figure \ref{fig:braid_closure}: a Legendrian link in $J^1\mathbb{R}_x$ whose front projection is the rainbow closure $\beta^{>}$ of $\beta$, which is equipped with the standard binary Maslov potential (zero values on all the strands in $\beta$). Here, $J^1\mathbb{R}_x=T^*\mathbb{R}_x\times\mathbb{R}_z=\mathbb{R}_{x,y,z}^3$ is identified with the standard contact three-space with contact $1$-form $\alpha=dz-ydx$. Similar to $\modulistack_r=\modulistack_r(\beta_1^{\circ})$, one can define a moduli stack 
$\modulistack_r(\beta^{>})=\modulistack_r(\mathbb{R}_{x,z}^2,\{\infty\},\{\beta^{>}\})$, which consists of microlocal rank $r$ (along $\beta^>$) constructible sheaves with compact support on $\mathbb{R}_{x,z}^2$.
Similar to the case of cylindrical closures, by a local study of the micro-support conditions, one can prove a natural isomorphism
\[
\modulistack_r(\beta_1^{\circ})\isomorphic \modulistack_r(\beta^{>}).
\]
For example, when $\beta=\sigma_1^3\in\Br_2^+$, so $\beta_1=\sigma_1^5$, both $\modulispace_1=\modulispace_1(\beta_1^{\circ})$ and $\modulispace_1(\beta^{>})$ are identified with the moduli space of $5$ one-dimensional subspaces  in $\field^2$ which are cyclically consecutive transversal.

\begin{figure}[!htbp]
\begin{center}
\begin{tikzcd}[column sep=5pc,ampersand replacement=\&]
\begin{tikzpicture}[baseline=-.5ex,scale=0.5]
\begin{scope}
\draw[thick] (0,0) to[out=0,in=180] (2,1);
\draw[white, fill=white] (1,0.5) circle (0.1);
\draw[thick] (0,1) to[out=0,in=180] (2,0) to[out=0,in=180] (4,1);
\draw[white, fill=white] (3,0.5) circle (0.2);
\draw[thick] (2,1) to[out=0,in=180] (4,0) to[out=0,in=180] (6,1);
\draw[white, fill=white] (5,0.5) circle (0.2);
\draw[thick] (4,1) to[out=0,in=180] (6,0);
\draw (3,-1.5) node[below] {$n=2, \beta=\sigma_1^3$};
\end{scope}
\end{tikzpicture}
\arrow[r,"{\tiny rainbow~closure}"]
\&
\begin{tikzpicture}[baseline=-.5ex,scale=0.5]
\begin{scope}
\draw[thick] (2,-1) to[out=0,in=180] (4,0);
\draw[white, fill=white] (3,-0.5) circle (0.2);
\draw[thick] (1,0.5) to[out=0,in=180] (4,-1) to[out=0,in=180] (6,0);
\draw[white, fill=white] (5,-0.5) circle (0.2);
\draw[thick] (4,0) to[out=0,in=180] (6,-1) to[out=0,in=180] (9,0.5);
\draw[white, fill=white] (7,-0.5) circle (0.2);
\draw[thick] (6,0) to[out=0,in=180] (8,-1);
\draw[thick] (-1,0.5) to[out=0,in=180] (2,-1);
\draw[thick] (8,-1) to[out=0,in=180] (11,0.5);

\draw[thick] (1,0.5) to[out=0,in=180] (5,2) to[out=0,in=180] (9,0.5);
\draw[thick] (-1,0.5) to[out=0,in=180] (5,3) to[out=0,in=180] (11,0.5);
\draw (5,-1.5) node[below] {$\beta^>\hookrightarrow \mathbb{R}_{x,z}^2$};
\end{scope}
\end{tikzpicture}
\end{tikzcd}
\end{center}
\caption{Rainbow closure: $n=2,\beta=\sigma_1^3$.}
\label{fig:braid_closure}
\end{figure}

\item
In contact geometry, one can associate to the Legendrian link $\beta^{>}\hookrightarrow (J^1\mathbb{R}_x,dz-ydx)$ a DGA, 
called Chekanov-Eliashberg (or Legendrian Contact Homology) DGA, as in \cite[\S 2.2]{NRSSZ20}. See also \cite{Eli98,Che02,ENS01}.
If we place $n$ base points $*_1,\text{\tiny$\cdots$},*_n$ on the right cusps of $\beta^{>}$, then the associated ($\mathbb{Z}$-graded) DGA $\cA=\cA(\beta^{>},*_1,\text{\tiny$\cdots$},*_n)$ 
is generated over $\mathbb{Z}\langle t_1^{\pm 1},\text{\tiny$\cdots$},t_n^{\pm 1}\rangle$ by Reeb chords of $\beta^{>}$. The differential $\differential$ of $\cA$ counts holomorphic disks in the symplectization $\mathbb{R}_t\times\mathbb{R}_{x,y,z}^3$ with boundary along the Lagrangian cylinder $\mathbb{R}_t\times \beta^{>}$, with one boundary puncture limiting a Reeb chord at $+\infty$ and several boundary punctures limiting Reeb chords at $-\infty$. For example, if $\beta=\sigma_1^3$, then
\begin{eqnarray*}
&&\mathcal{A}(\Lambda)=\mathbb{Z}\langle t_1^{\pm 1}, t_2^{\pm 1}, a_1,a_2,a_3, c_1, c_2\rangle,\quad |t_i^{\pm 1}|=0, |a_i|=0, |c_i|=1;\quad \partial t_i=\partial t_i^{-1}=0,\quad\partial a_i=0;\\
&&\partial c_1=t_1^{-1}+a_1+a_3+a_1a_2a_3;\quad \partial c_2=t_2^{-1}+a_2+(1+a_2a_3)t_1(1+a_1a_2).
\end{eqnarray*}
See Figure \ref{fig:Legendrian_trefoil_knot} for an illustration.
The associated \emph{augmentation variety} $\aug(\beta^{>},*_1,\text{\tiny$\cdots$},*_n)$ is the set of $\mathbb{Z}$-graded DGA maps $\epsilon:(\cA,\partial)\rightarrow (\field,0)$.
It's equipped with a natural action of the torus $T=(\field^{\times})^n$.
In the example of $\beta=\sigma_1^3\in\Br_2^+$, we have
\begin{eqnarray*}
\aug(\Lambda,\overrightarrow{*};\field)&\isomorphic&\{(\epsilon(t_1),\epsilon(t_2),\epsilon(a_1),\epsilon(a_2),\epsilon(a_3))\in(\field^*)^2\times\field^3|\\
&&\epsilon(t_1)^{-1}+\epsilon(a_1)+\epsilon(a_3)+\epsilon(a_1)\epsilon(a_2)\epsilon(a_3)=0, \epsilon(t_2)=-\epsilon(t_1)^{-1}\}.
\end{eqnarray*}
And, the torus action is: $((\lambda_1,\lambda_2),\epsilon)\in\mathbb{G}_m^2\times\aug(\Lambda;\field)\mapsto (\lambda_1,\lambda_2)\cdot\epsilon\in\aug(\Lambda;\field)$ with
\begin{eqnarray*}
&&((\lambda_1,\lambda_2)\cdot \epsilon)(t_1)=\lambda_2^{-1}\epsilon(t_1)\lambda_1, ((\lambda_1,\lambda_2)\cdot \epsilon)(t_2)=\lambda_1^{-1}\epsilon(t_2)\lambda_2.\\
&&((\lambda_1,\lambda_2)\cdot \epsilon)(a_1)=\lambda_1^{-1}\epsilon(a_1)\lambda_2, ((\lambda_1,\lambda_2)\cdot \epsilon)(a_2)=\lambda_2^{-1}\epsilon(a_2)\lambda_1, ((\lambda_1,\lambda_2)\cdot \epsilon)(a_3)=\lambda_1^{-1}\epsilon(a_3)\lambda_2.
\end{eqnarray*}
By ``augmentations are sheaves'' \cite{NRSSZ20}, we obtain a natural isomorphism
\[
[\aug(\beta^{>},*_1,\text{\tiny$\cdots$},*_n)/T] \isomorphic \modulistack_1(\beta^{>}).
\]
Altogether, we have natural identifications
\begin{equation}
[\aug(\beta^>,*_1,\text{\tiny$\cdots$},*_n)/T] \isomorphic \modulistack_1(\beta^{>}) \isomorphic \modulistack_1(\beta_1^{\circ}) \isomorphic [X(\Delta\beta;w_0)/T].
\end{equation}

\begin{figure}[!htbp]
\begin{center}
\begin{tikzcd}[row sep=1pc,column sep=2pc,ampersand replacement=\&]
\begin{tikzpicture}[baseline=-.5ex,scale=0.5]
\begin{scope}
\draw[thick] (2,-1) to[out=0,in=180] (4,0);
\draw[white, fill=white] (3,-0.5) circle (0.2);
\draw[thick] (1,0.5) to[out=0,in=180] (4,-1);
\draw[thick,->] (4,-1) to[out=0,in=180] (6,0);

\draw[white, fill=white] (5,-0.5) circle (0.2);

\draw[thick,->] (4,0) to[out=0,in=180] (6,-1);
\draw[thick] (6,-1) to[out=0,in=180] (9,0.5);
\draw[white, fill=white] (7,-0.5) circle (0.2);
\draw[thick] (6,0) to[out=0,in=180] (8,-1);
\draw[thick] (-1,0.5) to[out=0,in=180] (2,-1);
\draw[thick] (8,-1) to[out=0,in=180] (11,0.5);

\draw[thick] (1,0.5) to[out=0,in=180] (5,2);
\draw[thick,<-] (5,2) to[out=0,in=180] (9,0.5);

\draw[thick] (-1,0.5) to[out=0,in=180] (5,3);
\draw[thick,<-] (5,3) to[out=0,in=180] (11,0.5);

\draw (3,-0.6) node[below] {$a_1$};
\draw (5,-0.6) node[below] {$a_2$};
\draw (7.1,-0.6) node[below] {$a_3$};
\draw (8.8,0.4) node[below] {$c_1$};
\draw (10.8,0.4) node[below] {$c_2$};

\draw (8.5,0.5) node[right] {$*_1$};
\draw (10.5,0.5) node[right] {$*_2$};

\draw (5,-1.5) node[below] {$\beta^>\hookrightarrow\mathbb{R}_{x,z}^2$};

\end{scope}
\end{tikzpicture}
\arrow[r,"\mathrm{Res}"]\&
\begin{tikzpicture}[baseline=-.5ex,scale=0.5]
\begin{scope}
\draw[thick] (2,-1) to[out=0,in=180] (4,0);
\draw[white, fill=white] (3,-0.5) circle (0.2);
\draw[thick] (1,0.5) to[out=-90,in=180] (2,0) to[out=0,in=180] (4,-1);
\draw[thick,->] (4,-1) to[out=0,in=180] (6,0);

\draw[white, fill=white] (5,-0.5) circle (0.2);

\draw[thick,->] (4,0) to[out=0,in=180] (6,-1);
\draw[thick] (6,-1) to[out=0,in=180] (8,0) to[out=0,in=-135] (8.5,0.5) to[out=45,in=90] (9.5,0.5);

\draw[white, fill=white] (7,-0.5) circle (0.2);
\draw[white,fill=white] (8.5,0.5) circle(0.1);

\draw (8.6,0.5) node[below] {$c_1$};

\draw[thick] (6,0) to[out=0,in=180] (8,-1);
\draw[thick] (-1,0.5) to[out=-90,in=180] (2,-1);
\draw[thick] (8,-1) to[out=0,in=-135] (10.5,0.5) to[out=45,in=90] (11.5,0.5);

\draw[white,fill=white] (10.5,0.5) circle(0.1);

\draw (10.6,0.5) node[below] {$c_2$};

\draw[thick] (1,0.5) to[out=90,in=180] (5,2);
\draw[thick,<-] (5,2) to[out=0,in=135] (8.5,0.5) to[out=-45,in=-90] (9.5,0.5);

\draw[thick] (-1,0.5) to[out=90,in=180] (5,3);
\draw[thick,<-] (5,3) to[out=0,in=135] (10.5,0.5) to[out=-45,in=-90] (11.5,0.5);

\draw (3,-0.6) node[below] {$a_1$};
\draw (5,-0.6) node[below] {$a_2$};
\draw (7.1,-0.6) node[below] {$a_3$};

\draw (9,0.5) node[right] {$*_1$};
\draw (11,0.5) node[right] {$*_2$};

\draw (5,-1.5) node[below] {$\mathrm{Res}(\beta^>)\hookrightarrow \mathbb{R}_{x,y}^2$};
\end{scope}
\end{tikzpicture}
\end{tikzcd}
\end{center}
\caption{The right-handed Legendrian trefoil knot. Left: the front projection $\pi_{xz}(\Lambda)$ defined by $\beta^>$. Right: Ng's resolution construction \cite{Ng03} of $\Lambda$.}
\label{fig:Legendrian_trefoil_knot}
\end{figure}

\item
There is a natural $T$-equivariant isomorphism $\aug(\beta^{>},*_1,\text{\tiny$\cdots$},*_n)\isomorphic X(\Delta\beta;w_0)$, inducing the above identification.
Moreover, the ruling decomposition \cite{HR15} of $\aug(\beta^{>},*_1,\text{\tiny$\cdots$},*_n) $ coincides with the cell decomposition of $X(\Delta\beta;w_0)$ 
induced by Proposition \ref{prop:cell_decomposition_of_braid_varieties}.
\end{enumerate}
\end{remark}

\begin{remark}
In the first version of this paper, a weaker statement of Proposition \ref{prop:cell_decomposition} has been proved using augmentation varieties, 
under the assumption that $\beta_1^{\circ}$ is connected. In the current version, we drop the assumption ``$\beta_1^{\circ}$ is connected'' and use braid varieties instead.
The new proof here is more natural and can be generalized to the general case of arbitrary punctured Riemann surfaces. This will appear in a subsequent paper.
Nevertheless, by the above remark, the two approaches are essentially equivalent.
\end{remark}

\section{Dual boundary complexes of Betti moduli spaces over $(\mathbb{C}P^1,\infty)$}\label{sec:main_theorem}

\subsection{Background on dual boundary complexes}\label{sec:dual_boundary_complexes}

We firstly review the basic concepts, properties, and tools concerning the dual boundary complex of a complex algebraic variety.

If $X$ is a smooth quasi-projective variety over $\mathbb{C}$, by Hironaka's resolution of singularities, one can obtain a smooth compactification $\overline{X}$ with simple normal crossing divisor $D$ at infinity. We may also assume the intersections of the irreducible components of $D$ are connected. By encoding the information of the intersections of the irreducible components of $D$, we get a simplicial complex, called \emph{dual boundary complex} $\mathbb{D}\boundary X$.

More precisely, let $D=\cup_{i=1}^m D_i$ with $D_i$'s the irreducible components.
\begin{definition}
The dual boundary complex $\mathbb{D}\boundary X$ is the dual complex of $D$: The 0-simplices are $1,2,\ldots,m$, with $i$ corresponds to the irreducible component $D_i$. 
$\mathbb{D}\boundary X$ contains the simplex $[i_1,i_2,\ldots,i_k]$ if and only if $D_{i_1}\cap D_{i_2}\cap\ldots\cap D_{i_k}$ is non-empty.
\end{definition}

The definition of $\mathbb{D}\boundary X$ is justified by the following

\begin{theorem}\cite{Dan75}
The homotopy type of $\mathbb{D}\partial X$ is an invariant of $X$, and is independent of the choice of the compactification $\overline{X}$.
\end{theorem}

\begin{example}[Toy examples]
\begin{enumerate}
\item
For $X=\mathbb{A}^1$, we can take $\overline{X}=\mathbb{C}P^1$ with $D=\{\infty\}$. So, $\mathbb{D}\boundary(\mathbb{A}^1)=*$.
\item
For $X=\mathbb{G}_m$, we can take $\overline{X}=\mathbb{C}P^1$ with $D=\{0\}\sqcup\{\infty\}$. So, $\mathbb{D}\boundary(\mathbb{G}_m)=S^0$.
\item
For $X=\mathbb{A}^1\times\mathbb{G}_m$, we can take $\overline{X}=\mathbb{C}P^1\times\mathbb{C}P^1$, with 
\[D=(D_1=\{\infty\}\times\mathbb{C}P^1)\cup (D_2=\mathbb{C}P^1\times\{0\})\cup (D_3=\mathbb{C}P^1\times\{\infty\}).
\] 
So $\mathbb{D}\boundary(\mathbb{A}^1\times\mathbb{G}_m)=[1,2]\cup_{\{1\}}[1,3]$ is the union of two 1-simplices along a common 0-simplex, in particular contractible. Also, Observe that $\mathbb{D}\boundary(\mathbb{A}^1\times\mathbb{G}_m)=\{1\}\join\{2,3\}\isomorphic (\mathbb{D}\boundary\mathbb{A}^1)\join(\mathbb{D}\boundary\mathbb{G}_m)$ is the \emph{join} of the dual boundary complexes of the factors of $X$.
\item
For $X=\mathbb{G}_m\times\mathbb{G}_m$, we can take $\overline{X}=\mathbb{C}P^1\times\mathbb{C}P^1$, with 
\[
D=(D_1=\{0\}\times\mathbb{C}P^1)\cup(D_2=\{\infty\}\times \mathbb{C}P^1)\cup (D_3=\mathbb{C}P^1\times\{0\})\cup (D_4=\mathbb{C}P^1\times\{\infty\}).
\]
So $\mathbb{D}\boundary(\mathbb{G}_m\times\mathbb{G}_m)=[1,3]\cup[1,4]\cup[2,3]\cup[2,4]=\{1,2\}\join\{3,4\}\isomorphic S^1$, which is also the join $(\mathbb{D}\boundary\mathbb{G}_m)\join(\mathbb{D}\boundary\mathbb{G}_m)$.\\
It's then easy to see that $\mathbb{D}\boundary \mathbb{G}_m^N\isomorphic S^{N-1}$, which also follows immediately from Lemma \ref{lem:join} below.

\end{enumerate}
\end{example}

The previous example has already illustrated the compatible property of dual boundary complexes with products:

\begin{lemma}\cite[Lem.6.2]{Pay13}\label{lem:join}
Let $X,Y$ be smooth quasi-projective varieties over $\mathbb{C}$, then there's a natural homeomorphism
\[
\mathbb{D}\boundary(X\times Y)\isomorphic (\mathbb{D}\boundary X) \join (\mathbb{D}\boundary Y)
\]
where `$\join$' stands for `\emph{join}'. In particular, if $X=\mathbb{A}^1$, then $\mathbb{D}\boundary(\mathbb{A}^1\times Y)\homotopic *$ is contractible.  
\end{lemma} 

The following lemma then implies that, by removing a smaller piece of the form $\mathbb{A}^1\times Y$, the homotopy type of $\mathbb{D}\boundary X$ is unchanged:

\begin{lemma}\cite[Lem.2.3]{Sim16}\label{lem:dual_boundary_complex_remove_lemma}
Let $X$ be a smooth irreducible quasi-projective variety over $\mathbb{C}$. If $Z\subset X$ is a smooth irreducible closed sub-variety of smaller dimension with complement $U$, and $\mathbb{D}\boundary Z\homotopic *$ is contractible. Then the natural map $\mathbb{D}\boundary X\rightarrow\mathbb{D}\boundary U$ is a homotopy equivalence.
\end{lemma}

By an inductive procedure of removing smaller pieces of the form $\mathbb{A}^1\times Y$, one can further simplify the problem of determining the homotopy type of the dual boundary complex $\mathbb{D}\boundary X$:

\begin{lemma}\cite[Prop.2.6]{Sim16}\label{lem:dual_boundary_complex_inductive_remove}
Let $X$ be a smooth irreducible quasi-projective variety over $\mathbb{C}$. Let $U$ be a non-empty open subset, with the complement $Z=X\setminus U$ admitting a decomposition into finitely many locally closed sub-varieties of the form $Z_j\isomorphic \mathbb{A}^1\times Y$. If there is a total order on the indices such that $\cup_{j\leq a}Z_j$ is closed for all $a$. Then we obtain a homotopy equivalence $\mathbb{D}\boundary X\homotopic\mathbb{D}\boundary U$.
\end{lemma}

Notice that in the above statement, we no longer require that $Z_j$ is smooth. The reason is that, if some $Z_j$ is singular, we can further decompose $Z_j$ in to finitely many smooth pieces. Then an inductive procedure reduces the proof essentially to Lemma \ref{lem:dual_boundary_complex_remove_lemma}.
As in \cite{Sim16}, Lemma \ref{lem:dual_boundary_complex_inductive_remove} will be the key tool in the proof of our main Theorem \ref{thm:main}.

\subsection{The dual boundary complex of the Betti moduli space}

Finally, we come to our main theorem.
 i.e. the weak geometric P=W conjecture \ref{conj:homotopy type conj} holds for the Betti moduli space $\modulispace_B(\mathbb{C})$ considered in this article:
 
Recall that $\beta\in\Br_n^+$ is a $n$-strand positive braid and $\beta_1=\Delta\beta\Delta$ with $\ell=\ell(\beta_1)$. 
The orbits of $\permutation(\beta_1)\in S_n$ induces a partition $[n]=\sqcup_{v=1}^c O_v$, and $o_v=\min\{j\in O_v\}$.
Besides, $G_{1,1}=\{\vec{C}_{\Right}\in T: \det \vec{C}_{\Right} = \det\permutation(\beta_1) = (-1)^{\ell} = (-1)^{n-c}\}$.
Let $G_{1,1}\ni \vec{C}_{\Right}\isomorphic (C_v=C_{\Right}^{o_v})_{v=1}^c$ be reduced and generic, and $L$ be the local system on $\beta_1^{\circ}$ determined by the $\permutation(\beta_1)$-twisted conjugacy class of $\vec{C}_{\Right}$.
The Betti moduli space of interest is
\[
\modulispace_B:=\modulispace_L(\mathbb{C}P^1,\{\infty\},\{\beta_1^{\circ}\};\mathbb{C}).
\]

Now, our main theorem establishes the weak geometric P-W conjecture \ref{conj:homotopy type conj} holds for $\modulispace_B$:
\begin{theorem}\label{thm:main}
If nonempty, then $\modulispace_B$ is a smooth affine and connected variety, 
and the dual boundary complex $\mathbb{D}\partial\modulispace_B$ of $\modulispace_B$ is homotopy equivalent to $S^{\mathrm{dim}_{\mathbb{C}}\modulispace_B -1}$.
\end{theorem}

\begin{proof}
By Proposition \ref{prop:cell_decomposition}, it suffices to show the homotopy equivalence. 
The strategy is to use the cell decomposition for $\modulispace_B$ (Proposition \ref{prop:cell_decomposition}) and apply Lemma \ref{lem:dual_boundary_complex_inductive_remove}.

Recall that $d=\dim\modulispace_B = \ell(\beta)-n-c+2$.
Let $X:=\modulispace_B$, $U:=\modulispace_B^{p_m}\isomorphic (\mathbb{C}^*)^d$, and $Z:=X\setminus U$. 
For each $1\leq i<d$, define a locally closed subvariety $Z_i$ of $X$ to be the disjoint union of $\modulispace_B^p$ of dimension $i$, then we obtain a decomposition $Z=\sqcup_{1\leq i\leq d-1} Z_i$. 
By Proposition \ref{prop:cell_decomposition}, each $Z_i$ is of the form $\mathbb{A}^1\times Y_i$, and $\cup_{j\leq a}Z_j$ is closed for all $1\leq a\leq d-1$. 
Hence, the conditions of Lemma \ref{lem:dual_boundary_complex_inductive_remove} are satisfied. It then follows that we have a homotopy equivalence
\[
\mathbb{D}\partial\modulispace_B \sim \mathbb{D}\partial (U=(\mathbb{C}^*)^d) = S^{d-1}.
\]
This finishes the proof.
\end{proof}

\begin{remark}
Observe that $\det\permutation(\beta_1)=(-1)^{\ell}=(-1)^{n-c}$, so $d= \ell(\beta)-n-c+2$ is even.
Inspired by the non-abelian Hodge theory for wild character varieties \cite{BB04}, it's natural to expect that 
the Betti moduli space $\modulispace_B$ is hyperk\"{a}hler and log Calabi-Yau. 
By \cite{Boa07,Boa14,BY15}, at least in the cases of (wild) character varieties, we know that $\modulispace_B$ is 
holomorphic symplectic. See also \cite{Boa01,Boa09,Boa21} for related results.
It's also expected that $\modulispace_B$ is a cluster variety. This makes it very interesting to study the Hodge theory 
of and mirror symmetry for $\modulispace_B$. We hope to return to this subject in the future.
For example, in the working example with $n=2$ and $\beta=\sigma_1^3$, the Betti moduli space 
$\modulispace_B=\modulispace_L(\beta_1^{\circ}) = \modulispace_1(\beta_1^{\circ})$ is the $A_2$-cluster variety \cite{STWZ19}, 
and satisfies all the properties above. It's expected to be self-mirror. 
\end{remark}

\appendix
\addtocontents{toc}{\protect\setcounter{tocdepth}{1}}

\section{Proof of Proposition \ref{prop:local_diagram_for_Betti_moduli_stacks}}\label{sec:prove_local_diagram_for_Betti_moduli_stacks}

In this section, we prove Proposition \ref{prop:local_diagram_for_Betti_moduli_stacks}, which involves establishing the identification in the middle column.
The subsequent steps become more straightforward. 
Recall that $\beta^{\vec{r}}=\sigma_{i_{\ell}}^{\vec{r}_{\ell}}\text{\tiny$\cdots$}\sigma_{i_1}^{\vec{r}_1}\in\FBr_{n,N}^+$.

\subsection{$\modulistack_{\vec{r}}(\beta)$ via quiver representations}

Define a \emph{quiver with relations} $\hat{Q}(\beta)$ as in Figure \ref{fig:framed_quivers_with_relations_for_a_positive_braid} (left):
\begin{itemize}[wide,labelwidth=!,labelindent=0pt]
\item
the set $\hat{Q}_0(\beta)$ of vertices consists of the $q_{i,j}$'s for $0\leq i\leq \ell$ and $0\leq j\leq n$;
\item
the set $\hat{Q}_1(\beta)$ of arrows consists of two types:
\begin{itemize}
\item
the \emph{upward} arrows $a_{i,j}:q_{i,j}\rightarrow q_{i,j+1}$ for $0\leq i\leq \ell$ and $0\leq j\leq n-1$,
\item 
the \emph{rightward} arrows $b_{i,j}:q_{i,j}\rightarrow q_{i+1,j}$ for $0\leq i\leq \ell-1$, $0\leq j\leq n$, unless $(i,j)=(k-1,i_k)$.
\end{itemize}
\item
every \emph{rectangle} of arrows corresponds to a commutative diagram.
\end{itemize}

\begin{figure}[!htbp]
\vspace{-0.1in}
\begin{center}
\begin{tikzcd}[row sep=0.5pc,column sep=1pc,ampersand replacement=\&]

\begin{tikzpicture}[baseline=-.5ex,scale=0.7]
\begin{scope}

\draw[thin] (0.5,0)--(7.5,0);
\draw[thin] (0.5,2)--(1,2) to [out=0,in=180] (3,4) to[out=0,in=180] (5,6)--(7.5,6);
\draw[white,fill=white] (2,3) circle (0.2);
\draw[white,fill=white] (4,5) circle (0.2);
\draw[thin] (0.5,4)--(1,4) to[out=0,in=180] (3,2)--(5,2) to[out=0,in=180] (7,4)--(7.5,4);
\draw[white,fill=white] (6,3) circle (0.2);

\draw[thin] (0.5,6)--(3,6) to[out=0,in=180] (5,4) to[out=0,in=180] (7,2)--(7.5,2);
\draw[thin] (0.5,8)--(7.5,8);

\draw (0.7,0) node[left] {{\tiny$1$}};
\draw (0.7,2) node[left] {{\tiny$2$}};
\draw (0.7,4) node[left] {{\tiny$3$}};
\draw (0.7,6) node[left] {{\tiny$4$}};
\draw (0.7,8) node[left] {{\tiny$5$}};

\draw (7.3,0) node[right] {{\tiny$1$}};
\draw (7.3,2) node[right] {{\tiny$2$}};
\draw (7.3,4) node[right] {{\tiny$3$}};
\draw (7.3,6) node[right] {{\tiny$4$}};
\draw (7.3,8) node[right] {{\tiny$5$}};

\draw[black,fill=black] (1,-1) circle(0.1);
\draw (1,-1) node[below] {{\tiny$q_{0,0}$}};
\draw[black,fill=black] (1,1) circle(0.1);
\draw (1,1) node[below] {{\tiny$q_{0,1}$}};
\draw[black,fill=black] (1,3) circle(0.1);
\draw (1,3) node[below] {{\tiny$q_{0,2}$}};
\draw[black,fill=black] (1,5) circle(0.1);
\draw (1,5) node[below] {{\tiny$q_{0,3}$}};
\draw[black,fill=black] (1,7) circle(0.1);
\draw (1,7) node[below] {{\tiny$q_{0,4}$}};
\draw[black,fill=black] (1,9) circle(0.1);
\draw (1,9) node[below] {{\tiny$q_{0,5}$}};

\draw[red,thick,->] (1,-0.5)--(1,0.5);
\draw[red,thick,->] (1,1.5)--(1,2.5);
\draw[red,thick,->] (1,3.5)--(1,4.5);
\draw[red,thick,->] (1,5.5)--(1,6.5);
\draw[red,thick,->] (1,7.5)--(1,8.5);

\draw (0.8,-0.2) node[right] {{\tiny$a_{0,0}$}};
\draw (0.8,1.8) node[right] {{\tiny$a_{0,1}$}};
\draw (0.8,3.8) node[right] {{\tiny$a_{0,2}$}};
\draw (0.8,5.8) node[right] {{\tiny$a_{0,3}$}};
\draw (0.8,7.8) node[right] {{\tiny$a_{0,4}$}};

\draw[black,fill=black] (3,-1) circle(0.1);
\draw (3,-1) node[below] {{\tiny$q_{1,0}$}};
\draw[black,fill=black] (3,1) circle(0.1);
\draw (3,1) node[below] {{\tiny$q_{1,1}$}};
\draw[black,fill=black] (3,3) circle(0.1);
\draw (3,3) node[below] {{\tiny$q_{1,2}$}};
\draw[black,fill=black] (3,5) circle(0.1);
\draw (3,5) node[below] {{\tiny$q_{1,3}$}};
\draw[black,fill=black] (3,7) circle(0.1);
\draw (3,7) node[below] {{\tiny$q_{1,4}$}};
\draw[black,fill=black] (3,9) circle(0.1);
\draw (3,9) node[below] {{\tiny$q_{1,5}$}};

\draw[red,thick,->] (3,-0.5)--(3,0.5);
\draw[red,thick,->] (3,1.5)--(3,2.5);
\draw[red,thick,->] (3,3.5)--(3,4.5);
\draw[red,thick,->] (3,5.5)--(3,6.5);
\draw[red,thick,->] (3,7.5)--(3,8.5);

\draw (2.8,-0.2) node[right] {{\tiny$a_{1,0}$}};
\draw (2.8,1.8) node[right] {{\tiny$a_{1,1}$}};
\draw (2.8,3.8) node[right] {{\tiny$a_{1,2}$}};
\draw (2.8,5.8) node[right] {{\tiny$a_{1,3}$}};
\draw (2.8,7.8) node[right] {{\tiny$a_{1,4}$}};

\draw[black,fill=black] (5,-1) circle(0.1);
\draw (5,-1) node[below] {{\tiny$q_{2,0}$}};
\draw[black,fill=black] (5,1) circle(0.1);
\draw (5,1) node[below] {{\tiny$q_{2,1}$}};
\draw[black,fill=black] (5,3) circle(0.1);
\draw (5,3) node[below] {{\tiny$q_{2,2}$}};
\draw[black,fill=black] (5,5) circle(0.1);
\draw (5,5) node[below] {{\tiny$q_{2,3}$}};
\draw[black,fill=black] (5,7) circle(0.1);
\draw (5,7) node[below] {{\tiny$q_{2,4}$}};
\draw[black,fill=black] (5,9) circle(0.1);
\draw (5,9) node[below] {{\tiny$q_{2,5}$}};

\draw[red,thick,->] (5,-0.5)--(5,0.5);
\draw[red,thick,->] (5,1.5)--(5,2.5);
\draw[red,thick,->] (5,3.5)--(5,4.5);
\draw[red,thick,->] (5,5.5)--(5,6.5);
\draw[red,thick,->] (5,7.5)--(5,8.5);

\draw (4.8,-0.2) node[right] {{\tiny$a_{2,0}$}};
\draw (4.8,1.8) node[right] {{\tiny$a_{2,1}$}};
\draw (4.8,3.8) node[right] {{\tiny$a_{2,2}$}};
\draw (4.8,5.8) node[right] {{\tiny$a_{2,3}$}};
\draw (4.8,7.8) node[right] {{\tiny$a_{2,4}$}};

\draw[black,fill=black] (7,-1) circle(0.1);
\draw (7,-1) node[below] {{\tiny$q_{3,0}$}};
\draw[black,fill=black] (7,1) circle(0.1);
\draw (7,1) node[below] {{\tiny$q_{3,1}$}};
\draw[black,fill=black] (7,3) circle(0.1);
\draw (7,3) node[below] {{\tiny$q_{3,2}$}};
\draw[black,fill=black] (7,5) circle(0.1);
\draw (7,5) node[below] {{\tiny$q_{3,3}$}};
\draw[black,fill=black] (7,7) circle(0.1);
\draw (7,7) node[below] {{\tiny$q_{3,4}$}};
\draw[black,fill=black] (7,9) circle(0.1);
\draw (7,9) node[below] {{\tiny$q_{3,5}$}};

\draw[red,thick,->] (7,-0.5)--(7,0.5);
\draw[red,thick,->] (7,1.5)--(7,2.5);
\draw[red,thick,->] (7,3.5)--(7,4.5);
\draw[red,thick,->] (7,5.5)--(7,6.5);
\draw[red,thick,->] (7,7.5)--(7,8.5);

\draw (7.2,-0.2) node[left] {{\tiny$a_{3,0}$}};
\draw (7.2,1.8) node[left] {{\tiny$a_{3,1}$}};
\draw (7.2,3.8) node[left] {{\tiny$a_{3,2}$}};
\draw (7.2,5.8) node[left] {{\tiny$a_{3,3}$}};
\draw (7.2,7.8) node[left] {{\tiny$a_{3,4}$}};

\draw[red,thick,->] (1.5,-1)--(2.5,-1);
\draw[red,thick,->] (3.5,-1)--(4.5,-1);
\draw[red,thick,->] (5.5,-1)--(6.5,-1);

\draw (2,-0.9) node[below] {{\tiny$b_{0,0}$}};
\draw (4,-0.9) node[below] {{\tiny$b_{1,0}$}};
\draw (6,-0.9) node[below] {{\tiny$b_{2,0}$}};

\draw[red,thick,->] (1.5,1)--(2.5,1);
\draw[red,thick,->] (3.5,1)--(4.5,1);
\draw[red,thick,->] (5.5,1)--(6.5,1);

\draw (2,1.1) node[below] {{\tiny$b_{0,1}$}};
\draw (4,1.1) node[below] {{\tiny$b_{1,1}$}};
\draw (6,1.1) node[below] {{\tiny$b_{2,1}$}};

\draw[red,thick,->] (3.5,3)--(4.5,3);

\draw (4,3.1) node[below] {{\tiny$b_{1,2}$}};

\draw[red,thick,->] (1.5,5)--(2.5,5);
\draw[red,thick,->] (5.5,5)--(6.5,5);

\draw (2,5.1) node[below] {{\tiny$b_{0,3}$}};
\draw (6,5.1) node[below] {{\tiny$b_{2,3}$}};

\draw[red,thick,->] (1.5,7)--(2.5,7);
\draw[red,thick,->] (3.5,7)--(4.5,7);
\draw[red,thick,->] (5.5,7)--(6.5,7);

\draw (2,7.1) node[below] {{\tiny$b_{0,4}$}};
\draw (4,7.1) node[below] {{\tiny$b_{1,4}$}};
\draw (6,7.1) node[below] {{\tiny$b_{2,4}$}};

\draw[red,thick,->] (1.5,9)--(2.5,9);
\draw[red,thick,->] (3.5,9)--(4.5,9);
\draw[red,thick,->] (5.5,9)--(6.5,9);

\draw (2,9.1) node[below] {{\tiny$b_{0,5}$}};
\draw (4,9.1) node[below] {{\tiny$b_{1,5}$}};
\draw (6,9.1) node[below] {{\tiny$b_{2,5}$}};

\end{scope}
\end{tikzpicture}

\&

\begin{tikzpicture}[baseline=-.5ex,scale=0.8]
\begin{scope}

\draw[thin] (0.5,1)--(7.5,1);
\draw[thin] (0.5,2)--(1,2) to [out=0,in=180] (3,4) to[out=0,in=180] (5,6)--(7.5,6);
\draw[red,thick,->] (1.2,2.1) to[out=25,in=200] (2.8,4.1);
\draw[white,fill=white] (2,3.1) circle (0.2);
\draw (1.8,2.4) node[below] {{\tiny$\overline{b}_{0,2}$}};

\draw[red,thick,->] (3.2,4.1) to[out=25,in=200] (4.8,6.1);
\draw[white,fill=white] (4,5.1) circle (0.2);
\draw (3.8,4.4) node[below] {{\tiny$\overline{b}_{1,3}$}};

\draw[thin] (0.5,4)--(1,4) to[out=0,in=180] (3,2)--(5,2) to[out=0,in=180] (7,4)--(7.5,4);
\draw[red,thick,->] (1.2,4.1) to[out=-20,in=155] (2.8,2.1);
\draw (1.8,3.7) node[above] {{\tiny$\overline{b}_{0,3}$}};

\draw[red,thick,->] (5.2,2.1) to[out=25,in=200] (6.8,4.1);
\draw (5.8,2.4) node[below] {{\tiny$\overline{b}_{2,2}$}};
\draw[white,fill=white] (6,3.1) circle (0.2);

\draw[thin] (0.5,6)--(3,6) to[out=0,in=180] (5,4) to[out=0,in=180] (7,2)--(7.5,2);
\draw[red,thick,->] (3.2,6.1) to[out=-20,in=155] (4.8,4.1);
\draw (3.8,5.7) node[above] {{\tiny$\overline{b}_{1,4}$}};

\draw[red,thick,->] (5.2,4.1) to[out=-20,in=155] (6.8,2.1);
\draw (5.8,3.7) node[above] {{\tiny$\overline{b}_{2,3}$}};

\draw[thin] (0.5,7)--(7.5,7);

\draw (0.7,1) node[left] {{\tiny$1$}};
\draw (0.7,2) node[left] {{\tiny$2$}};
\draw (0.7,4) node[left] {{\tiny$3$}};
\draw (0.7,6) node[left] {{\tiny$4$}};
\draw (0.7,7) node[left] {{\tiny$5$}};

\draw (7.3,1) node[right] {{\tiny$1$}};
\draw (7.3,2) node[right] {{\tiny$2$}};
\draw (7.3,4) node[right] {{\tiny$3$}};
\draw (7.3,6) node[right] {{\tiny$4$}};
\draw (7.3,7) node[right] {{\tiny$5$}};

\draw[black,fill=black] (1,1) circle(0.1);
\draw (1,1) node[below] {{\tiny$\overline{q}_{0,1}$}};
\draw[black,fill=black] (1,2) circle(0.1);
\draw (1,2) node[below] {{\tiny$\overline{q}_{0,2}$}};
\draw[black,fill=black] (1,4) circle(0.1);
\draw (1,4) node[below] {{\tiny$\overline{q}_{0,3}$}};
\draw[black,fill=black] (1,6) circle(0.1);
\draw (1,6) node[below] {{\tiny$\overline{q}_{0,4}$}};
\draw[black,fill=black] (1,7) circle(0.1);
\draw (1,7) node[above] {{\tiny$\overline{q}_{0,5}$}};

\draw[black,fill=black] (3,1) circle(0.1);
\draw (3,1) node[below] {{\tiny$\overline{q}_{1,1}$}};
\draw[black,fill=black] (3,2) circle(0.1);
\draw (3,2) node[below] {{\tiny$\overline{q}_{1,2}$}};
\draw[black,fill=black] (3,4) circle(0.1);
\draw (3,4) node[below] {{\tiny$\overline{q}_{1,3}$}};
\draw[black,fill=black] (3,6) circle(0.1);
\draw (3,6) node[below] {{\tiny$\overline{q}_{1,4}$}};
\draw[black,fill=black] (3,7) circle(0.1);
\draw (3,7) node[above] {{\tiny$\overline{q}_{1,5}$}};

\draw[black,fill=black] (5,1) circle(0.1);
\draw (5,1) node[below] {{\tiny$\overline{q}_{2,1}$}};
\draw[black,fill=black] (5,2) circle(0.1);
\draw (5,2) node[below] {{\tiny$\overline{q}_{2,2}$}};
\draw[black,fill=black] (5,4) circle(0.1);
\draw (5,4) node[below] {{\tiny$\overline{q}_{2,3}$}};
\draw[black,fill=black] (5,6) circle(0.1);
\draw (5,6) node[below] {{\tiny$\overline{q}_{2,4}$}};
\draw[black,fill=black] (5,7) circle(0.1);
\draw (5,7) node[above] {{\tiny$\overline{q}_{2,5}$}};

\draw[black,fill=black] (7,1) circle(0.1);
\draw (7,1) node[below] {{\tiny$\overline{q}_{3,1}$}};
\draw[black,fill=black] (7,2) circle(0.1);
\draw (7,2) node[below] {{\tiny$\overline{q}_{3,2}$}};
\draw[black,fill=black] (7,4) circle(0.1);
\draw (7,4) node[below] {{\tiny$\overline{q}_{3,3}$}};
\draw[black,fill=black] (7,6) circle(0.1);
\draw (7,6) node[below] {{\tiny$\overline{q}_{3,4}$}};
\draw[black,fill=black] (7,7) circle(0.1);
\draw (7,7) node[above] {{\tiny$\overline{q}_{3,5}$}};

\draw[red,thick,->] (1.5,1.1)--(2.5,1.1);
\draw[red,thick,->] (3.5,1.1)--(4.5,1.1);
\draw[red,thick,->] (5.5,1.1)--(6.5,1.1);

\draw (2,1.1) node[below] {{\tiny$\overline{b}_{0,1}$}};
\draw (4,1.1) node[below] {{\tiny$\overline{b}_{1,1}$}};
\draw (6,1.1) node[below] {{\tiny$\overline{b}_{2,1}$}};

\draw[red,thick,->] (3.5,2.1)--(4.5,2.1);
\draw (4,2.1) node[below] {{\tiny$\overline{b}_{1,2}$}};

\draw[red,thick,->] (1.5,6.1)--(2.5,6.1);
\draw[red,thick,->] (5.5,6.1)--(6.5,6.1);

\draw (2,6.1) node[below] {{\tiny$\overline{b}_{0,4}$}};
\draw (6,6.1) node[below] {{\tiny$\overline{b}_{2,4}$}};

\draw[red,thick,->] (1.5,7.1)--(2.5,7.1);
\draw[red,thick,->] (3.5,7.1)--(4.5,7.1);
\draw[red,thick,->] (5.5,7.1)--(6.5,7.1);

\draw (2,7.1) node[above] {{\tiny$\overline{b}_{0,5}$}};
\draw (4,7.1) node[above] {{\tiny$\overline{b}_{1,5}$}};
\draw (6,7.1) node[above] {{\tiny$\overline{b}_{2,5}$}};

\end{scope}
\end{tikzpicture}
\end{tikzcd}
\end{center}
\vspace{-0.1in}
\caption{Quiver with relations associated to $\beta\in\FBr_n^+$. In the figure, $n=5$, $\ell=3$, and $i_1=2, i_2=3, i_3=2$, i.e. $\beta=[\sigma_2|\sigma_3|\sigma_2]$. 
Left: $\hat{Q}(\beta)$ is the \emph{quiver with relations}, consisting of \emph{black} vertices and \emph{red} arrows, such that all the rectangles are commutative.
Right: The quiver $\overline{Q}(\beta,\ast)$ consists of \emph{black vertices} $\overline{q}_{i,j}$ and \emph{red arrows} $\overline{b}_{i,j}$.}
\label{fig:framed_quivers_with_relations_for_a_positive_braid}
\end{figure}

Denote 
\[
\vec{r}^i:=\vec{r}_{i,\Right}=\vec{r}_{i+1,\Left}, \forall 0\leq i\leq \ell;\quad \vec{r}(q_{i,j}):=\text{\tiny$\sum_{k=1}^j$}\vec{r}^i(k).
\]
Similar to \cite[Prop.6.2, Prop.5.20]{STZ17}, a local study of micro-support conditions shows that 
\[
\modulistack_{\vec{r}}(\beta)\simeq [\hat{\cB}_{\vec{r}}(\beta)/\hat{\cG}_{\vec{r}}(\beta)], 
\]
where:
\begin{enumerate}[wide,labelwidth=!,labelindent=0pt]
\item
Denote
$\hat{\cA}_{\vec{r}}(\beta):=\prod_{a\in \hat{Q}_1(\beta)}\mathrm{Hom}_{\field}(\field^{\vec{r}(s(a))},\field^{\vec{r}(t(a))})$, where $s(a)$ (resp. $t(a)$) is the source (resp. target) of $a$.
Then, $\hat{\cB}_{\vec{r}}(\beta)$ is the locally closed subvariety of $\hat{\cA}_{\vec{r}}(\beta)$ consisting of quiver representations $\hat{F}$ (so $\hat{F}(q_{i,j})=\field^{\vec{r}(q_{i,j})}$) 
of $\hat{Q}(\beta)$ such that,
\begin{enumerate} 
\item
$\hat{F}(a_{i,j})$ are injections with free cokernels, and $\hat{F}(b_{i,j})$ are isomorphisms;
\item
Each rectangle  of $\hat{Q}(\beta)$ (see Figure \ref{fig:framed_quivers_with_relations_for_a_positive_braid}. Left) containing a crossing (say, the $k$-th crossing $q_k$) is sent to a pullback diagram under $\hat{F}$: 
\[
\begin{tikzcd}[row sep=2pc,column sep=6pc]
\field^{\vec{r}(q_{k-1,i_k-1})}\arrow[r,hookrightarrow,"{\scriptscriptstyle \hat{F}(a_{k-1,i_k-1})}"]\arrow[d,>->,"{\scriptscriptstyle\hat{F}(a_{k,i_k-1})\circ\hat{F}(b_{k-1,i_k-1})}"]\arrow[dr,phantom,"\lrcorner",very near start] & \field^{\vec{r}(q_{k-1,i_k})}\arrow[d,>->,"{\scriptscriptstyle \hat{F}(b_{k-1,i_k+1})\circ \hat{F}(a_{k-1,i_k})}"]\\
\field^{\vec{r}(q_{k,i_k})}\arrow[r,hookrightarrow,swap,"{\scriptscriptstyle \hat{F}(a_{k,i_k})}"] & \field^{\vec{r}(q_{k,i_k+1})}
\end{tikzcd}
\]
\end{enumerate}

\item
The isomorphisms of such data are encoded into a linear reductive algebraic group 
\[
\hat{\cG}_{\vec{r}}(\beta):=\text{\tiny$\prod_{v\in \hat{Q}_0(\beta)}$} GL(\vec{r}(v))
\] 
with the obvious action on $\hat{\cB}_{\vec{r}}(\beta)$. 
\end{enumerate}

\subsection{A parabolic reduction}
For simplicity, denote 
\[
P^i:=P_{\vec{r}^i}=P_{\vec{r}^i(1),\cdots,\vec{r}^i(n)}\subset G, \forall 0\leq i\leq \ell; ~~\rightsquigarrow~~ P_{\Left}=P^0, P_{\Right}=P^{\ell}.
\]
\emph{Define} $\hat{\cB}_{\vec{r},P}(\beta)$ to be the closed subvariety of $\hat{\cB}_{\vec{r}}(\beta)$ consisting of $\hat{F}^{\std}$ such that 
$\hat{F}^{\std}(a_{i,j})=(I_{\vec{r}(q_{i,j})},0)^T:\field^{\vec{r}(q_{i,j})}\hookrightarrow\field^{\vec{r}(q_{i,j+1})}$ is the \emph{standard inclusion}. Such an $\hat{F}^{\std}$ will be called a \emph{parabolic quiver representation}.

Observe that the isomorphisms in $\hat{\cG}_{\vec{r}}(\beta)$ preserving $\hat{\cB}_{\vec{r},P}(\beta)$ form a closed subgroup
\begin{equation}
\hat{\cP}_{\vec{r}}(\beta):=\text{\tiny$\prod_{i=\ell}^0$} P^i\hookrightarrow \hat{\cG}_{\vec{r}}(\beta)=\text{\tiny$\prod_{q_{i,j}\in\hat{Q}_0(\beta)}$} GL(\vec{r}(q_{i,j})).
\end{equation}
The inclusion is induced by the obvious compositions
$
\text{\tiny$\prod_{i=\ell}^0$} P^i\twoheadrightarrow P^i\twoheadrightarrow P_{\vec{r}^i(1),\cdots,\vec{r}^i(j)}\hookrightarrow GL(\vec{r}(q_{i,j})).
$
Moreover, 
$\hat{\cG}_{\vec{r}}(\beta)\times^{\hat{\cP}_{\vec{r}}(\beta)} \hat{\cB}_{\vec{r},P}(\beta):=(\hat{\cG}_{\vec{r}}(\beta)\times \hat{\cB}_{\vec{r},P}(\beta))/{\hat{\cP}_{\vec{r}}(\beta)} \isomorphic \hat{\cB}_{\vec{r}}(\beta)$.
Thus,
\[
[\hat{\cB}_{\vec{r}}(\beta)/\hat{\cG}_{\vec{r}}(\beta)]\simeq [\hat{\cB}_{\vec{r},P}(\beta)/\hat{\cP}_{\vec{r}}(\beta)].
\]

\subsection{Monodromy and microlocal monodromy}

By definition, the diagram (\ref{eqn:local_diagram_for_Betti_moduli_stacks}) is now equivalent to the following diagram of quotient stacks
\begin{equation}\label{eqn:local_diagram_for_Betti_moduli_stacks_1}
\begin{tikzcd}
{[\hat{\cB}_{\vec{r}}^{\infty}(\beta)/\hat{\cG}_{\vec{r}}^{\infty}(\beta)]} & {[\hat{\cB}_{\vec{r},P}(\beta)/\hat{\cP}_{\vec{r}}(\beta)]}\arrow[l,swap,"{\mon}"]\arrow[r,"{\mumon}"] & {[\overline{\cB}_{\vec{r}}(\beta)/\overline{\cG}_{\vec{r}}(\beta)]},
\end{tikzcd}
\end{equation}
where:
\begin{enumerate}[wide,labelwidth=!,labelindent=0pt]
\item
Denote
\begin{equation}
\hat{Q}^{\infty}(\beta):=(q_{0,n}\xrightarrow[]{b_{0,n}} q_{1,n}\xrightarrow[]{b_{1,n}}\cdots\xrightarrow[]{b_{\ell-1,n}} q_{\ell,n}).
\end{equation}
Then, the monodromy map $\mon$ is induced by restricting quiver representations of $\hat{Q}(\beta)$ to $\hat{Q}^{\infty}(\beta)$. In particular, 
\[
\hat{\cB}_{\vec{r}}^{\infty}(\beta)\isomorphic G^{\ell}.
\] 
The isomorphisms of such data are encoded by the linear reductive algebraic group 
\[
\hat{\cG}_{\vec{r}}^{\infty}(\beta):=\text{\tiny$\prod_{i=\ell}^0$} GL(\vec{r}(q_{i,n}))=G^{\ell+1},
\]
with the obvious action on $\hat{\cB}_{\vec{r}}^{\infty}(\beta)$. 

\item
\emph{Define} $\overline{Q}(\beta)$ to be the quiver associated to $\beta$ as in Figure \ref{fig:framed_quivers_with_relations_for_a_positive_braid} (right): 
\begin{itemize}[wide,labelwidth=!,labelindent=0pt]
\item
the set $\overline{Q}_0(\beta)$ of vertices consists of the $\overline{q}_{i,j}$'s for $0\leq i\leq \ell$ and $1\leq j\leq n$;

\item
the set $\overline{Q}_1(\beta)$ of arrows consists of $\overline{b}_{k-1,j}:\overline{q}_{k-1,j}\rightarrow \overline{q}_{k,s_{i_k}(j)}$, $1\leq k\leq \ell, 1\leq j\leq n$.
\end{itemize}
Then, the microlocal monodromy map $\mumon$ is induced by taking cokernels of upward maps of quiver representations in $\hat{\cB}_{\vec{r}}(\beta)$.
In particular, $\overline{\cB}_{\vec{r}}(\beta)$ is the algebraic variety parametrizing the quiver representations $\overline{F}$ of $\overline{Q}(\beta)$ 
such that 
$\overline{F}(\overline{q}_{i,j})=\field^{\vec{r}(\overline{q}_{i,j})}$ and $\overline{F}(\overline{b}_{i,j})\in GL(\vec{r}(\overline{q}_{i,j}))$. That is,
\[
\overline{\cB}_{\vec{r}}(\beta)=\text{\tiny$\prod_{i=\ell-1}^0\prod_{j=1}^n$}GL(\vec{r}(\overline{q}_{i,j}))=\text{\tiny$\prod_{i=\ell-1}^0$}GL(\vec{r}^i).
\]
The isomorphisms of such data are encoded into the linear reductive algebraic group
\[
\overline{\cG}_{\vec{r}}(\beta):=\text{\tiny$\prod_{0\leq i\leq \ell, 1\leq j\leq n}$} GL(\vec{r}(\overline{q}_{i,j}))=\text{\tiny$\prod_{i=\ell}^0$} GL(\vec{r}^i),
\]
with the obvious action on $\overline{\cB}_{\vec{r}}(\beta)$. 
\end{enumerate}

\subsection{A more concrete description of $[\hat{\cB}_{\vec{r},P}(\beta)/\hat{\cP}_{\vec{r}}(\beta)]$}

For any two maps $\vec{t},\vec{s}:[n]\rightarrow\mathbb{N}$ of total rank $N$, and any 
$X\in M_{N\times N}(\field)=M_{(\vec{t}(1)+\text{\tiny$\cdots$} + \vec{t}(n))\times(\vec{s}(1)+\text{\tiny$\cdots$} + \vec{s}(n))}(\field)$, 
let $X_{i,j}^{\vec{t},\vec{s}}$ denote the $(i,j)$-block, $\forall 1\leq i,j\leq n$. For any subsets $I=\{i_1<\text{\tiny$\cdots$}<i_a\}, J=\{j_1<\text{\tiny$\cdots$}<j_b\}$ of $[n]$, denote
$X_{I,J}^{\vec{t},\vec{s}}:=(X_{i_u,j_v}^{\vec{t},\vec{s}})_{1\leq u\leq a, 1\leq v\leq b}$.
Denote by $e_1,\cdots,e_N$ the \emph{standard basis} for $\field^N$. 
For any map $\vec{t}:[n]\rightarrow\mathbb{N}$ of total rank $N$, and any $1\leq i\leq n$, denote
\begin{equation*}
\bfe_i^{\vec{t}}:=(e_{\text{\tiny$\sum_{1\leq u\leq i-1}$}\vec{t}(u)+1},\cdots,e_{\text{\tiny$\sum_{1\leq u\leq i}$}\vec{t}(u)}).
\end{equation*}
Then, $(\bfe_1^{\vec{t}},\cdots,\bfe_n^{\vec{t}})$ is called the \emph{standard block basis} of $\field^N$ with respect to $\vec{t}$.

\begin{lemma}\label{lem:parameter_space_with_parabolic_reduction_for_a_positive_braid}
\begin{enumerate}[wide,labelwidth=!,labelindent=0pt]
\item
There exists a natural identification
\begin{eqnarray*}
\hat{\cB}_{\vec{r},P}(\beta)&\isomorphic&\{(\phi_{k-1,n})_{k=\ell}^1:\phi_{k-1,n}\in P_{\vec{r}^k(1),{\scriptstyle\cdots},\vec{r}^k(i_k)+\vec{r}^k(i_k+1),{\scriptstyle\cdots},\vec{r}^k(n)},
(\phi_{k-1,n})_{i_k+1,i_k}^{\vec{r}^k,\vec{r}^{k-1}}\in GL(\vec{r}^{k-1}(i_k),\field)\}\\
&\isomorphic& \text{\tiny$\prod_{k=\ell}^1$} (P^k\times M_{\vec{r}^{k-1}(i_k)\times\vec{r}^{k-1}(i_k+1)}(\field)),
\end{eqnarray*}
where the second identification is uniquely determined by the \emph{unique factorization}:
\begin{equation}\label{eqn:unique_factorization}
\phi_{k-1,n}=p_k\bfB_{i_k}^{\vec{r}_k}(\epsilon_k),\quad (p_k,\epsilon_k)\in P^k\times M_{\vec{r}^{k-1}(i_k)\times\vec{r}^{k-1}(i_k+1)}(\field).
\end{equation}

\item
Under the second identification, the action of $(h_k)_{k=\ell}^0\in\hat{\cG}_{\vec{r},P}=\prod_{k=\ell}^0P^k$ on 
$(p_k,\epsilon_k)_{k=\ell}^1\in \hat{\cB}_{\vec{r},P}(\beta)\isomorphic\prod_{k=\ell}^1 (P^k\times M_{\vec{r}^{k-1}(i_k)\times\vec{r}^{k-1}(i_k+1)}(\field))$ becomes the following: 
$(\hat{p}_k,\hat{\epsilon}_k)_{k=\ell}^1:=(h_k)_{k=\ell}^0\cdot (p_k,\epsilon_k)_{k=\ell}^1$ is uniquely determined by 
\[
[h_k]^{\vec{r}^k}\circ[p_k]^{\vec{r}^k}\circ[\bfB_{i_k}(\epsilon_k)]\circ[h_{k-1}^{-1}]^{\vec{r}^{k-1}} = [\hat{p}_k]^{\vec{r}^k}\circ[\bfB_{i_k}^{\vec{r}_k}(\hat{\epsilon}_k)]\in\underline{\FBD}_{n,N}.
\]

\item
The diagram (\ref{eqn:local_diagram_for_Betti_moduli_stacks_1}) is now induced by
\[
\begin{tikzcd}
\hat{\cB}_{\vec{r}}^{\infty}(\beta) & \hat{\cB}_{\vec{r},P}(\beta) \arrow[l,swap,"{\mon_0}"]\arrow[r,"{\mumon_0}"] 
& \overline{\cB}_{\vec{r}}(\beta),
\end{tikzcd}
\]
where $\mon_0((p_k,\epsilon_k)_{k=1}^{\ell})(b_{k-1,n}):=p_k\bfB_{i_k}^{\vec{r}_k}(\epsilon_k)$, and 
\[
\bfs_{i_k}^{\vec{r}_k}\mumon_0((p_k,\epsilon_k)_{k=1}^{\ell})\diag(\overline{b}_{k-1,1},\text{\tiny$\cdots$},\overline{b}_{k-1,n})=D(p_k)\bfs_{i_k}^{\vec{r}_k}.
\]
\end{enumerate}
\end{lemma}
From now on, we will always use the identification $\hat{\cB}_{\vec{r},P}(\beta)\isomorphic \text{\tiny$\prod_{k=\ell}^1$} (P^k\times M_{\vec{r}^{k-1}(i_k)\times\vec{r}^{k-1}(i_k+1)}(\field))$.

\begin{proof}
By Definition, any parabolic quiver representation $\hat{F}^{\std}\in \hat{\cB}_{\vec{r},P}(\beta)$ is determined by $\phi_{k,j}:=\hat{F}^{\std}(b_{k,j})\in GL(\vec{r}(q_{k,j}))$ such that:
\begin{itemize}[wide,labelwidth=!,labelindent=0pt]
\item
Each rectangle of $\hat{Q}(\beta)$ (see Figure \ref{fig:framed_quivers_with_relations_for_a_positive_braid}. Left) is sent to a commutative diagram under $\hat{F}^{\std}$.
Equivalently, $\phi_{k-1,n}\in P_{\vec{r}^{k-1}(1),\text{\tiny$\cdots$},\vec{r}^{k-1}(i_k)+\vec{r}^{k-1}(i_k+1),\text{\tiny$\cdots$},\vec{r}^{k-1}(n)}
=P_{\vec{r}^k(1),\text{\tiny$\cdots$},\vec{r}^k(i_k)+\vec{r}^k(i_k+1),\text{\tiny$\cdots$},\vec{r}^k(n)}$ for $1\leq k\leq \ell$, and 
\begin{equation}\label{eqn:parameter_space_with_parabolic_reduction_for_a_positive_braid_1}
\phi_{k-1,j}=(\phi_{k-1,n})_{\leq j,\leq j}^{\vec{r}^k,\vec{r}^{k-1}}, 1\leq k\leq \ell.
\end{equation}

\item
Each rectangle  of $\hat{Q}(\beta)$ (see Figure \ref{fig:framed_quivers_with_relations_for_a_positive_braid}. Left) containing a crossing (say, $q_k$) is sent to a pullback diagram under $\hat{F}^{\std}$: 
\[
\begin{tikzcd}[row sep=2pc,column sep=6pc]
\field^{\vec{r}^{k-1}(q_{k-1,i_k-1})}\arrow[r,hookrightarrow,"{\scriptscriptstyle \hat{F}^{\std}(a_{k-1,i_k-1})=(I,0)^T}"]\arrow[d,>->,"{\scriptscriptstyle\hat{F}^{\std}(a_{k,i_k-1})\circ\phi_{k-1,i_k-1}}"]\arrow[dr,phantom,"\lrcorner",very near start] & \field^{\vec{r}^{k-1}(q_{k-1,i_k})}\arrow[d,>->,"{\scriptscriptstyle \phi_{k-1,i_k+1}\circ \hat{F}^{\std}(a_{k-1,i_k})}"]\\
\field^{\vec{r}^k(q_{k,i_k})}\arrow[r,hookrightarrow,swap,"{\scriptscriptstyle \hat{F}^{\std}(a_{k,i_k})=(I,0)^T}"] & \field^{\vec{r}^k(q_{k,i_k+1})}
\end{tikzcd}
\]
Equivalently,  observe that
{\small\begin{eqnarray*}
\im(\phi_{k-1,i_k+1}\circ \hat{F}^{\std}(a_{k-1,i_k})\circ\hat{F}^{\std}(a_{k-1,i_k-1}))&=&\im(\hat{F}^{\std}(a_{k,i_k})\circ \hat{F}^{\std}(a_{k,i_k-1})\circ\phi_{k-1,i_k-1})\\
&=&\mathrm{Span}(\bfe_1^{\vec{r}^k},\cdots,\bfe_{i_k-1}^{\vec{r}^k})\subset \field^{\vec{r}^k(q_{k,i_k+1})},\\
\im(\hat{F}^{\std}(a_{k,i_k}))&=&\mathrm{Span}(\bfe_1^{\vec{r}^k},\cdots,\bfe_{i_k}^{\vec{r}^k})\subset \field^{\vec{r}^k(q_{k,i_k+1})},
\end{eqnarray*}}
then $(\bfe_1^{\vec{r}^k},\cdots,\bfe_{i_k}^{\vec{r}^k},\phi_{k-1,i_k+1}(\bfe_{i_k}^{\vec{r}^{k-1}}))$ is a basis for $\field^{\vec{r}^k(q_{k,i_k+1})}$. In other words, by (\ref{eqn:parameter_space_with_parabolic_reduction_for_a_positive_braid_1}), we have
\begin{equation}\label{eqn:parameter_space_with_parabolic_reduction_for_a_positive_braid_2}
(\phi_{k-1,n})_{i_k+1,i_k}^{\vec{r}^k,\vec{r}^{k-1}}\in GL(\vec{r}^{k-1}(i_k),\field),\quad 1\leq k\leq \ell.
\end{equation}
\end{itemize}
This shows the first identification. The second identification follows immediately from the unique factorization (\ref{eqn:unique_factorization}), which is just elementary linear algebra.
$(2)$ and $(3)$ follow from the definition and an easy application of the diagram calculus in $\underline{\FBD}_n$, the details are omitted.
This finishes the proof.
\end{proof}

\subsection{Finish the proof}

By the previous discussions, Proposition \ref{prop:local_diagram_for_Betti_moduli_stacks} now follows from the following

\begin{lemma}
There exists a natural equivalence of algebraic stacks
\[
\Psi_{\beta}:[\hat{\cB}_{\vec{r},P}(\beta)/\hat{\cP}_{\vec{r}}(\beta)] \xrightarrow[]{\simeq} [\hat{X}_{\vec{r}}(\beta)/(P_{\Right}\times P_{\Left})]
\]
compatible with left and right restrictions, monodromy and microlocal monodromy.
\end{lemma}

\begin{proof}
Recall that $\hat{\cB}_{\vec{r},P}(\beta)\isomorphic \text{\tiny$\prod_{k=\ell}^1$} (P^k\times M_{\vec{r}^{k-1}(i_k)\times\vec{r}^{k-1}(i_k+1)}(\field))$, 
and $\hat{\cP}_{\vec{r}}(\beta)=\text{\tiny$\prod_{k=\ell}^0$}P^k = P_{\Right}\times (\text{\tiny$\prod_{k=\ell-1}^1$}P^k)\times P_{\Left}$.
Besides, $\hat{X}_{\vec{r}}(\beta) = P_{\Right} \times \text{\tiny$\prod_{k=\ell}^1$} M_{\vec{r}^{k-1}(i_k)\times\vec{r}^{k-1}(i_k+1)}(\field)$.
Define a \emph{surjective} morphism of algebraic varieties
\[
\fc_0:\hat{\cB}_{\vec{r},P}(\beta) \rightarrow \hat{X}_{\vec{r}}(\beta): (p_k,\epsilon_k)\mapsto \epsilon'=(\epsilon_{\Right}',\vec{\epsilon}')=(\epsilon_{\Right}',\epsilon_{\ell}',\text{\tiny$\cdots$},\epsilon_1'),
\]
where $\epsilon'$ is uniquely determined by
\[
\text{\tiny$\prod_{k=\ell}^1$} ([p_k]^{\vec{r}^k}\circ[\bfB_{i_k}^{\vec{r}_k}(\epsilon_k)])
= [\epsilon_{\Right}']^{\vec{r}_{\Right}}\circ[\bfB_{\beta}^{\vec{r}}(\vec{\epsilon}')]'\in\underline{\FBD}_{n,N}.
\]
Observe that $\text{\tiny$\prod_{k=\ell-1}^1$}P^k$ acts freely on $\hat{\cB}_{\vec{r},P}(\beta)\isomorphic \text{\tiny$\prod_{k=\ell}^1$} (P^k\times M_{\vec{r}^{k-1}(i_k)\times\vec{r}^{k-1}(i_k+1)}(\field))$,
and $\fc_0$ is a geometric quotient for this action.
Thus, $[\hat{\cB}_{\vec{r},P}(\beta)/(\text{\tiny$\prod_{k=\ell-1}^1$}P^k)]\simeq \hat{X}_{\vec{r}}(\beta)$, which then induces an equivalence 
$\Psi_{\beta}:[\hat{\cB}_{\vec{r},P}(\beta)/\hat{\cP}_{\vec{r}}(\beta)] \simeq [\hat{X}_{\vec{r}}(\beta)/(P_{\Right}\times P_{\Left})]$.
The rest is straightforward. 
\end{proof}

\section{Irregular Riemann-Hilbert correspondence over curves}\label{sec:irregular_RH}

Here, we give a microlocal reformulation of the irregular Riemann-Hilbert correspondence over curves, following \cite[\S 3.3]{STWZ19}.
For related discussions in a different perspective, see also \cite{Boa21}.

Let $C$ be a compact Riemann surface with punctures (or marked points) $\{p_1,p_2,\ldots,p_k\}$. The main related case in this article is $C=\mathbb{C}P^1$ with puncture $\{p_1=\infty\}$. We're interested in integrable meromorphic connections $(H,\nabla)$ on $C$ with prescribed (possibly irregular) singularities along the punctures. 

In the regular case, by taking \emph{monodromy} of connections, the Riemann-Hilbert correspondence states that there's an equivalence of categories between the (rank $r$) integrable meromorphic connections on $C$ with regular singularities along the punctures $\{p_1,p_2,\ldots,p_k\}$, and (rank $r$) locally constant sheaves on $C\setminus\{p_1,p_2,\ldots,p_k\}$, in other words, (rank $r$) representation of the fundamental group of $C\setminus\{p_1,p_2,\ldots,p_k\}$. Passing to moduli spaces, we then obtain a \emph{complex analytic} isomorphism between a \emph{de Rham moduli space} $\modulispace_{\mathrm{DR}}$ of regular connections, and a \emph{Betti moduli space} $\modulispace_{\mathrm{B}}$ of locally constant sheaves (or fundamental group representations), i.e. character varieties.
The isomorphism is not algebraic, since taking the monodromy of a connection involves an exponential.

The regular Riemann-Hilbert correspondence admits a generalization to the irregular case, by specifying the \emph{formal types of irregular singularities} of the connections at the punctures on the one hand, and encoding the \emph{Stokes data} of the solution sheaves on the other. This is termed as \emph{irregular Riemann-Hilbert correspondence}, about which we review some details here.

Fix a puncture $p_i$, let $u$ be the local coordinate near $p_i$ with $u(p_i)=0$. Let $(H,\nabla)$ be a rank $r$ holomorphic bundle with meromorphic connection on $C$ as above, hence $\nabla$ is holomorphic on a punctured disk at $p_i$.

\begin{notation}
$\mathbb{C}\{u\}:=\text{the algebra of convergent power series in $u$}$, and $\mathbb{C}\{u\}[u^{-1}]:= \text{the fraction field}$;\\
$\mathbb{C}[[u]]:=\text{the algebra of formal power series in $u$}$, and $\mathbb{C}((u)):=\text{the fraction field}$.
\end{notation}

\subsection{Formal classification of irregular connections}

Let $(\mathscr{M},\nabla)$ be the germ of a rank $r$ \emph{meromorphic} bundle with connection, termed as \emph{meromorphic connection over $\mathbb{C}\{u\}[u^{-1}]$}. That is, $\mathscr{M}$ is a finite dimensional $\mathbb{C}\{u\}[u^{-1}]$-vector space, and
$\nabla:\mathscr{M}\rightarrow\mathbb{C}\{u\}[u^{-1}]du\otimes_{\mathbb{C}\{u\}[u^{-1}]}\mathscr{M}$ is a $\mathbb{C}$-linear map satisfying the Leibniz rule. By a \emph{holomorphic extension} of $(\mathscr{M},\nabla)$, we mean the germ of a rank $r$ holomorphic bundle with meromorphic connection $(\mathscr{H},\nabla)$ such that $(\mathscr{M},\nabla)=(\mathscr{H},\nabla)|_{\mathbb{C}\{u\}[u^{-1}]}$, i.e. $\mathscr{M}=\mathscr{H}\otimes_{\mathbb{C}\{u\}}\mathbb{C}\{u\}[u^{-1}]$.

In our case, let $(\mathscr{H},\nabla):=(H,\nabla)|_{\mathbb{C}\{u\}}$ be the germ of $(H,\nabla)$ at $p_i$. Then $(\mathscr{M},\nabla)=(\mathscr{H},\nabla)|_{\mathbb{C}\{u\}[u^{-1}]}$, and( $\mathscr{H},\nabla)$ defines a \emph{holomorphic extension} of $(\mathscr{M},\nabla)$.

Taking a trivialization $\mathscr{H}\isomorphic \mathbb{C}\{u\}^r$, the connection $\nabla$ can be written as:
\begin{eqnarray}\label{eqn:meromorphic_connection_local_model}
\nabla=d+A_n\frac{du}{u^n}+\ldots+A_2\frac{du}{u^2}+(A_1\frac{du}{u}+B(u)du).
\end{eqnarray}
where $A_i\in\mathrm{gl}_r(\mathbb{C})$, and $B(u)\in\mathrm{gl}_r(\mathbb{C}\{u\})$ is holomorphic. We say that $(\mathscr{M}, \nabla)$ has a \emph{regular singularity} at $0$ if and only if, for some holomorphic extension with a trivialization (equivalently, a trivialization $\mathscr{M}\isomorphic \mathbb{C}\{u\}[u^{-1}]^r$), we have $A_i=0$ for $i>1$.
Otherwise, we say that $(\mathscr{M},\nabla)$ has an \emph{irregular singularity} at $0$.

For any $f\in t^{-1}\mathbb{C}[t^{-1}]$, \emph{denote} $\mathscr{E}^f:=(\mathbb{C}\{t\},d-df)$, a rank $1$ holomorphic bundle with meromorphic connection.
The formal classification of $(\mathscr{M},\nabla)$ is given by the \emph{Levelt-Turrittin formal decomposition theorem}:

\begin{theorem}[Levelt-Turrittin. See e.g. {\cite[Thm.3.1]{VdPS12},\cite[Sec.4. Prop.]{Lev75}}]
There exists a finite base change $\pi_N:\mathrm{Spec}\mathbb{C}\{t\}\rightarrow\mathrm{Spec}\mathbb{C}\{u\}$ with $\pi_N(t):=t^N=u$, distinct polynomials $g_i\in t^{-1}\mathbb{C}[t^{-1}], 1\leq i\leq s$, and $\mathbb{C}\{t\}[t^{-1}]$-vector spaces $\mathscr{R}_i$ with meromorphic connections $\nabla_i$ having a regular singularity at $t=0$, such that
\begin{eqnarray}\label{eqn:Levelt-Turrittin}
\pi_N^*(\mathscr{M},\nabla)\otimes_{\mathbb{C}\{t\}[t^{-1}]}\mathbb{C}((t))\isomorphic [\oplus_{i=1}^s\mathscr{E}^{g_i}\otimes_{\mathbb{C}\{t\}[t^{-1}]} (\mathscr{R}_i,\nabla_i)]\otimes_{\mathbb{C}\{t\}[t^{-1}]}\mathbb{C}((t))
\end{eqnarray}
and 
\[\nabla_i=d-C_i\frac{dt}{t}\]
with respect to a trivialization $\mathscr{R}_i\isomorphic \mathbb{C}\{t\}[t^{-1}]^{r_i}$, such that: 
\begin{eqnarray}\label{eqn:Jordan_normal_form}
\text{$C_i$ is of Jordan normal form, and for any two eigenvalues $\lambda,\mu\in\mathbb{C}$, $\lambda-\mu\in\mathbb{Z}\Rightarrow \lambda=\mu$.}
\end{eqnarray}

\noindent{}Moreover, the formal decomposition is \emph{unique} in the following sense:\\
the positive integer $N$ and the polynomials $g_i$ with multiplicity $r_i$ for all $i$, are uniquely determined;\\
for each $i$, the Jordan normal form $C_i$ satisfying (\ref{eqn:Jordan_normal_form}) is uniquely determined, up to permuting the Jordan blocks and adding to each block an integer multiple.
\end{theorem} 

Now, we can make the following
\begin{definition}\cite[Sec.3.3]{STWZ19}
Let $(\mathscr{M},\nabla)$ be a meromorphic connection over $\mathbb{C}\{u\}[u^{-1}]$, with its Levelt-Turrittin formal decomposition as in the theorem.
The \emph{formal type of irregular singularity} of $(\mathscr{M},\nabla)$ is simply the set of polynomials $\{g_i\in u^{-\frac{1}{N}}\mathbb{C}[u^{-\frac{1}{N}}]\}$.
\end{definition}

\begin{corollary}\label{cor:action_on_formal_type}
Given $(\mathscr{M},\nabla)$ with the Levelt-Turrittin formal decomposition as in the theorem, there's a \emph{group action} of $C_N=\mathbb{Z}/N$ on the set $\{(g_i,r_i),[C_i]\}$.\\
Here, $C_N=\langle\sigma\rangle$ is the Galois group of $\pi_N:\mathrm{Spec}\mathbb{C}\{t\}\rightarrow\mathrm{Spec}\mathbb{C}\{u\}$, in which $\sigma$ acts by $\sigma\cdot t=\zeta t$ with $\zeta=e^{\frac{2\pi i}{N}}$ the $N$-th primitive root of unity. And, $[C_i]$ denotes the equivalence class of $C_i$ up to permuting the Jordan blocks and adding to each block an integer multiple.
\end{corollary}
\begin{proof}
$\pi_N\circ\sigma=\pi_N$ implies that 
\begin{eqnarray*}
\pi_N^*(\mathscr{M},\nabla)\otimes_{\mathbb{C}\{t\}[t^{-1}]}\mathbb{C}((t))\isomorphic [\oplus_{i=1}^s\mathscr{E}^{\sigma^*g_i}\otimes_{\mathbb{C}\{t\}[t^{-1}]} (\sigma^*\mathscr{R}_i,\sigma^*\nabla_i)]\otimes_{\mathbb{C}\{t\}[t^{-1}]}\mathbb{C}((t)),
\end{eqnarray*}
and with respect to the trivialization $\sigma^*\mathscr{R}_i\isomorphic\sigma^*(\mathbb{C}\{t\}[t^{-1}]^{r_i})\isomorphic\mathbb{C}\{t\}[t^{-1}]^{r_i}$, we have $\sigma^*\nabla_i=d-C_i\frac{d(\sigma^*t)}{\sigma^*t}=d-C_i\frac{dt}{t}$. It follows from the uniqueness of the Levelt-Turrittin formal decomposition that 
\[
\{(\sigma^*(g_i,r_i),[C_i]),1\leq i\leq s\}=\{(g_i,r_i,[C_i]),1\leq i\leq s\}.
\] 
 In other words, the group $C_N=\langle\sigma\rangle$ acts on $\{(g_i,r_i,[C_i])\}$.
\end{proof}

\subsection{Birkhoff extension}

Let $(\mathscr{M},\nabla)$ be the germ of a \emph{meromorphic} bundle with connection as above. 

\begin{lemma}\cite[II.Cor.3.7]{Sab02}
The functor $\mathscr{R}_0$ of taking germ at $0\in\mathbb{C}P^1$ is an equivalence of categories from
\[
(\text{algebraic vector bundles with connections on $\mathbb{A}^1\setminus\{0\}$ with regular singularity at $\infty$})
\]
to
\[
(\text{finite dimensional $\mathbb{C}\{u\}[u^{-1}]$-vector spaces with meromorphic connections}).
\]
\end{lemma}
Denote by $\mathscr{B}_0$ a fixed inverse functor of $\mathscr{R}_0$. It's called the \emph{algebraization} or \emph{Birkhoff extension} functor.

\subsection{The main asymptotic existence theorem}
Let $(\mathscr{M},\nabla)$ be a rank $r$ meromorphic connection over $\mathbb{C}\{u\}[u^{-1}]$. 
Take the Birkhoff extension $(M,\nabla)=\mathscr{B}_0(\mathscr{M},\nabla)$. Let $D^2\subset \mathbb{C}P^1$ be a disk centered at $0$. The sheaf of local $\nabla$-horizontal sections of $M^{\mathrm{an}}$ defines a locally constant sheaf $\sol$ on $D^2\setminus\{0\}$. Take the Levelt-Turrittin formal decomposition of $(\mathscr{M},\nabla)$ as in (\ref{eqn:Levelt-Turrittin}),
with formal type of irregular singularity $\tau=\{N\in\mathbb{N};(g_i,r_i)_{i=1}^s:g_i\in t^{-1}\mathbb{C}[t^{-1}] \text{ distinct, and } r_i\in\mathbb{N}\}$, where $t=u^{\frac{1}{N}}$. 

The formal decomposition implies that, with respect to the corresponding formal basis of $\mathscr{M}((t))=\mathscr{M}\otimes_{\mathbb{C}\{t\}[t^{-1}]}\mathbb{C}((t))$, the columns of
\begin{eqnarray*}
\mathrm{diag}(e^{g_i}t^{C_i}:1\leq i\leq s)
\end{eqnarray*}
give $r$ linearly independent \emph{formal} $\nabla$-horizontal sections near $u=0$, denoted by $(\hat{S_1},\ldots,\hat{S}_s)$. Here $\hat{S}_i=(\hat{S_i}^1,\ldots, \hat{S_i}^{r_i})$ consists of $r_i$ sections corresponding to the $r_i$ columns of the block $e^{g_i}t^{C_i}$. 
Specifically, let $(\vec{e}_1(t),\ldots,\vec{e}_s(t))$ be the formal basis, where $\vec{e}_i(t)=(e_i^1(t),\ldots,e_i^{r_i}(t))$ with $e_i^j(t)\in\mathscr{M}((t))=\mathscr{H}((t))$. Here, $(\mathscr{H},\nabla)$ is any \emph{holomorphic extension} of $(\mathscr{M},\nabla)$.
Then $\hat{S_i}=\vec{e}_i(t)e^{g_i}t^{C_i}$. Equivalently, $e^{-g_i}\hat{S}_it^{-\mathrm{D}_i}=\vec{e}_i(t)t^{\mathrm{N}_i}$, where $C_i=\mathrm{D}_i+\mathrm{N}_i$ with $\mathrm{D}_i=\mathrm{diag}(\lambda_i^1,\ldots,\lambda_i^{r_i})$ the semisimple part and $\mathrm{N}_i$ the nilpotent part. Hence, we can write
\begin{eqnarray}
\hat{S_i}^j=e^{g_i(t)}t^{\lambda_i^j}\sum_{n=n_{i,j}}^{\infty}\sum_{h=0}^{r_i-1}a_{i,h,n}^jt^n(\mathrm{log}t)^h,\quad a_{i,h,n}^j\in\mathscr{H}.
\end{eqnarray}
The \emph{main asymptotic existence theorem} states that, one can lift the formal sections to \emph{honest} local $\nabla$-horizontal sections over any small sector, with the same asymptotics. More precisely, we have

\begin{theorem}[The main asymptotic existence theorem. See, e.g. {\cite[Ch.IV-V. Thm.12.1,Thm.19.1]{Was18}}]\label{thm:asymptotic}
For any $\varphi_0\in S^1$, there exists a neighborhood $V_{\epsilon}=(\varphi_0-\epsilon,\varphi_0+\epsilon)$ of $\varphi_0$ and some $0<\delta\ll1$ such that, 
for any choices of branches of $t=u^{\frac{1}{N}}$ and $\mathrm{log}(t)=\frac{1}{N}\mathrm{log}(u)$ over $\mathrm{Sec}(V_{\epsilon})$, there exists a \emph{basis} of $r$ sections $\{S_i^j:1\leq i\leq s, 1\leq j\leq r_i\}$ of $\sol$ over $Sec(V_{\epsilon})\cap D_{\delta}^2$, such that
$S_i^j\sim \hat{S_i}^j$. That is,
\begin{eqnarray}\label{eqn:asymptotic_estimate}
||S_i^j(u)-\hat{S}_{i,L}^j(u)||=|e^{g_i(t)}t^{\lambda_i^j}|\cdot o(|t|^L)=e^{\mathrm{Re}(g_i(t))}\cdot o(|t^{\lambda_i^j+L}|), u\rightarrow 0 \text{ in $\mathrm{Sec}(V_{\epsilon})\cap D_{\delta}^2$},
\end{eqnarray}
where $\hat{S}_{i,L}^j$ is the finite partial sum
\[
\hat{S}_{i,L}^j(u):=e^{g_i}t^{\lambda_i^j}\sum_{n\leq L}\sum_{h=0}^{r_i-1}a_{i,h,n}^jt^n(\mathrm{log}t)^h.
\]
And, $||\cdot||$ is any Hermitian norm on the holomorphic bundle $\mathscr{H}$ induced from a trivialization $\mathscr{H}\isomorphic\mathbb{C}\{t\}^r$.
\end{theorem}
\noindent{}For simplicity, we \emph{write} $S_i:=(S_i^1,\ldots, S_i^{r_i})$, which consists of $r_i$ sections.

\begin{remark}\label{rem:branches_vs_orbit}
For each $i$ and any small open $V\subset S_u^1$, the set of all the branches of $g_i$ over $\mathrm{Sec}(V)$, denoted by $\{g_i^k\}$, is identical to the orbit $C_N\cdot g_i$ in Corollary \ref{cor:action_on_formal_type}, where in the later we fix a choice of the branches of $t=u^{\frac{1}{N}}$ over $\mathrm{Sec}(V)$.
\end{remark}

\subsection{Stokes data}
As above, let $(\mathscr{M},\nabla)$ be a meromorphic connection over $\mathbb{C}\{u\}[u^{-1}]$. 
Take the Birkhoff extension $(M,\nabla)=\mathscr{B}_0(\mathscr{M},\nabla)$, and solution sheaf $\sol$ over a disk $D^2\subset \mathbb{C}P^1$ centered at $0$. 
Let $\pi:\tilde{D^2}\rightarrow D^2$ be the real oriented blow-up at $0$ and $j:D^2\setminus\{0\}\hookrightarrow \tilde{D^2}$ be the natural inclusion, then define $\mathcal{S}:=j_*\sol|_{\pi^{-1}(0)=S^1}$. 

There's a natural \emph{local filtration} $\{\mathcal{S}_{\leq f}\}_{f\in{\Del}}$ by subsheaves on $\mathcal{S}$, termed as \emph{Stokes filtration}. Here, $\Del$ is Deligne's sheaf of ``growth rates''. $\{\mathcal{S}_{\leq f}\}_{f\in{\Del}}$ means that, for any open subset $U\subset S^1$ and $f\in\Del(U)$, $\mathcal{S}_{\leq f}\subset\mathcal{S}|_U$ defines a sub-sheaf, which induces a stalk-wise increasing filtration on $\mathcal{S}_{\varphi}, \varphi\in S^1$. More precisely,

\begin{definition}[Stokes filtration. See, e.g. {\cite[Sec.1.5]{Kas16}, \cite[Sec.2.1.3]{KKP08}}]~
\begin{enumerate}[wide,labelwidth=!, labelindent=0pt]
\item
Deligne's sheaf of ``growth rates'': $\Del$ is a locally constant sheaf of $\infty$-dimensional complex vector spaces on $S^1$ such that, for any open subset $U$ of $S^1$, let $\mathrm{Sec}(U)$ to be the sector in $D^2$ induced by $U$, i.e. $\mathrm{Sec}(U):=\{re^{i\varphi}\in D^2:r>0,\varphi\in U\}$, then 
\[
\Del(U):=\{f\in u^{-\frac{1}{N}}\mathbb{C}[u^{-\frac{1}{N}}]: \text{ for some $N\in\mathbb{N}$ and some branch $u^{\frac{1}{N}}$ defined on $\mathrm{Sec}(U)$}\}.
\]
In addition, $\Del$ is stalk-wise naturally (partially) \emph{ordered}: for any $f\neq g$ in $\Del(U)$ and $\varphi\in U$, we can write
\[f-g=c_au^a+(\text{higher order terms})\]
for some $a\in\mathbb{Q}_{<0}$ and $c_a\neq 0$. We \emph{say} $f<_{\varphi}g$ if and only if $\mathrm{Re}(c_ae^{i\varphi a})<0$.
\item
For any $f\in\Del(U)$, define $\mathcal{S}_{\leq f}\subset\mathcal{S}|_U$ stalk-wise as follows: for any $\varphi_0\in U\subset S^1$, set
\begin{eqnarray*}
(\mathcal{S}_{\leq f})_{\varphi_0}&:=&\{s\in\mathcal{S}_{\varphi_0}=\sol(\mathbb{R}_+e^{i\varphi_0}\cap D^2): ||(e^{-f}\cdot s)(re^{i\varphi})||=O(r^{-N}) (r\rightarrow 0),\\
&&\text{ uniformly in a neighborhood of $\varphi_0$, and for some $N\gg 0$}\}.
\end{eqnarray*}
where $||\cdot||$ is any Hermitian norm on the meromorphic bundle $\mathscr{M}$ induced from a trivialization $\mathscr{M}\isomorphic\mathbb{C}\{u\}[u^{-1}]^r$. In other words, $s\in(\mathcal{S}_{\leq f})_{\varphi_0}$ if and only $e^{-f}\cdot s$ has moderate (i.e. at most polynomial) growth in all directions near $\varphi_0$.
\end{enumerate}
By the ordering on $\Del_{\varphi}$, $\{\mathcal{S}_{\leq f}\}_{f\in\Del_{\varphi}}$ then defines an increasing filtration of $\mathcal{S}_{\varphi}$. The data of stalk-wise filtration $\{\mathcal{S}_{\leq f}\}_{f\in\Del}$ defined above is called the \emph{Stokes filtration} of $\mathcal{S}$, and we \emph{say} $\mathcal{S}$ is a \emph{$\Del$-filtered sheaf}.
\end{definition}

We can encode the data of the Stokes filtration by a simpler object, called \emph{Stokes Legendrian link}.
\begin{definition/proposition}[Stokes Legendrian link]\label{def:Stokes_Legendrian}
Take the Levelt-Turrittin formal decomposition of $(\mathscr{M},\nabla)$ as in (\ref{eqn:Levelt-Turrittin}),
with formal type of irregular singularity $\tau=\{g_i\in t^{-1}\mathbb{C}[t^{-1}], t=u^{\frac{1}{N}}\}$.
\begin{enumerate}[wide,labelwidth=!,labelindent=0pt]
\item
Fix $0<\epsilon\ll 1$, for each $i$, we define a real-valued (analytic) function $\mathrm{Re}^t_{\epsilon}(g_i):S_t^1\rightarrow\mathbb{R}$ via $\mathrm{Re}^t_{\epsilon}(g_i)(\theta):=\mathrm{Re}(g_i(\epsilon^{\frac{1}{N}} e^{i\theta}))$, which can be regarded as a multi-valued function $\mathrm{Re}^u_{\epsilon}(g_i)=\mathrm{Re}^t_{\epsilon}(g_i)\circ\pi_N^{-1}$ on $S_u^1$ via $\pi_N:S_t^1\rightarrow S_u^1: \pi_N(t)=t^N$, with local branches $\mathrm{Re}^u_{\epsilon}(g_i^k)$ (for $g_i^k$'s in Remark \ref{rem:branches_vs_orbit}). Notice that, as a multi-valued function on $S_u^1$, we always have $\mathrm{Re}^u_{\epsilon}(g_i)=\mathrm{Re}^u_{\epsilon}(\sigma^*g_i)$, for $\sigma$ in Corollary \ref{cor:action_on_formal_type}.

\noindent{}The union of the graphs of $\mathrm{Re}^u_{\epsilon}(g_i)$ in $S_u^1\times\mathbb{R}_z$ is called the \emph{Stokes diagram} of $(\mathscr{M},\nabla)$. 

\noindent{}The \emph{Stokes Legendrian link} of $(\mathscr{M},\nabla)$ is the Legendrian link $\Lambda$ in $J^1S_u^1\times\mathbb{R}_z\isomorphic T^{\infty,-}(S_u^1\times\mathbb{R}_z)$ whose front projection is the Stokes diagram. By definition, the Stokes Legendrian link $\Lambda$ (up to Legendrian isotopy) depends only on the formal type $\tau$.

\item
By the main asymptotic existence theorem, the Stokes filtration is determined by a local system of \emph{finite sets} $\Del_{(\mathscr{M},\nabla)}\subset\Del$. For any (small) open $U\subset S^1$, set
\[
\Del_{(\mathscr{M},\nabla)}(U):=\{g_i\in u^{-\frac{1}{N}}\mathbb{C}[u^{-\frac{1}{N}}]\}
\]
where we fix a branch of $u^{\frac{1}{N}}$ over $U$ as in Theorem \ref{thm:asymptotic}. Notice that by Remark \ref{rem:branches_vs_orbit}, $\Del_{(\mathscr{M},\nabla)}(U)$ contains all the branches $\{g_i^k\}$ of $g_i$ over $U$, for each $i$. 
By theorem \ref{thm:asymptotic}, each $g_i$ determines a \emph{locally constant sub-sheaf} $\mathcal{S}_{g_i}$ of $\mathcal{S}|_U$, spanned by the local sections $S_i^j$. Then for any $f\in\Del(U)$ and $\varphi\in U$, have
\[
(\mathcal{S}_{\leq f})_{\varphi}=\oplus_{g_i\leq_{\varphi}f}\mathcal{S}_{g_i}.
\]
\item
Via the projection $S^1\times\mathbb{R}_z\rightarrow S^1\isomorphic\pi^{-1}(0)$, we can regard $\mathcal{S}$ as a locally constant sheaf on $S^1\times\{z\gg0\}$. Fix $0<\epsilon\ll 1$, define a sheaf $\mathcal{S}^{\epsilon}$ on $S^1\times\mathbb{R}_z$, termed as \emph{Stokes sheaf} of $(\mathscr{M},\nabla)$,  via
\begin{eqnarray}
\mathcal{S}^{\epsilon}_{\varphi,z}:=\oplus_{\mathrm{Re}^u_{\epsilon}(g_i)(\varphi)<z}\mathcal{S}_{g_i}.
\end{eqnarray} 
Here, by fixing a branch of $u^{\frac{1}{N}}$ near $\varphi$ as above, we're fixing a branch of $\mathrm{Re}^u_{\epsilon}(g_i):S_u^1\rightarrow\mathbb{R}_z$.
Equivalently, the stalk of $\mathcal{S}^{\epsilon}$ at $(\varphi,z)$ consists of local sections of the solution sheaf $\sol$ over a small sector near $\varphi$, which grow at most polynomially faster than $e^{g_i^k}$, for \emph{some} $g_i^k$ with $\mathrm{Re}(g_i^k)(\epsilon e^{i\varphi})<z$.  

\noindent{}Then $\mathcal{S}^{\epsilon}|_{\{z\gg 0\}}=\mathcal{S}$, and the local $\Del_{(\mathscr{M},\nabla)}$-filtration $\{\mathcal{S}_{\leq g_i}\}_{g_i\in\Del_{(\mathscr{M},\nabla)}}$ is equivalent to the property that, the micro-support at infinity of $\mathcal{S}^{\epsilon}$ is equal to the Stokes Legendrian link $\Lambda$ defined above. 

\item
Moreover, passing to the associated local $\Del_{(\mathscr{M},\nabla)}$-graded local system of $(\mathcal{S},\{\mathcal{S}_{\leq g_i}\})$ over $S_u^1$
\[
\mathrm{Gr_f\mathcal{S}}_{\varphi}:=(\mathcal{S}_{\leq f})_{\varphi}/\oplus_{g_i <_{\varphi} f}(\mathcal{S}_{\leq g_i})_{\varphi},
\]
amounts to taking the \emph{microlocal monodromy} of the Stokes sheaf $\mathcal{S}^{\epsilon}$ on $S_u^1\times\mathbb{R}_z$, i.e. a local system over the Stokes Legendrian link $\Lambda$.
\end{enumerate}
\end{definition/proposition}

In fact we can be more concrete. Let $(\mathscr{M}^{\land}=\mathscr{M}\otimes_{\mathbb{C}\{u\}[u^{-1}]}\mathbb{C}((u)),\nabla)$ be the formal completion. Take the Levelt-Turrittin formal decomposition of $(\mathscr{M}^{\land},\nabla)$ as in (\ref{eqn:Levelt-Turrittin}), remembering the group action of $C_N$.
The connected components of the Stokes Legendrian link $\Lambda$ of $(\mathscr{M},\nabla)$ are in one-to-one correspondence with the orbits $\{C_N\cdot g_i\}$.  
The connected component corresponding to $C_N\cdot g_i$, denoted by $\Lambda_{[i]}$, has front projection the graph of the (multi-valued) function $\mathrm{Re}^u_{\epsilon}(g_i)$. In particular, $\Lambda_{[i]}$ is a $n_i$-to-$1$ cover of $S_u^1$ via the projection $\Lambda_{[i]}\hookrightarrow J^1S^1\rightarrow S^1$, with $n_i:=|C_N\cdot g_i|$. And, $N=\mathrm{lcm}(n_1,\ldots,n_s)$.

Alternatively, the \emph{microlocal monodromy} of $\mathcal{S}^{\epsilon}$ over $\Lambda_{[i]}$ is a \emph{rank $r_i$} local system, obtained as follows: the graph of $\mathrm{Re}^t_{\epsilon}(g_i)$ defines a Legendrian knot $\Lambda^t_i\hookrightarrow J^1S_t^1$. It's naturally identified with $S_t^1$, and equipped with an action of $\mathrm{Stab}_{C_N}(g_i)=<\sigma^{n_i}>\isomorphic\mathbb{Z}/(N/n_i)$. The quotient by the group is the natural map $\Lambda^t_i\rightarrow \Lambda_{[i]}$, in particular a $(N/n_i)$-to-$1$ cover.

Now, by taking local horizontal sections, the regular factor $(\mathscr{R}_i,\nabla_i)$ in the Levelt-Turrittin formal decomposition gives a well-defined local system of rank $r_i$ on $\Lambda_i^t$. The local system has monodromy $\exp(2\pi iC_i)$. The uniqueness of the decomposition implies that $\mathrm{Stab}_{C_N}(g_i)$ preserves $(\mathscr{R}_i,\nabla_i)$, hence the local system is $\mathrm{Stab}_{C_N}(g_i)$-equivariant, equivalently, descends to a local system of rank $r_i$ on $\Lambda_{[i]}$ as desired. 
The monodromy $T_i$ of the latter satisfies $T_i^{N/n_i}=\exp(2\pi iC_i)$. We see that $T_i$ is not uniquely determined by $[C_i]$, but by $[C_i]$ together with the group action of $\mathrm{Stab}_{C_N}(g_i)$ on the local system over $\Lambda_i^t$ determined by $[C_i]$.

As a consequence, fix a formal type $\tau=\{g_i\in u^{-\frac{1}{N}}\mathbb{C}[u^{-\frac{1}{N}}]\}$, with the associated Stokes Legendrian link $\Lambda$, we obtain a commutative diagram of functors
\[
\begin{tikzcd}
&
\mathcal{L}oc(S_{\infty}^1)\\
\left(
\text{
\minipage[c]{2.5in}
meromorphic connections $(\mathscr{M},\nabla)$ over $\mathbb{C}\{u\}[u^{-1}]$ with prescribed formal type $\tau$
\endminipage
}
\right)
\arrow{r}{\St}\arrow{d}{F}\arrow[ur,"{\sol}"]
&
\Sh_{\Lambda}(S_u^1\times\mathbb{R}_z,\{z=-\infty\})\arrow{d}{\mumon}\arrow[u,"{\mon}"]\\
\left(
\text{
\minipage[c]{2.5in}
formal connections $(\mathscr{M}^{\land}, \nabla)$ over $\mathbb{C}((u))$ with prescribed formal type $\tau$
\endminipage
}
\right)
\arrow{r}{\overline{\St}}
&\mathcal{L}oc(\Lambda)
\end{tikzcd}
\]
where: $S_{\infty}^1=S_u^1\times\{z_0\}$ for some $z\gg 1$; $\sol$ denotes the solution sheaf (near $\infty$) induced by the Birkhoff extension; 
$\St$ sends $(\mathscr{M},\nabla)$ to its Stokes sheaf $\mathcal{S}^{\epsilon}$; $F$ is the formal completion; $\Sh_{\Lambda}(S_u^1\times\mathbb{R}_z,\{z=-\infty\})$ is the full sub-category of $\Sh_{\Lambda}(S_u^1\times\mathbb{R}_z)$ whose objects have zero stalks at $z\ll 0$, and $\mumon$ is the microlocal monodromy. 
The definition of $\overline{\St}$ needs some attention: 
\begin{itemize}[wide,labelwidth=!,labelindent=0pt]
\item
Given a formal connection $(\mathscr{M}^{\land},\nabla)$, we lift it to a meromorphic connection $(\mathscr{M},\nabla)$ \cite[Lem.4.3]{Mal82}, and apply to $(\mathscr{M},\nabla)$ the ``associated graded local system'' construction in Definition \ref{def:Stokes_Legendrian}.(4), then $\overline{\St}(\mathscr{M}^{\land},\nabla):=\mathrm{Gr}\circ\St(\mathscr{M},\nabla)$, which is just the commutativity of the diagram. The definition is independent of the meromorphic lift \cite[\S.4.B)]{Mal82}.
\item
Alternatively, we can define $\overline{\St}$ directly via the Levelt-Turrittin formal decomposition of $(\mathscr{M}^{\land},\nabla)$ equipped with the group action of $C_N$, described as in the previous paragraph.
\end{itemize}

\noindent{}\textbf{Note}: here by a formal connection $(\mathscr{M}^{\land},\nabla)$ with prescribed formal type $\tau$, we mean its formal type of irregular singularity $\tau(\mathscr{M}^{\land},\nabla)$ is a subset of $\tau$.

The local form of irregular Riemann-Hilbert correspondence now states:
\begin{theorem}[Local form of irregular Riemann-Hilbert correspondence. {\cite[Thm.4.2, Thm.4.4]{Mal82}}]\label{thm:local_irregular_RH}
The functors $\St$ and $\overline{\St}$ are equivalences of categories.
\end{theorem}

\subsection{Irregular Riemann-Hilbert correspondence}

Finally, we come back to our global setting and state the irregular Riemann-Hilbert correspondence for curves. Let $C$ be a Riemann surface with punctures $p_1,\ldots,p_k$. Fix a formal type of irregular singularity $\tau_i$ at each $p_i$. Let $\mathcal{C}_{\mathrm{dR}}(C,\{p_i\},\{\tau_i\})$ be the category of meromorphic connections on $C$ with prescribed formal irregular singularity $\tau_i$ at $p_i$. Notice that the case $\tau_i=\{0\}$ corresponds to the regular singularity at $p_i$. 

Let $\Lambda_i$ be the Stokes Legendrian link associated to $\tau_i$. Take the real oriented blow-up $\pi_i:\mathrm{Bl}_{p_i}(C)\rightarrow C$ at each $p_i$, and glue $S^1\times\mathbb{R}_z$ into $\mathrm{Bl}_{p_i}(C)$ along the identification $S^1\times\{\infty\}\isomorphic\pi^{-1}(p_i)$. The resulting surface can again be identified with the punctured surface $C\setminus\{p_i\}$. Now, $\Lambda_i\hookrightarrow J^1S^1=T^*S^1\times\mathbb{R}_z$ becomes a Legendrian link in $T^{\infty}C$ whose front projection lives near $p_i$ via the identification $S^1\times\mathbb{R}_z\isomorphic D_{\delta}^*(p_i)$, where $D_{\delta}^*(p_i)$ is a small punctured disk in $C$ centered at $p_i$. Then recall that $\mathcal{C}_{\mathrm{B}}(C,\{p_i\},\{\Lambda_i\})$ is the full sub-category of $\Sh_{\cup\Lambda_i}(C)$ of constructible sheaves whose stalk at each $p_i$ vanishes.

Let $(M,\nabla)$ be any meromorphic connection in $\mathcal{C}_{\mathrm{dR}}(C,\{p_i\},\{\tau_i\})$, take the solution sheaf $\sol$ of local $\nabla$-horizontal sections and apply the Stokes sheaf construction to $\sol$ at each $p_i$, the previous subsection shows that we then obtain a \emph{global Stokes sheaf} $\mathcal{S}^{\epsilon}\in\mathcal{C}_{\mathrm{B}}(C,\{p_i\},\{\Lambda_i\})$. As in the local case, this gives a well-defined \emph{global Stokes sheaf} functor $\St:\mathcal{C}_{\mathrm{dR}}(C,\{p_i\},\{\tau_i\})\rightarrow\mathcal{C}_{\mathrm{B}}(C,\{p_i\},\{\Lambda_i\})$, which fits into the following commutative diagram of functors:
\begin{eqnarray}
\begin{tikzcd}
\mathcal{C}_{\mathrm{dR}}(C,\{p_i\},\{\tau_i\})\arrow{r}{\St}\arrow{d}{\prod_i\mathscr{R}_i} & \mathcal{C}_{\mathrm{B}}(C,\{p_i\},\{\Lambda_i\})\arrow{d}{\prod_i\fr_i}\\
\prod_i\left(
\text{
\minipage[c]{2.5in}
meromorphic connections $(\mathscr{M}_i, \nabla)$ over $\mathbb{C}\{u\}[u^{-1}]$ with prescribed formal type $\tau_i$
\endminipage
}
\right)
\arrow{r}{\prod_i\St_i}
&\prod_i\Sh_{\Lambda_i}(D^*_{\delta}(p_i))_0
\end{tikzcd}
\end{eqnarray}
where: 
$\Sh_{\Lambda_i}(D^*_{\delta}(p_i))_0=\Sh_{\Lambda_i}(S^1\times\mathbb{R}_z)_0$ under the identification $D^*_{\delta}(p_i)\isomorphic S^1\times\mathbb{R}_z$ mentioned above.

Now the global form of irregular Riemann-Hilbert correspondence over curves states:
\begin{theorem}[Irregular Riemann-Hilbert correspondence over curves]
The global Stokes sheaf functor $\St$ is an equivalence of categories.
\end{theorem}

\subsection{Example}

Take $C=\mathbb{C}P^1$ with one puncture $p=\infty$ and let $z$ be the standard complex coordinate on $\mathbb{C}P^1$. Fix $M=\mathcal{O}_{\mathbb{C}P^1}^{\oplus 2}$ to be the trivial rank 2 holomorphic bundle on $\mathbb{C}P^1$. Let $\nabla:M\rightarrow M\otimes\Omega_{\mathbb{C}P^1}^1$ be the meromorphic connection whose matrix form is
\begin{eqnarray*}
\nabla=d-Adz,\quad A=\left(\begin{array}{cc}
0 & 1\\
z^n & 0
\end{array}\right)
\end{eqnarray*}
This corresponds to the linear ODE $f''-z^nf=0$ (\emph{generalized Airy equation}), via $
\left(\begin{array}{c}
f\\
f'
\end{array}\right)'=A\left(\begin{array}{c}
f\\
f'
\end{array}\right)$. Here, $'=\frac{d}{dz}$. 
We want to determine the associated Stokes Legendrian link at the (irregular) singularity $\infty$. We firstly take the coordinate $x=z^{-1}$ centered at $\infty$ so that $x(\infty)=0$, and define the differential operator $\delta=x\frac{d}{dx}=-z\frac{d}{dz}$, then the linear ODE becomes $Lf=0$ with $L=\delta^2+\delta-x^{-(n+2)}$. We can use Newton polygons \cite[Chap.3.3]{VdPS12} to take a formal decomposition of $L$ over some finite field extension of $\mathbb{C}((x))$.

\begin{figure}
\begin{center}
\begin{tikzpicture}[
scale=2,
line/.style={thin},
important line/.style={very thick}, 
dashed line/.style={dashed, thin},
dot/.style={circle,fill=black,minimum size=0pt,inner sep=0pt,
            outer sep=0pt},
important dot/.style= {circle,fill=black,minimum size=4pt,inner sep=0pt,
            outer sep=-1pt},           
]
                    
\draw[line] 
(-.5,0) coordinate (es) -- (1.5,0) coordinate (ee);
\node[important dot,label=above right:$2$] at (1,0)(){}; 
\node[important dot,label=above right:$1$] at (.5,0)(){};
              
\draw[line]
(0,.5) coordinate (es)--(0,-2) coordinate (ee);
\node[dot,label=left:{-1}] at (0,-.5)(){};
\node[dot,label=left:{-2}] at (0,-1)(){};
\node[important dot,label=below left:{-3}] at (0,-1.5)(){};

\draw[important line]
(1,0) coordinate (es) -- (1,.5) coordinate (ee);

\draw[important line]
(1,0) coordinate (es) -- (0,-1.5) coordinate (es);

\draw[important line]
(0,-1.5) coordinate (es) -- (-.5,-1.5) coordinate (ee);

\draw[dashed line]
(-0.1,-1.5) coordinate (es) -- (1,.15) coordinate (ee);

\draw[dashed line]
(-0.2,-1.5) coordinate (es) -- (1,.3) coordinate (ee);

\draw[dashed line]
(-0.3,-1.5) coordinate (es) -- (1,.45) coordinate (ee);              
          
\end{tikzpicture}
\end{center}
\vspace{-0.2in}
\caption{$n=1$: Newton polygon for $L=\delta^2+\delta-x^{-3}$.}
\label{fig:Newton polygon}
\end{figure}

The Newton polygon $\mathrm{N}(L)$ of $L$ is the convex hull of the second quadrants with vertices $(i,j)$ corresponding to the monomials $x^j\delta^i$ in $L$. In the case $n=1$, it's given by the shaded region in Figure \ref{fig:Newton polygon}. $\mathrm{N}(L)$ has 2 extremal points $\{(n_1,m_1)=(0,-(n+2)),(n_2,m_2)=(2,0)\}$ with $0\leq n_1<n_2$. $L$ has a unique \emph{positive slope} $k_1=\frac{m_2-m_1}{n_2-n_1}=\frac{n+2}{2}$. Write $\frac{b}{a}:=\frac{n+2}{2}$ with $a>0, \mathrm{gcd}(a,b)=1$. Say, $n$ is odd, then $(a,b)=(2,n+2)$. We make the field extension $\mathbb{C}((t))\supset\mathbb{C}((x))$ with $t^a=t^2=x$, and define $\Delta:=t^b\delta=t^{(n+2)}\delta$. We can rewrite $L=t^{-2b}(\Delta^2+(1-b/2)t^b\Delta-1)$, so $\tilde{L}:=t^{2b}L=\Delta^2+(1-b/2)t^b\Delta-1=(\Delta-1)(\Delta+1) (\mathrm{mod}~ t)$ (in fact true $\mathrm{mod~} t^b$). A version of Hensel's lemma (\cite[Prop.3.50]{VdPS12}) or a direct calculation shows that 
\[
\tilde{L}=(\Delta-1+t^b(\frac{2-b}{4}+\ldots))(\Delta+1+t^b(\frac{2-b}{4}+\ldots))
\]
where `$\ldots$' means some formal power series in $t\mathbb{C}[[t]]$. Hence, another simple calculation gives 
\[
L=(\delta-t^{-(n+2)}+\frac{2+b}{4}+\ldots)(\delta+t^{-(n+2)}+\frac{2-b}{4}+\ldots).
\]
By solving $(\delta+t^{-(n+2)}+\frac{2-(n+2)}{4}+\ldots)f_-=0$, we then obtain a formal solution to $Lf_-=0$, of the form 
\[
f_-=\exp(\frac{2}{n+2}x^{-\frac{n+2}{2}})x^{\frac{n}{4}}\sum_{m\geq 0}a_m^-x^{\frac{m}{2}}.
\]
Similarly, using $\Delta^2-1=(\Delta+1)(\Delta-1)$, one obtains another factorization of $L$, through which one immediately find another formal solution to $Lf=0$, of the form
\[
f_+=\exp(-\frac{2}{n+2}x^{-\frac{n+2}{2}})x^{\frac{n}{4}}\sum_{m\geq 0}a_m^+x^{\frac{m}{2}}.
\]
The exponents $\{\pm\frac{2}{n+2}x^{-\frac{n+2}{2}}\}$ in $f_-,f_+$ are the $g_i's$ in the Levelt-Turrittin formal decomposition of $(M,\nabla)$ near $\infty$, i.e. the formal type of the irregular singularity at $\infty$. Now, it's direct to see that the Stokes Legendrian link of $(M,\nabla)$ near $\infty$ has front diagram the $(2,n+2)$-braid in $\mathbb{C}P^1\setminus\{\infty\}$ encircling $\infty$.

\end{document}